\documentclass[12pt,british,a4paper,reqno]{amsart}
    \usepackage[T1]{fontenc}
\usepackage[utf8]{inputenc}
	\usepackage[british]{babel}
    \DeclareMathSizes{12}{12}{7}{6}
\usepackage{wasysym}
	\usepackage{amsmath}
	\usepackage{amssymb}
	\usepackage{amsthm}
	
	\usepackage{adjustbox}
	\usepackage{bbm}
	\usepackage{braket}
	\usepackage{xcolor}
	\usepackage[shortlabels]{enumitem}
	\usepackage{etoolbox}
	\usepackage{fancyhdr}
	\usepackage{geometry}
	\usepackage{graphicx}
	\usepackage{ifsym}
	\usepackage{indentfirst}
	\usepackage{mathrsfs}
	\usepackage{mathtools}
	\usepackage{oplotsymbl} %\pentago
	\usepackage{proof}
	\usepackage{qtree}
	\usepackage{setspace}
	\usepackage{soul}
	\usepackage{stmaryrd}
	\usepackage{tensor}
	\usepackage{tikz}
	\usepackage{tikz-cd}
	\usepackage[normalem]{ulem}
	\usepackage{cancel}

    \usepackage[hidelinks, hyperfootnotes=false]{hyperref}
		\usepackage[capitalise]{cleveref}
		\usepackage{bookmark}
\makeatletter
\if@cref@capitalise
\crefname{theorem}{TheoUpperCase}{TheoUppercasesS}
\else
\crefname{theorem}{theoLowercase}{theoLowercaseS}
\fi
\makeatother
\DeclareFontFamily{U}{rcjhbltx}{}
\DeclareFontShape{U}{rcjhbltx}{m}{n}{<->rcjhbltx}{}
\DeclareSymbolFont{hebrewletters}{U}{rcjhbltx}{m}{n}
\let\aleph\relax\let\beth\relax
\let\gimel\relax\let\daleth\relax

\DeclareMathSymbol{\aleph}{\mathord}{hebrewletters}{39}
\DeclareMathSymbol{\beth}{\mathord}{hebrewletters}{98}
\DeclareMathSymbol{\gimel}{\mathord}{hebrewletters}{103}
\DeclareMathSymbol{\daleth}{\mathord}{hebrewletters}{100}

\DeclareMathSymbol{\lamed}{\mathord}{hebrewletters}{108}
\DeclareMathSymbol{\mem}{\mathord}{hebrewletters}{109}
\DeclareMathSymbol{\ayin}{\mathord}{hebrewletters}{96}
\DeclareMathSymbol{\tsadi}{\mathord}{hebrewletters}{118}
\DeclareMathSymbol{\qof}{\mathord}{hebrewletters}{113}
\DeclareMathSymbol{\shin}{\mathord}{hebrewletters}{152}

\newcommand{\updave}{\textup{\davidsstar}}
	\bookmarksetup{numbered,open}
	\setlist[enumerate]{itemsep=3pt,topsep=3pt,
    labelsep=5pt, 
    leftmargin=29.40003pt,
	}
	\renewcommand{\qed}{\hfill \mbox{\raggedright \rule{0.1in}{0.1in}}}
	
	\renewcommand{\geq}{\geqslant}
	\renewcommand{\leq}{\leqslant}
	\renewcommand{\phi}{\varphi}

	\providecommand{\corollaryname}{Corollary}
	\providecommand{\definitionname}{Definition}
	\providecommand{\notationname}{Notation}
	\providecommand{\examplename}{Example}
	\providecommand{\lemmaname}{Lemma}
	\providecommand{\propositionname}{Proposition}
	\providecommand{\remarkname}{Remark}
	\providecommand{\theoremname}{Theorem}
	\providecommand{\setupname}{Setup}
	\providecommand{\conjecturename}{Conjecture}
	\providecommand{\questionname}{Question}

	\theoremstyle{plain}
		\newtheorem{thm}{\protect\theoremname}[section] % reset thm numbering for each section
		\newtheorem{prop}[thm]{\protect\propositionname}
		    \crefalias{prop}{proposition}
		\newtheorem{lem}[thm]{\protect\lemmaname}
		\newtheorem{cor}[thm]{\protect\corollaryname}
		
		\newtheorem{thmx}{\protect\theoremname}
		\newtheorem{corx}[thmx]{\protect\corollaryname}

	\theoremstyle{definition}
		\newtheorem{defn}[thm]{\protect\definitionname}
		\newtheorem{notn}[thm]{\protect\notationname}
		    \crefname{notn}{notation}{notations}
		    \Crefname{notn}{Notation}{Notations}
		\newtheorem{example}[thm]{\protect\examplename}
		\newtheorem{setup}[thm]{\setupname}
		    \crefname{setup}{setup}{setups}
		    \Crefname{setup}{Setup}{Setups}

	\theoremstyle{remark}
		\newtheorem{rem}[thm]{\protect\remarkname}
		
	\numberwithin{figure}{section}
	\numberwithin{equation}{section}

	\usetikzlibrary{matrix,arrows,decorations.pathmorphing,positioning,decorations.pathreplacing}
	\tikzset{commutative diagrams/.cd, 
		mysymbol/.style = {start anchor=center, end anchor = center, draw = none}}

	\let\amph=& %needed for matrix environments in arrow labels
	
	\tikzcdset{every label/.append style = {font = \footnotesize}}

	\newcommand{\BE}{\mathbb{E}}
			\renewcommand{\BE}{\mathchoice
  			{\mathbb{E}}% \displaystyle
  			{\mathbb{E}}% \textstyle
  			{\raisebox{-0.2pt}{\scalebox{.6}{\(\mathbb{E}\)}}}% \scriptstyle
  			{\scalebox{.5}{\(\mathbb{E}\)}}% \scriptscriptstyle
		}
	\newcommand{\BF}{\mathbb{F}}

	\newcommand{\BZ}{\mathbb{Z}}

	\newcommand{\CA}{\mathcal{A}}
	
	\newcommand{\CC}{\mathcal{C}}
	\newcommand{\CD}{\mathcal{D}}

	\newcommand{\CI}{\mathcal{I}}

	\newcommand{\CN}{\mathcal{N}}

	\newcommand{\CS}{\mathcal{S}}
	\newcommand{\CT}{\mathcal{T}}

	\newcommand{\CX}{\mathcal{X}}

    \newcommand{\SA}{\mathscr{A}}
	
	\newcommand{\SE}{\mathscr{E}}
	\newcommand{\SF}{\mathscr{F}}
	\newcommand{\SG}{\mathscr{G}}
	\newcommand{\SH}{\mathscr{H}}

    \newcommand{\SL}{\mathscr{L}}
    \newcommand{\SM}{\mathscr{M}}
    \newcommand{\SN}{\mathscr{N}}
	
	\newcommand{\SQ}{\mathscr{Q}}

	\newcommand{\SY}{\mathscr{Y}}

		\newcommand{\Ab}{\operatorname{\mathsf{Ab}}\nolimits}

		\newcommand{\Exactcat}{\operatorname{\mathsf{Exact}}\nolimits}
		\newcommand{\exactcat}{\operatorname{\mathsf{exact}}\nolimits}

		\newcommand{\op}{\raisebox{0.5pt}{\scalebox{0.6}{\textup{op}}}}
		
		\newcommand{\sse}{\subseteq}

		\newcommand{\iso}{\cong}
		\newcommand{\niso}{\ncong}
		
		\newcommand{\Hom}{\operatorname{Hom}\nolimits}

		\newcommand{\into}{\hookrightarrow}
		
		\makeatletter
            \newcommand{\xmapsfrom}[2][]{%
            \ext@arrow3095\leftarrowfill@{#1}{#2}\mapsfromchar
            }
        \makeatother
		
		\newcommand{\cok}{\operatorname{coker}\nolimits}

		\newcommand{\iden}[1]{\tensor[]{\mathrm{id}}{_{#1}}}

		\newcommand{\idfunc}[1]{\tensor[]{\mathrm{id}}{_{#1}}}
	    \newcommand\restr[2]{{\left.\kern-\nulldelimiterspace#1
						\right|_{#2}}}
        \newcommand{\pent}{\mathchoice
  			{\pentago}% \displaystyle
  			{\pentago}% \textstyle
  			{\scalebox{.7}{\pentago}}% \scriptstyle
  			{\scalebox{.5}{\pentago}}% \scriptscriptstyle
		}
		\newcommand{\inn}{\mathchoice
  			{\raisebox{0.6pt}{\(\hspace{3pt}\in\hspace{3pt}\)}}% \displaystyle
  			{\raisebox{0.6pt}{\(\hspace{3pt}\in\hspace{3pt}\)}}% \textstyle
  			{\raisebox{0.5pt}{\scalebox{.5}{\hspace{1pt}\(\in\)\hspace{1pt}}}}% \scriptstyle
  			{\scalebox{.5}{\(\in\)}}% \scriptscriptstyle
		}

		\newcommand{\Ext}{\operatorname{Ext}\nolimits}
		\newcommand{\dExt}[1]{\operatorname{-Ext}\nolimits(#1)}

		\newcommand{\yExtn}[2]{\operatorname{YExt}_{#1}^{#2}\nolimits}
		\newcommand{\fs}{\mathfrak{s}}
		\newcommand{\ft}{\mathfrak{t}}

		\newcommand{\sus}{\Sigma} % shift/suspension functor notation
		\newcommand{\NT}{\Xi}

		\newcommand{\rMod}[1]{\operatorname{\mathsf{Mod}\,--}\nolimits{#1}}

		\newcommand{\kom}{\mathsf{K}}
		\newcommand{\com}{\mathsf{C}}
		
		\newcommand{\Exang}[1]{{#1}\operatorname{-\,\mathsf{Exang}}\nolimits}
		\newcommand{\exang}[1]{{#1}\operatorname{-\,\mathsf{exang}}\nolimits}
			
		\newcommand{\cone}{\mathrm{MC}}
		
		\newcommand{\combul}{\raisebox{0.5pt}{\scalebox{0.6}{\(\bullet\)}}}

		\let\amsamp=&

	\newcommand{\deff}{\coloneqq}
	
	\newcommand{\lan}{\langle}
	\newcommand{\ran}{\rangle}

	\newcommand{\wt}[1]{\widetilde{#1}}
	
	\newcommand{\ol}[1]{\overline{#1}}

\makeatletter
	\renewcommand{\andify}{%
		\nxandlist{\unskip, }{\unskip{} \@@and~}{\unskip \penalty-2 \space \@@and~}}
    
	\renewcommand\author@andify{%
  		\nxandlist {\unskip ,\penalty-1 \space\ignorespaces}%
		{\unskip {} \@@and~}%
		{\unskip \penalty-2 \space \@@and~}
	}

	    \newenvironment{acknowledgements}{%
	    \renewcommand\abstractname{\textbf{Acknowledgements}}
\global\setbox\abstractbox=\vtop \bgroup
\normalfont\Small
\list{}{\labelwidth\z@
\leftmargin3pc \rightmargin\leftmargin
\listparindent\normalparindent \itemindent\z@
\parsep\z@ \@plus\p@

}%
\item[\hskip\labelsep\scshape\abstractname.]%
}{%
\endlist\egroup
\ifx\@setabstract\relax \@setabstracta \fi
}

\let\oldtocsection=\tocsection
\let\oldtocsubsection=\tocsubsection

\renewcommand{\tocsection}[2]{\hspace{0em}\vspace*{1pt}\oldtocsection{#1}{#2}}

    \renewcommand{\tocsubsection}[2]{\hspace{21pt}\vspace*{0pt}\oldtocsubsection{#1}{#2}}
\makeatother

\begin{document}
\title[Categories of extensions and \(n\)-exangulated functors]{The category of extensions and a characterisation of \(n\)-exangulated functors}

    \author[Bennett-Tennenhaus]{Raphael Bennett-Tennenhaus}
         \address{
		Department of Mathematics\\
		Aarhus University\\
% 		Ny Munkegade 118\\
		8000 Aarhus C\\
		Denmark}
        \email{raphaelbennetttennenhaus@gmail.com}
        
    \author[Haugland]{Johanne Haugland}
        \address{Department of Mathematical Sciences\\ 
        NTNU\\ 
        NO-7491 Trondheim\\ 
        Norway}
        \email{johanne.haugland@ntnu.no}
        
    \author[Sand{\o{}}y]{Mads Hustad Sand{\o{}}y}
         \address{Department of Mathematical Sciences\\ 
         NTNU\\ 
         NO-7491 Trondheim\\ 
         Norway}
        \email{mads.sandoy@ntnu.no}
        
    \author[Shah]{Amit Shah}
        \address{
		Department of Mathematics\\
		Aarhus University\\
% 		Ny Munkegade 118\\
		8000 Aarhus C\\
		Denmark}
        \email{amit.shah@math.au.dk}
\date{\today}
\keywords{%
Additive category, 
biadditive functor, 
category of extensions, 
extriangulated category, 
extriangulated functor, 
% higher homological algebra, 
\(n\)-exangulated category, 
\(n\)-exangulated functor,
\(2\)-category,
\(2\)-functor}
\subjclass[2020]{Primary 18E05; Secondary 18E10, 18G80, 18N10}

{\setstretch{1}
\begin{abstract}
Additive categories play a fundamental role in mathematics and related disciplines. 
Given an
additive category 
equipped with a biadditive functor,
one can 
construct 
its category of extensions, 
which 
encodes important structural information. 
We study how functors between categories of extensions relate to those at the level of the 
original 
categories. 
When the additive categories in question are $n$-exangulated, this leads to a characterisation of $n$-exangulated functors.

Our approach enables us to study $n$-exangulated categories from a $2$-categorical perspective. 
We introduce 
$n$-exangulated natural transformations 
and characterise them 
using
categories of extensions. 
Our characterisations allow us to establish a $2$-functor 
between the $2$-categories of 
small $n$-exangulated categories and 
small exact categories. 
A similar result with no smallness assumption is also proved.

We employ our theory to produce various examples of $n$-exangulated functors and natural transformations. 
Although the motivation for this article stems from representation theory and the study of $n$-exangulated categories, our results are widely applicable: several require only an additive category equipped with a biadditive functor with no extra assumptions; others can be applied by endowing an additive category with its split $n$-exangulated structure.
\end{abstract}}

\maketitle

\vspace{-0.5cm}
    {\setstretch{1}
    % \small
    % \scriptsize
    \tableofcontents}
\vspace{-1cm}

%%%%%%%%%%%%%%%%%%%%%%%%%%%%%%%%%%%%%%%%%%%%%%%%%%%%%%%%%%%%%%%
%%%%%%%%%%%%%%%%%%%%%%%%%%%%%%%%%%%%%%%%%%%%%%%%%%%%%%%%%%%%%%%

\section{Introduction}
\label{sec:introduction}

\emph{Additive categories} appear in various branches of mathematics and related disciplines. For the more algebraically inclined mathematician, perhaps the category of abelian groups is the prototypical example; for the more analytical, perhaps the category of Banach spaces over the real numbers; and the geometer may opt for some category of sheaves. The unsuspecting theoretical physicist might uncover that certain additive categories 
control possibilities in 
string theory or, more broadly, particle physics. Phenomena of this last kind demonstrate the power of category theory and the importance of its study.

The motivation for this article stems from algebra, yet many of the results are widely applicable. 
Indeed, 
as we build up our theory, 
we ask only for an additive category \(\CC\) equipped with a biadditive functor \(\BE(-,-) \colon \CC^{\op}\times\CC \to \Ab\), where \(\Ab\) denotes the category of abelian groups. 
As an example, for any additive category \(\CC\), one can already take \(\BE(C,A)\) to be the abelian group \(\CC(C,A)\) of morphisms \(C\to A\); see \cref{example:arrow-category-as-category-of-extensions}. More interesting choices of \(\BE\) can be made depending on the structure of \(\CC\); see \cref{sec:examples-of-n-exangulated-categories-and-functors}.

For now, let us focus on two examples from classical homological algebra, namely \emph{abelian categories} and \emph{triangulated categories}. 
If \(\CC\) is a skeletally small abelian (or exact) category, then a possible choice for \(\BE\) is the functor \(\Ext_{\CC}^{1}\). In this case, the abelian group \(\BE(C,A) = \Ext_{\CC}^{1}(C,A)\) consists of equivalence classes of short exact sequences in \(\CC\) of the form 
\(
\begin{tikzcd}[column sep=0.5cm]
0 \arrow{r} &A \arrow{r}& - \arrow{r}& C\arrow{r}&0.
\end{tikzcd}
\)
If \(\CC\) is a triangulated category with suspension functor \(\sus\), then one could set \(\BE(C,A) = \CC(C,\sus A)\), which is in bijection with equivalence classes of distinguished triangles of the form 
\(
\begin{tikzcd}[column sep=0.5cm]
A \arrow{r} &- \arrow{r}& C \arrow{r}& \sus A.
\end{tikzcd}
\) 
In both these examples, we see that the bifunctor \(\BE\) encodes the basic building blocks of the homological structure of \(\CC\).

Nakaoka--Palu \cite{NakaokaPalu-extriangulated-categories-hovey-twin-cotorsion-pairs-and-model-structures} recently used the observations above to establish the theory of \emph{extriangulated categories}, giving a simultaneous generalisation of exact and triangulated categories.  
An extriangulated category consists of a triplet \((\CC,\BE,\fs)\), where \(\CC\) is an additive category, \(\BE\colon\CC^{\op}\times\CC \to \Ab\) is a biadditive functor 
and \(\fs\) is a \emph{realisation} (see \cref{def:exact-realisation}), satisfying certain axioms. 
This new framework has proven significance: 
aside from permitting the unification and extension of many known results 
(see e.g.\ 
\cite{HassounShah-integral-and-quasi-abelian-hearts-of-twin-cotorsion-pairs-on-extriangulated-categories,Haugland-the-grothendieck-group-of-an-n-exangulated-category,LiuYNakaoka-hearts-of-twin-cotorsion-pairs-on-extriangulated-categories,Msapato-the-karoubi-envelope-and-weak-idempotent-completion-of-an-extriangulated-category}), 
it has led to novel insights and explained previously mysterious connections 
(see e.g.\ 
\cite{JorgensenShah-grothendieck-groups-of-d-exangulated-categories-and-a-modified-CC-map,JorgensenShah-the-index-with-respect-to-a-rigid-subcategory,PadrolPaluPilaudPlamondon-associahedra-for-finite-type-cluster-algebras-and-minimal-relations-between-g-vectors}). 
This again underlines the benefits of abstraction in mathematics.

The realisation \(\fs\) of an extriangulated category \((\CC,\BE,\fs)\) encapsulates a core idea from classical homological algebra, namely that one \emph{realises} each \(\delta\inn\BE(C,A)\) by an equivalence class 
\(\fs(\delta)
    = [\begin{tikzcd}[column sep=0.5cm]
    A \arrow{r} & B \arrow{r}& C
    \end{tikzcd}]
\)
of a pair of composable morphisms.
Hence, the realisation allows us to visualise the structure encoded by \(\BE\) via \(3\)-term sequences. 
For an integer \(n\geq 1\), Iyama \cite{Iyama-higher-dimensional-auslander-reiten-theory-on-maximal-orthogonal-subcategories} discovered that \emph{\(n\)-cluster tilting} subcategories of module categories exhibit 
structures reminiscent of those from classical homological algebra, but now involving longer sequences. 
It was 
demonstrated 
that \((n+2)\)-term sequences could be used to study a higher-dimensional analogue of Auslander--Reiten theory in such settings. 
Jasso \cite{Jasso-n-abelian-and-n-exact-categories} formalised this idea by introducing \(n\)-exact and \(n\)-abelian categories. 
Similar observations in the triangulated setting (see e.g.\ 
\cite{Iyama-higher-dimensional-auslander-reiten-theory-on-maximal-orthogonal-subcategories},
\cite{Iyama-cluster-tilting-for-higher-auslander-algebras}, 
\cite{IyamaYoshino-mutation-in-tri-cats-rigid-CM-mods}) motivated the axiomatisation of \emph{\((n+2)\)-angulated categories} by Geiss--Keller--Oppermann \cite{GeissKellerOppermann-n-angulated-categories}. 
These formal frameworks constitute what has become known as \emph{higher homological algebra}, where the case \(n=1\) recovers classical theory of exact, abelian and triangulated categories. Higher homological algebra is linked to modern developments in various branches of mathematics, ranging from 
    representation theory, 
    cluster theory, 
    commutative algebra, 
    algebraic geometry, 
    homological mirror symmetry and 
    symplectic geometry 
to 
    string theory, 
    conformal field theory and 
    combinatorics 
(see e.g.\ 
\cite{%
    AmiotIyamaReiten-stable-categories-of-Cohen-Macaulay-modules-and-cluster-categories,%
    Dyckerhoff-Jasso-Lekili-the-symplectic-geometry-of-higher-auslander-algebras-symmetric-products-of-disks,%
    EvansPugh-the-Nakayama-automorphism-of-the-almost-Calabi-Yau-algebras-associated-to-SU3-modular-invariants,%
    GeissKellerOppermann-n-angulated-categories,%
    HauglandSandoy-higher-Koszul-duality-and-connections-with-n-hereditary-algebras,% 
    HerschendIyamaMinamotoOppermann-representation-theory-ofGeigle-Lenzing-complete-intersections,%
    IyamaWemyss-maximal-modifications-and-Auslander-Reiten-duality-for-non-isolated-singularities,%
    Jorgensen-tropical-friezes-and-the-index-in-higher-homological-algebra,%
    OppermannThomas-higher-dimensional-cluster-combinatorics-and-representation-theory,%
    Williams-new-interpretations-of-the-higher-stasheff-tamari-orders%
    }).

A central idea in the higher setup is that a suitable \(n\)-cluster tilting subcategory \(\CT\) of an abelian (resp.\ triangulated) category \(\CA\) is \(n\)-abelian \cite[Thm.~3.16]{Jasso-n-abelian-and-n-exact-categories} 
(resp.\ \((n+2)\)-angulated \cite[Thm.~1]{GeissKellerOppermann-n-angulated-categories}). 
Each admissible \(n\)-exact sequence 
\(
\begin{tikzcd}[column sep=0.5cm]
A \arrow{r} & X^{1} \arrow{r}& \cdots\arrow{r} & X^{n} \arrow{r} & C
\end{tikzcd}
\)
in the \(n\)-abelian category \(\CT\) 
is obtained by splicing together \(n\) short exact sequences 
\(
\begin{tikzcd}[column sep=0.35cm]
Y^{i-1} \arrow[hook]{r} & X^{i} \arrow[two heads]{r}&  Y^{i}
\end{tikzcd}
\) 
from \(\CA\) as indicated in the following diagram
\begin{equation}\label{eqn:splicing}
\tag{\(*\)}
\begin{tikzcd}[column sep=0.5cm, row sep=0.3cm]
    A \arrow{rr} \arrow[equal]{dr} & & X^{1} \arrow{rr} \arrow[two heads]{dr} & & X^{2}\arrow[two heads]{dr}\arrow{r} & \hspace{0.2cm}\cdots\hspace{0.2cm}\arrow{r} & X^{n} \arrow{rr} \arrow[two heads]{dr} && C. \\
    & Y^{0} \arrow[hook]{ur} && Y^{1} \arrow[hook]{ur} && \hspace{0.5cm}\cdots\hspace{0.5cm}\arrow[hook]{ur} && Y^{n} \arrow[equal]{ur}
\end{tikzcd}
\end{equation}
In this way, the higher structure of \(\CT\) is compatible with the classical 
structure of \(\CA\).

Of course, now one asks: What does ``compatible'' \emph{formally} mean? 
The answer is work in preparation by the authors (see 
\cite{Bennett-TennenhausHauglandSandoyShah-structure-preserving-functors-between-higher-exangulated-categories,Bennett-TennenhausHauglandSandoyShah-the-category-of-extensions-and-the-Krull-Remak-Schmidt-property}), but is inspired by the results and methodology from the present article. 
The approach taken involves the higher analogue of extriangulated categories, namely 
\emph{\(n\)-exangulated categories} in the sense of Herschend--Liu--Nakaoka \cite{HerschendLiuNakaoka-n-exangulated-categories-I-definitions-and-fundamental-properties} (see \cref{def:n-exangulated-category}). 
As for extriangulated categories, an \(n\)-exangulated category consists of a triplet \((\CC,\BE,\fs)\) where \(\CC\) is an additive category and \(\BE \colon \CC^{\op}\times\CC \to \Ab\) a biadditive functor. 
For each pair of objects \(A,C\inn\CC\) and to each 
\emph{extension} 
\(\delta\inn\BE(C,A)\), the realisation associates 
an equivalence class 
\(
\fs(\delta) 
    = [X^{\combul}]
    = [\begin{tikzcd}[column sep=0.5cm]
    A 
    \arrow{r}%{\tensor*[]{d}{_{X}^{0}}} 
    & X^{1}
    \arrow{r}%{\tensor*[]{d}{_{X}^{1}}}
    &\cdots
    \arrow{r}%{\tensor*[]{d}{_{X}^{n-1}}}
    & X^{n}
    \arrow{r}%{\tensor*[]{d}{_{X}^{n}}} 
    & C
    \end{tikzcd}]
\) 
of an \((n+2)\)-term sequence.
Each \((n+2)\)-angulated and \(n\)-exact category is \(n\)-exangulated 
(see 
Examples~\ref{example:n+2-angulated-category-is-n-exangulated} and \ref{example:n-exact-category-is-n-exangulated}), and a category is extriangulated if and only if it is \(1\)-exangulated 
(see 
\cref{example:extriangulated-is-1-exangulated}).

Structure-preserving functors between \(n\)-exangulated categories have been formalised recently 
in \cite{Bennett-TennenhausShah-transport-of-structure-in-higher-homological-algebra}. 
Given \(n\)-exangulated categories \((\CC,\BE,\fs)\) and \((\CC',\BE',\fs')\), an \emph{\(n\)-exangulated functor} \((\SF,\Gamma) \colon (\CC,\BE,\fs) \to (\CC',\BE',\fs')\) is a pair 
consisting of an additive functor \(\SF\colon \CC\to \CC'\) and a natural transformation 
\(\Gamma 
    \colon \BE(-,-) \Rightarrow \BE'(\SF-,\SF-)
\) 
satisfying a 
certain condition (see \cref{def:n-exangulated-functor}). 
Compatibility of structures can be naturally expressed by means of \(n\)-exangulated functors in the case where the domain and codomain categories are both \(n\)-exangulated for the same \(n\).

Let us return to the example of the inclusion functor \(\CT\into \CA\) of an \(n\)-cluster tilting subcategory \(\CT\) into an ambient abelian 
category \(\CA\). 
As soon as \(n > 1\), we have that \(\CT\) and \(\CA\) are higher abelian categories of differing ``dimension'', and hence the established notion of an \mbox{\(n\)-exangulated} functor does not apply. 
This demonstrates the need for 
terminology that allows one to describe compatibility of structures also in this more general setup. 
A naive 
attempt
to fill this gap might be to define an \emph{\((n,1)\)-exangulated functor} from an \(n\)-exangulated category \((\CC,\BE,\fs)\) to a \(1\)-exangulated category \((\CC',\BE',\fs')\) as a pair \((\SF,\Gamma)\), where \(\SF\colon \CC \to \CC'\) is an additive functor as before, but where \(\Gamma\) is now a natural transformation 
from \(\BE\) to an \(n\)-fold product arising from \(\BE'\) 
satisfying some compatibility conditions. 
For instance, in the situation of \eqref{eqn:splicing} above, 
one would want 
\(\Gamma\) to take one equivalence class 
\(
\delta
    = [\begin{tikzcd}[column sep=0.5cm]
    A \arrow{r} & X^{1} \arrow{r} &\cdots\arrow{r}& X^{n} \arrow{r} & C
    \end{tikzcd}]
\) 
of an admissible \(n\)-exact sequence 
to an \(n\)-tuple 
\((\rho_{n},\ldots,\rho_{1})\) 
of equivalence classes of 
short exact sequences with 
\(
\rho_{i} = [\begin{tikzcd}[column sep=0.5cm]
Y^{i-1} \arrow[hook]{r} & X^{i} \arrow[two heads]{r}&  Y^{i}
\end{tikzcd}]
\).
However, the careful reader quickly spots that for \(n>1\) there cannot 
be a 
natural transformation 
\(\Gamma\) of this kind, as 
the domain and codomain of \(\Gamma\) are functors with different domains. 
Furthermore, note that although we have focused on the \((n,1)\)-case above for expository purposes, 
we more generally aim to study \((n,q)\)-exangulated functors for \(q\geq 1\).

Instead of expecting to describe compatibility of \(n\)-exangulated and \(q\)-exangulated structures by use of a natural transformation \(\Gamma\) as above,
our main result of \cref{section:3} (see \cref{thmx:characterisation-of-n-exangulated-functors})
opens another avenue of approach. In this result, 
we characterise 
\(n\)-exangulated functors 
via the corresponding \emph{categories of extensions}. 
Given an additive category \(\CC\) with a biadditive functor \(\BE\colon\CC^{\op}\times\CC \to \Ab\), there is an associated category denoted by \(\BE\dExt{\CC}\), which has as its objects  
extensions \(\delta\inn\BE(C,A)\)
as \(A,C\) vary over objects in \(\CC\). 
For the unexplained terminology
used in \cref{thmx:characterisation-of-n-exangulated-functors}, see Definitions~\ref{def:respects-morphisms-over-SF} and \ref{def:SE-respects-distinguished-n-exangle-over-SF}.

\begin{thmx}[See \cref{thm:characterisation-of-n-exangulated-functors}]
\label{thmx:characterisation-of-n-exangulated-functors}  
Let \((\CC,\BE,\fs)\) and \((\CC',\BE',\fs')\) be \(n\)-exangulated categories. 
Then there is a one-to-one correspondence
\[
\begin{adjustbox}{
scale=0.93,
center}
\hspace{10pt}$\displaystyle
\begin{aligned}[t]
    \Set{
        \begin{array}{c}
            n\textit{-exangulated functors}\\
            (\SF,\Gamma)\colon (\CC,\BE,\fs) \to (\CC',\BE',\fs')
        \end{array}
    }
& \longleftrightarrow 
    \Set{
        \begin{array}{c}
            \textit{pairs } (\SF,\SE) \textit{ of additive functors }\SF \colon \CC \to \CC'\\\textit{and } \SE\colon\BE\dExt{\CC}\to\BE'\dExt{\CC'}\textit{, where } \SE \textit{ respects}\\\textit{ morphisms and distinguished }
            n\textit{-exangles over } \SF
        \end{array}
    }.
\end{aligned}$
\end{adjustbox}
\]
\end{thmx}

The category \(\BE\dExt{\CC}\) comes equipped with an exact structure \(\CX_{\BE}\) 
determined by the sections and retractions in \(\CC\); see \cref{prop:ECC-is-a-category} and \cref{rem:lit-review-of-the-category-of-extensions}. 
Furthermore, if \(\SF\colon \CC\to \CC'\) is an additive functor, then any functor \(\SE\colon \BE\dExt{\CC}\to\BE'\dExt{\CC'}\) which respects morphisms over \(\SF\) satisfies \(\SE(\CX_{\BE})\sse \CX_{\BE'}\), i.e.\ \(\SE\) is exact; see \cref{prop:F-additive-iff-E-additive-iff-E-exact}.

In addition to permitting a new perspective on the problem of defining structure-preserving functors between higher exangulated categories of possibly different dimensions, the 
one-to-one correspondence above 
is 
interesting in its own right. 
From \cref{thmx:characterisation-of-n-exangulated-functors} 
we deduce \cref{corx:criterion-for-n-exangulated-functor}, which provides a characterisation of what it means for an additive functor between \(n\)-exangulated categories to be \(n\)-exangulated.
This 
is a useful tool for detecting \(n\)-exangulated functors, because 
it is often easier to 
observe that distinguished \(n\)-exangles are 
sent to distinguished \(n\)-exangles in a functorial way, 
than to check that the corresponding natural transformation 
is indeed natural; 
see 
\cref{example:frobenius-n-exact-to-n-angulated-quotient,example:rest-yoneda-extrian-func}.

\begin{corx}
\label{corx:criterion-for-n-exangulated-functor}
Let \((\CC,\BE,\fs)\) and \((\CC',\BE',\fs')\) be \(n\)-exangulated categories. 
For an additive functor \(\SF\colon \CC \to \CC'\), the following statements are equivalent.
\begin{enumerate}[label=\textup{(\roman*)}]
    \item There exists a natural transformation \(\Gamma\colon\BE(-,-) \Rightarrow\BE'(\SF-,\SF-)\) such that the pair \((\SF,\Gamma)\) is an \(n\)-exangulated functor. 
    \item There exists an additive functor \(\SE\colon\BE\dExt{\CC}\to\BE'\dExt{\CC'}\) which respects both morphisms and distinguished \(n\)-exangles over \(\SF\).
\end{enumerate}
\end{corx}

In \cref{section:4} we study \(n\)-exangulated categories in a \(2\)-category-theoretic setting by considering morphisms of \(n\)-exangulated functors. We introduce a higher version of natural transformations of extriangulated functors as defined by Nakaoka--Ogawa--Sakai \cite[Def.~2.11(3)]{NakaokaOgawaSakai-localization-of-extriangulated-categories}, which we call  \emph{\(n\)-exangulated natural transformations}; see \cref{def:n-exangulated-natural-transformation}. 
Applying \cref{thmx:characterisation-of-n-exangulated-functors} and using the notation \(\tensor[]{\SE}{_{(\SF,\Gamma)}}\) for the 
exact functor 
\(
(\BE\dExt{\CC},\CX_{\BE})\to(\BE'\dExt{\CC'},\CX_{\BE'})
\) 
arising from an \(n\)-exangulated functor \((\SF,\Gamma)\colon (\CC,\BE,\fs) \to (\CC',\BE',\fs')\), we give a characterisation of \(n\)-exangulated natural transformations; see \cref{thmx:nat-tran-of-n-exan-functors-induces-nat-trans-of-functors-on-Ext-categories}. 
In the following we use the Hebrew letter \(\beth\) (beth). 
See \cref{def:balanced-natural-transformation} for the meaning of \emph{balanced}.

\begin{thmx}[See \cref{thm:characterisation-of-n-exangulated-natural-transformations}]
\label{thmx:nat-tran-of-n-exan-functors-induces-nat-trans-of-functors-on-Ext-categories}
Suppose \((\SF,\Gamma),(\SG,\Lambda)\colon (\CC,\BE,\fs)\to(\CC',\BE',\fs')\) are \(n\)-exan\-gu\-lat\-ed functors. 
Then there is a one-to-one correspondence between 
\(n\)-exangulated natural transformations 
\((\SF,\Gamma)\overset{\beth}{\Longrightarrow}(\SG,\Lambda)\) 
and balanced natural transformations 
\(\tensor[]{\SE}{_{(\SF,\Gamma)}}\overset{\lan\beth\ran}{\Longrightarrow}\tensor[]{\SE}{_{(\SG,\Lambda)}}\). 
\end{thmx}

For \(n\geq 1\), we consider the category \(\Exang{n}\) of \emph{all} \(n\)-exangulated categories, which has properties just like a \(2\)-category (see e.g.\ \cref{prop:hom-category-of-n-exangulated-categories}). However, due to the set-theoretic issue outlined in \cref{rem:not-2-categories}, we cannot formally call \(\Exang{n}\) a \(2\)-category. 
If we consider instead \emph{small} categories, then we avoid such problems, and may talk of the \(2\)-category \(\exang{n}\) of small \(n\)-exangulated categories. 
We use the correspondences of 
Theorems~\ref{thmx:characterisation-of-n-exangulated-functors} and \ref{thmx:nat-tran-of-n-exan-functors-induces-nat-trans-of-functors-on-Ext-categories} to construct a \(2\)-functor from 
\(\exang{n}\) to the \(2\)-category \(\exactcat\) of small exact categories.

\begin{thmx}[See \cref{cor:finalcorollary}]
\label{thmx:finalcorollary}
There is \(2\)-functor \(\exang{n} \to \exactcat\), which sends a \(0\)-cell \((\CC,\BE,\fs)\) to \((\BE\dExt{\CC},\CX_{\BE})\), a \(1\)-cell \((\SF,\Gamma)\) to \(\tensor[]{\SE}{_{(\SF,\Gamma)}}\) and a \(2\)-cell \(\beth\) to \(\lan\beth\ran\).
\end{thmx}

\cref{thmx:finalcorollary} is a consequence of \cref{thm:2-functor}. This latter result is more general, 
in that one can construct a functor from \(\Exang{n}\) to the category \(\Exactcat\) of all exact categories which behaves just like the \(2\)-functor described in \cref{thmx:finalcorollary}. 
Ignoring  
\cref{rem:not-2-categories}, one should interpret \cref{thm:2-functor} as establishing a \(2\)-functor \(\Exang{n} \to \Exactcat\).

In \cref{sec:examples-of-n-exangulated-categories-and-functors} we provide several examples of \(n\)-exangulated categories, functors and natural transformations. 
Some of these examples also produce \(n\)-exangulated subcategories in the sense of \cite[Def.~3.7]{Haugland-the-grothendieck-group-of-an-n-exangulated-category}.

\medskip

\textbf{Conventions.}\; 
We write \(A\inn \CC\) to denote that an object \(A\) lies in a category \(\CC\). 
For \(A,B\inn \CC\), we write \(\CC(A,B)\) for the collection of morphisms \(A\to B\) in \(\CC\). 
Unless stated otherwise: our subcategories are always assumed to be full; and our functors are always assumed to be  covariant.  
We write \(\Ab\) for the category of abelian groups. 
Throughout this paper, let \(n\geq 1\) denote a positive integer. 

%%%%%%%%%%%%%%%%%%%%%%%%%%%%%%%%%%%%%%%%%%%%%%%%
%%%%%%%%%%%%%%%%%%%%%%%%%%%%%%%%%%%%%%%%%%%%%%%%

\section{Preliminaries on \texorpdfstring{\(n\)}{n}-exangulated categories}
\label{sec:n-exangulated-categories}

We follow \cite[Sec.\ 2]{HerschendLiuNakaoka-n-exangulated-categories-I-definitions-and-fundamental-properties} in 
briefly 
recalling the definition of an \(n\)-exangulated category, 
which is a higher analogue of an extriangulated category as introduced in \cite{NakaokaPalu-extriangulated-categories-hovey-twin-cotorsion-pairs-and-model-structures}. 
See also \cite{HerschendLiuNakaoka-n-exangulated-categories-II}.

\begin{setup}
\label{setup:2.1}
Throughout this section, we assume that \(\CC\) is an additive category and that \(\BE\colon\CC^{\op}\times\CC\to \Ab\) is a \emph{biadditive} functor.   
The latter means that for all \(A,C\inn\CC\), the functors \(\BE(C,-)\colon \CC \to \Ab\) and \(\BE(-,A)\colon \CC^{\op}\to\Ab\) are both additive. 
\end{setup}

Let \(A,C \inn\CC\) be arbitrary. 
The identity element of the abelian group \(\BE(C,A)\) is denoted by \(\tensor[_{A}]{0}{_{C}}\).  
An element \(\delta\) of \(\BE(C,A)\) is called an \emph{\(\BE\)-extension},
and we set
\[
\tensor[]{x}{_{\BE}}\delta\coloneqq\BE(C,x)(\delta)\inn\BE(C,X)
\hspace{1cm}\text{and} 
\hspace{1cm}
\tensor[]{z}{^{\BE}}\delta\coloneqq\BE(z,A)(\delta)\inn\BE(Z,A)
\]
for morphisms \(x\colon A\to X\) and \(z\colon Z\to C\) in \(\CC\). 
It follows that 
\(
\tensor[]{z}{^{\BE}}\tensor[]{x}{_{\BE}}\delta
    = \BE(z,x)(\delta)
    = \tensor[]{x}{_{\BE}}\tensor[]{z}{^{\BE}}\delta
\). 
Given \(\delta\inn\BE(C,A)\) and \(\rho\inn\BE(D,B)\), a \emph{morphism of \(\BE\)-extensions} from \(\delta\) to \(\rho\) is a pair \((a,c)\) of morphisms \(a\colon A\to B\) and \(c\colon C\to D\) in \(\CC\) such that 
\begin{equation}
\label{eqn:morphism-of-extensions-defining-property}
\tensor[]{a}{_{\BE}}\delta  
    = \tensor[]{c}{^{\BE}}\rho.
\end{equation}
If there is no confusion about the biadditive functor involved, \(\BE\)-extensions and morphisms of \(\BE\)-extensions are simply called \emph{extensions} and \emph{morphisms of extensions}, respectively. The Yoneda Lemma yields two natural transformations denoted and defined by
\[
\begin{array}{l}
\hspace{6.7pt}
\tensor*[_{\BE}]{\delta}{}\colon
 \CC(A,-) \Longrightarrow \BE(C,-)\\
\tensor*[_{\BE}]{\delta}{_{X}}\colon \hspace{33pt}x\longmapsto \tensor[]{x}{_{\BE}}\delta
\end{array}
\hspace{1cm}\text{and} 
\hspace{1cm}
\begin{array}{l}
\hspace{5.5pt}\tensor[^{\BE}]{\delta}{}\colon \CC(-,C)\Longrightarrow \BE(-,A)
\\ \tensor*[^{\BE}]{\delta}{_{Z}}\colon\hspace{34pt} z\longmapsto\tensor[]{z}{^{\BE}}\delta.
\end{array}
\]

In order to explain how to associate a homotopy class of a complex to each extension, we recall some 
terminology and notation.
We denote by \(\com_{\CC}^{\raisebox{0.5pt}{\scalebox{0.6}{\(n\)}}}\) the 
subcategory of the category of complexes \(\tensor[]{\com}{_{\CC}}\) in \(\CC\) consisting of complexes concentrated in degrees \(0,1,\ldots, n,n+1\). 
That is, if \(X^{\combul}\inn\com_{\CC}^{\raisebox{0.5pt}{\scalebox{0.6}{\(n\)}}}\), then \(X^{i}\) is zero if \(i<0\) or \(i>n+1\). 
We 
write such a complex as 
\[
\begin{tikzcd}
X^{0} \arrow{r}{\tensor*[]{d}{_{X}^{0}}} & X^{1} \arrow{r}{\tensor*[]{d}{_{X}^{1}}}& \cdots \arrow{r}{\tensor*[]{d}{_{X}^{n-1}}}& X^{n} \arrow{r}{\tensor*[]{d}{_{X}^{n}}} & X^{n+1}.
\end{tikzcd}
\]

\begin{defn}
\label{def:n-exangles}
(See \cite[Def.~2.13]{HerschendLiuNakaoka-n-exangulated-categories-I-definitions-and-fundamental-properties}.) 
Suppose \(X^{\combul}\inn\com_{\CC}^{\raisebox{0.5pt}{\scalebox{0.6}{\(n\)}}}\) and \(\delta\inn\BE(X^{n+1},X^{0})\).  
If
\[
\begin{tikzcd}[column sep=1.7cm]
\CC(-,X^{0})\arrow[Rightarrow]{r}[yshift=2pt]{\CC(-,\tensor*[]{d}{_{X}^{0}})}&\CC(-,X^{1})\arrow[Rightarrow]{r}[yshift=2pt]{\CC(-,\tensor*[]{d}{_{X}^{1}})}&\cdots\arrow[Rightarrow]{r}[yshift=2pt]{\CC(-,\tensor*[]{d}{_{X}^{n}})}&\CC(-,X^{n+1})\arrow[Rightarrow]{r}[yshift=2pt]{\tensor[^{\BE}]{\delta}{}}& \BE(-,X^{0})
\end{tikzcd}
\]
and 
\[
\begin{tikzcd}[column sep=1.7cm]
\CC(X^{n+1},-)\arrow[Rightarrow]{r}[yshift=2pt]{\CC(\tensor*[]{d}{_{X}^{n}},-)}&\CC(X^{n},-)\arrow[Rightarrow]{r}[yshift=2pt]{\CC(\tensor*[]{d}{_{X}^{n-1}},-)}&\cdots\arrow[Rightarrow]{r}[yshift=2pt]{\CC(\tensor*[]{d}{_{X}^{0}},-)}&\CC(X^{0},-)\arrow[Rightarrow]{r}[yshift=2pt]{\tensor[_{\BE}]{\delta}{}}& \BE(X^{n+1},-)
\end{tikzcd}
\]
are both exact sequences of functors, then we call the pair \(\lan X^{\combul},\delta\ran\) an \(n\)-\emph{exangle}.
\end{defn}

For 
\(A,C\inn\CC\) the not-necessarily-full subcategory 
\(
\com_{(A,C)}^{\raisebox{0.5pt}{\scalebox{0.6}{\(n\)}}}
\) 
of \(\com_{\CC}^{\raisebox{0.5pt}{\scalebox{0.6}{\(n\)}}}\) is defined as follows. Objects of \(\com_{(A,C)}^{\raisebox{0.5pt}{\scalebox{0.6}{\(n\)}}}\) are complexes \(X^{\combul}\inn\com_{\CC}^{\raisebox{0.5pt}{\scalebox{0.6}{\(n\)}}}\) with \(X^{0}=A\) and \(X^{n+1}=C\). 
Given \(X^{\combul}, Y^{\combul}\inn\com_{(A,C)}^{\raisebox{0.5pt}{\scalebox{0.6}{\(n\)}}}\), set
\[
\com_{(A,C)}^{\raisebox{0.5pt}{\scalebox{0.6}{\(n\)}}}(X^{\combul},Y^{\combul})\deff
\set{
f^{\combul} = (f^{0},\ldots, f^{n+1})\inn\com_{\CC}^{\raisebox{0.5pt}{\scalebox{0.6}{\(n\)}}}(X^{\combul},Y^{\combul})
|
f^{0}=\iden{A}\text{ and }f^{n+1}=\iden{C}
}.
\]
The usual notion of a \emph{homotopy} between morphisms of complexes restricts to give an equivalence relation \(\sim\) on \(\com_{(A,C)}^{\raisebox{0.5pt}{\scalebox{0.6}{\(n\)}}}(X^{\combul},Y^{\combul})\). 
This gives rise to a new category \(\kom_{(A,C)}^{\raisebox{0.5pt}{\scalebox{0.6}{\(n\)}}}\) with the same objects as \(\com_{(A,C)}^{\raisebox{0.5pt}{\scalebox{0.6}{\(n\)}}}\) and with
\(
\kom_{(A,C)}^{\raisebox{0.5pt}{\scalebox{0.6}{\(n\)}}}(X^{\combul},Y^{\combul})\deff \com_{(A,C)}^{\raisebox{0.5pt}{\scalebox{0.6}{\(n\)}}}(X^{\combul},Y^{\combul})/{\sim}.
\)
If the image of a morphism  \(f^{\combul}\inn\com_{(A,C)}^{\raisebox{0.5pt}{\scalebox{0.6}{\(n\)}}}(X^{\combul},Y^{\combul})\) in \(\kom_{(A,C)}^{\raisebox{0.5pt}{\scalebox{0.6}{\(n\)}}}(X^{\combul},Y^{\combul})\) is an isomorphism, then \(f^{\combul}\) is called a \emph{homotopy equivalence} and we say that \(X^{\combul}\) and \(Y^{\combul}\) are \emph{homotopy equivalent}. 
We denote the isomorphism class in \(\kom_{(A,C)}^{\raisebox{0.5pt}{\scalebox{0.6}{\(n\)}}}\) of an object \(X^{\combul}\) by \([X^{\combul}]\).

\begin{defn}
\label{def:exact-realisation}
(See \cite[Def.~2.22]{HerschendLiuNakaoka-n-exangulated-categories-I-definitions-and-fundamental-properties}.) 
Let \(\fs\) be a correspondence that, for each \(A,C\inn\CC\), associates to an extension  \(\delta\inn\BE(C,A)\) 
an isomorphism class \(\fs(\delta)=[X^{\combul}]\) 
in 
\(\kom_{(A,C)}^{\raisebox{0.5pt}{\scalebox{0.6}{\(n\)}}}\). 
Such an \(\fs\) is said to be an \emph{exact realisation of \(\BE\)} if the following conditions are satisfied. 

\begin{enumerate}[label=\textup{(R\arabic*)},
    labelsep=5pt, 
    leftmargin=35.00003pt,
]
\setcounter{enumi}{-1}

    \item\label{R0}
    Let \(\delta\inn\BE(C,A)\) and \(\rho\inn\BE(D,B)\) be extensions with  \(\fs(\delta)=[X^{\combul}]\) and \( \fs(\rho)=[Y^{\combul}]\). For any morphism of extensions \((a,c)\colon \delta\to\rho\), there exists \(f^{\combul}\inn\com_{\CC}^{\raisebox{0.5pt}{\scalebox{0.6}{\(n\)}}}(X^{\combul},Y^{\combul})\) such that \(f^{0}=a\) and \(f^{n+1}=c\). 
    In this setting, we say that \(X^{\combul}\) \emph{realises} \(\delta\) and \(f^{\combul}\)  \emph{realises} \((a,c)\).

    \item\label{R1}
    The pair \(\lan X^{\combul},\delta\ran\) is an \(n\)-exangle whenever \(\fs(\delta)=[X^{\combul}]\).

    \item\label{R2}
    For all \(A\inn\CC\), we have   
    \(\fs(\tensor[_{A}]{0}{_{0}})=[
    \begin{tikzcd}[column sep=1.2cm] A\arrow{r}{\iden{A}}&A\arrow{r}&0\arrow{r}&\cdots\arrow{r}&0\end{tikzcd}
    ]\) 
    and 
    \(\fs(\tensor[_{0}]{0}{_{A}})=[
    \begin{tikzcd}[column sep=1.2cm]
    0\arrow{r}&\cdots\arrow{r}&0\arrow{r}&A\arrow{r}{\iden{A}}&A\end{tikzcd}
    ]\).
\end{enumerate}
\end{defn}

If \(\fs\) is an exact realisation of \(\BE\) and 
\(
\fs(\delta) 
	= [X^{\combul}] 
	= [
	\begin{tikzcd}[column sep=0.55cm]
	X^{0} \arrow{r}{\tensor*[]{d}{_{X}^{0}}}& X^{1} \arrow{r}& \cdots \arrow{r}& X^{n}\arrow{r}{\tensor*[]{d}{_{X}^{n}}}& X^{n+1}
	\end{tikzcd}
	],
\)
then \(\tensor*[]{d}{_{X}^{0}}\) is known as an \emph{\(\fs\)-inflation} and \(\tensor*[]{d}{_{X}^{n}}\) as an \emph{\(\fs\)-deflation}. 

Before stating the main definition of this section, we recall the notion of a mapping cone.

\begin{defn}
\label{def:mapping-cone-as-in-Jasso}
(See \cite[Def.~2.27]{HerschendLiuNakaoka-n-exangulated-categories-I-definitions-and-fundamental-properties}.) 
Suppose 
\(
f^{\combul}\inn\com_{\CC}^{\raisebox{0.5pt}{\scalebox{0.6}{\(n\)}}}(X^{\combul},Y^{\combul})
\) with \(f^{0}=\iden{A}\) for some \(A=X^{0}=Y^{0}\). 
The \emph{mapping cone} \(M^{\combul}\deff \cone(f)^{\combul}\) of \(f^{\combul}\) is the complex
\[
\begin{tikzcd}
X^{1}
    \arrow{r}{d_{M}^{0}}
& X^{2}\oplus Y^{1} 
    \arrow{r}{d_{M}^{1}}
&X^{3}\oplus Y^{2}
    \arrow{r}{d_{M}^{2}}
& \cdots
    \arrow{r}{d_{M}^{n-1}}
&X^{n+1}\oplus Y^{n} 
    \arrow{r}{d_{M}^{n}}
& Y^{n+1}
\end{tikzcd}
\]
in \(\com_{\CC}^{\raisebox{0.5pt}{\scalebox{0.6}{\(n\)}}}\), 
where 
\(
d_{M}^{0}
    \deff 
        \begin{psmallmatrix}
        -\tensor*[]{d}{_{X}^{1}}\\
        f^{1} 
        \end{psmallmatrix}
\), 
\(
d_{M}^{n}
    \deff
        (\, f^{n+1} \;\, d_{Y}^{n} \,)
\), 
and 
\(
d_{M}^{i}
    \deff
        \begin{psmallmatrix}
        -d_{X}^{i+1} & 0\\
        f^{i+1} & d_{Y}^{i} 
        \end{psmallmatrix}
\) 
when \(0< i <n\).
\end{defn}

\begin{defn}
\label{def:n-exangulated-category}
(See \cite[Def.~2.32]{HerschendLiuNakaoka-n-exangulated-categories-I-definitions-and-fundamental-properties}.) 
Let \(\CC\) be an additive category, 
\(\BE\colon\CC^{\op}\times\CC\to \Ab\) a biadditive functor 
and \(\fs\) an exact realisation of \(\BE\). 
The triplet \((\CC,\BE,\fs)\) is called an \emph{\(n\)-exangulated category} if the following conditions are satisfied.

\begin{enumerate}[label=\textup{(EA\arabic*)},
wide=0pt, leftmargin=46pt, labelwidth=41pt, labelsep=5pt, align=right]

    \item\label{nEA1}
    The composition of any two \(\fs\)-inflations is again an \(\fs\)-inflation, and the composition of any two \(\fs\)-deflations is again an \(\fs\)-deflation.
        
    \item\label{nEA2}
    For any \(\delta\inn\BE(D,A)\) and any \(c\inn\CC(C,D)\) with \(\fs(\tensor[]{c}{^{\BE}}\delta)=[X^{\combul}]\) and \(\fs(\delta)=[Y^{\combul}]\), there exists a morphism  \(f^{\combul}\colon X^{\combul}\to Y^{\combul}\) realising \((\iden{A},c)\) such that \(\fs((\tensor*[]{d}{_{X}^{0}}\tensor[]{)}{_{\BE}}\delta)=[ \cone(f)^{\combul}]\).

    \item[\textup{(EA}\(2^{\op}\))]\label{nEA2op}
    Dual of \ref{nEA2}.
    
\end{enumerate}
\end{defn}

If \((\CC,\BE,\fs)\) is an \(n\)-exangulated category and \(\fs(\delta)=[X^{\combul}]\) for an extension \(\delta\inn\BE(C,A)\), then we call \(\lan X^{\combul},\delta\ran\) a \emph{distinguished} \(n\)-exangle.

%%%%%%%%%%%%%%%%%%%%%%%%%%%%%%%%%%%%%%%%%%%%%%%%%%%%%
%%%%%%%%%%%%%%%%%%%%%%%%%%%%%%%%%%%%%%%%%%%%%%%%%%%%%

\section{The category of extensions and 
\texorpdfstring{\(n\)}{n}-exangulated functors}
\label{section:3}

Our main result in \cref{section:3} is  \cref{thm:characterisation-of-n-exangulated-functors}, which is \cref{thmx:characterisation-of-n-exangulated-functors}  from \cref{sec:introduction}. In \cref{subsection:3.1} we recall the definition of the category of extensions associated to an additive category equipped with a biadditive functor. 
In Subsection~\ref{subsection:3.2} we characterise natural transformations of a certain form; see Setup~\ref{setup:3} and \cref{prop:functor-on-extensions}. 
In Subsection~\ref{subsection:3.3} we use this characterisation 
to prove \cref{thm:characterisation-of-n-exangulated-functors}.

%%%%%%%%%%%%%%%%%%%%%%%%%%%%%%%%%%%%%%%%%%%%%%%%
%%%%%%%%%%%%%%%%%%%%%%%%%%%%%%%%%%%%%%%%%%%%%%%%

\subsection{The category of extensions}
\label{subsection:3.1}

For this subsection, assume that \(\CC\) is an additive category and that \(\BE\colon\CC^{\op}\times\CC\to \Ab\) is a biadditive functor (see Setup~\ref{setup:2.1}). 
The \emph{category \(\BE\dExt{\CC}\) of \(\BE\)-extensions} was considered in \cite[Def.~2.3]{NakaokaPalu-extriangulated-categories-hovey-twin-cotorsion-pairs-and-model-structures}. 
The authors thank Thomas Br{\"{u}}stle for informing them 
that similar ideas already appeared in the literature before; see \cref{rem:lit-review-of-the-category-of-extensions}.

The objects and morphisms of \(\BE\dExt{\CC}\) are given by \(\BE\)-extensions and morphisms of \(\BE\)-extensions, respectively, 
as defined in \cref{sec:n-exangulated-categories}. 
Recall that, for \(\delta\inn\BE(C,A)\) and \(\rho\inn\BE(D,B)\), a morphism \(\delta\to\rho\) of \(\BE\)-extensions is a pair \((a,c)\) of morphisms \(a\colon A\to B\) and \(c\colon C\to D\) in \(\CC\) such that 
\(
\tensor[]{a}{_{\BE}}\delta
    = \tensor[]{c}{^{\BE}}\rho;
\) 
see \eqref{eqn:morphism-of-extensions-defining-property}. 
If \(\delta\) is an extension in \(\BE(C,A)\), then the identity morphism \(\iden{\delta}\) of \(\delta\) is given by the pair \((\iden{A},\iden{C})\). The composition of morphisms \((a,c)\colon\delta\to\rho\) and \((b,d)\colon\rho\to\eta\) in \(\BE\dExt{\CC}\) is the pair \((ba,dc)\). It is straightforward to check that \((ba,dc)\) is again a morphism of extensions and that \(\BE\dExt{\CC}\) is a category under this composition rule.

As one might expect, the category of \(\BE\)-extensions is an additive category. 
Moreover, motivated by 
\cite[Sec.\ 9.1, Exam.\ 5]{GabrielRoiter-reps-of-finite-dimensional-algebras}, 
we show that \(\BE\dExt{\CC}\) can be equipped with an exact structure \(\CX_{\BE}\) that is not necessarily the split exact structure. Suppose \(\delta\inn\BE(C,A)\),
\(\rho'\inn\BE(D,B)\) and 
\(\eta\inn\BE(C',A')\). 
We declare that a sequence 
\begin{equation}
\label{eqn:conflation-in-exact-structure}
\begin{tikzcd}
\delta \arrow{r}{(a,c)}& \rho' \arrow{r}{(b,d)}& \eta
\end{tikzcd}
\end{equation}
of morphisms in \(\BE\dExt{\CC}\) lies in \(\CX_{\BE}\) 
if and only if 
\(a\) and \(c\) are both sections with \(b=\cok a\) and \(d=\cok c\).
It follows from \cite[Prop.\ 2.7]{Shah-AR-theory-quasi-abelian-cats-KS-cats} that 
\eqref{eqn:conflation-in-exact-structure} belongs to \(\CX_{\BE}\) 
if and only if 
\(b\) and \(d\) are both retractions with \(a=\ker b\) and \(c=\ker d\).  
This is again equivalent to the underlying sequences 
\(
\begin{tikzcd}[column sep=0.5cm]
A \arrow{r}{a}& B\arrow{r}{b} & A'
\end{tikzcd}
\) 
and 
\(
\begin{tikzcd}[column sep=0.5cm]
C \arrow{r}{c}& D\arrow{r}{d} & C'
\end{tikzcd}
\) 
being split exact in \(\CC\). We call a sequence \eqref{eqn:conflation-in-exact-structure} belonging to \(\CX_{\BE}\) a \emph{conflation}, and in this case the morphism \((a,c)\) an \emph{inflation} and \((b,d)\) a \emph{deflation}. Notice that \(\CX_{\BE}\) is closed under isomorphisms.

We use column and row notation \(\tensor[]{\iota}{_{X}} = \begin{psmallmatrix} \iden{X} \\ 0\end{psmallmatrix} \colon X \to X \oplus Y\) and \(\tensor[]{\pi}{_{X}} = \begin{psmallmatrix}
\iden{X} & 0 
\end{psmallmatrix} \colon X \oplus Y \to X\) for the canonical inclusion and projection, respectively, associated to the biproduct of two objects \(X\) and \(Y\) in the additive category \(\CC\). 
Then, given a conflation \eqref{eqn:conflation-in-exact-structure}, 
there are isomorphisms \(h\colon B \to A\oplus A'\) and \(g\colon D\to C\oplus C'\) in \(\CC\) 
such that \eqref{eqn:conflation-in-exact-structure} is isomorphic to 
\begin{equation}
\label{eqn:conflation-in-exact-structure-canonical-split}
\begin{tikzcd}[column sep=1.6cm]
\delta 
    \arrow{r}{(\tensor[]{\iota}{_{A}},\tensor[]{\iota}{_{C}})}
&\rho
    \arrow{r}{(\tensor[]{\pi}{_{A'}},\tensor[]{\pi}{_{C'}})}
&\eta,
\end{tikzcd}
\end{equation} 
where \(\rho = (g^{-1}\tensor[]{)}{^{\BE}}\tensor[]{h}{_{\BE}}\rho' \). 
Moreover, the sequence \eqref{eqn:conflation-in-exact-structure-canonical-split} also lies in \(\CX_{\BE}\). 
If \((a,c)\) is a morphism 
in \(\BE\dExt{\CC}\)
consisting of a pair of sections both admitting cokernels, 
it is not a priori clear that it is an inflation. The following lemma verifies this.

\begin{lem}
\label{lem:every-pair-of-sections-is-an-inflation}
Let \((a,c)\colon\delta \to \rho'\) be a morphism in \(\BE\dExt{\CC}\) for \(\delta\inn\BE(C,A)\) and
\(\rho'\inn\BE(D,B)\). Suppose
\(a\) and \(c\) are sections with cokernels \(b=\cok a\) and \(d=\cok c\). 
Then \((a,c)\) completes to a kernel-cokernel pair  \eqref{eqn:conflation-in-exact-structure}, which is in \(\CX_{\BE}\). 
\end{lem}

\begin{proof}
By our remarks above, we may assume \(B = A\oplus A'\), \(D = C\oplus C'\) 
and 
that \((a,c)\) is of the form \((\tensor[]{\iota}{_{A}},\tensor[]{\iota}{_{C}})\colon \delta \to \rho\). 
Consider the sequence \eqref{eqn:conflation-in-exact-structure-canonical-split} with \(\eta 
    \deff (\tensor[]{\iota}{_{C'}}\tensor[]{)}{^{\BE}}
            (\tensor[]{\pi}{_{A'}}\tensor[]{)}{_{\BE}}\rho
\). 
Using that  
\((\tensor[]{\iota}{_{A}},\tensor[]{\iota}{_{C}})
\) is a morphism in \(\BE\dExt{\CC}\) 
and writing \(\iden{D}\) as \(\tensor[]{\iota}{_{C}}\tensor[]{\pi}{_{C}}+\tensor[]{\iota}{_{C'}}\tensor[]{\pi}{_{C'}}\), 
it is straightforward to check that
\( ( \tensor[]{\pi}{_{A'}} , \tensor[]{\pi}{_{C'}} ) \) 
is a morphism of extensions \(\rho \to \eta\). 
It follows that 
\eqref{eqn:conflation-in-exact-structure-canonical-split} 
is in \(\CX_{\BE}\).

To show that 
\eqref{eqn:conflation-in-exact-structure-canonical-split} is a kernel-cokernel pair, we first observe that 
\((\tensor[]{\pi}{_{A'}},\tensor[]{\pi}{_{C'}})(\tensor[]{\iota}{_{A}},\tensor[]{\iota}{_{C}}) = 0\). 
Let \(\alpha\inn \BE(Z,X)\) 
be an extension 
and consider a morphism 
\(\left(
\begin{psmallmatrix}
x\\ x'
\end{psmallmatrix},
\begin{psmallmatrix}
z\\ z'
\end{psmallmatrix}\right)
\colon \alpha \to \rho
\) 
in \(\BE\dExt{\CC}\) 
with 
\(
( \tensor[]{\pi}{_{A'}} , \tensor[]{\pi}{_{C'}} )
\left(
\begin{psmallmatrix}
x\\ x'
\end{psmallmatrix},
\begin{psmallmatrix}
z\\ z'
\end{psmallmatrix}
\right)
 = (0,0)
\). 
This implies that \(x'\) and \(z'\) are zero. 
We claim that \((x,z)\) is a morphism of extensions \(\alpha \to \delta\). Indeed, we have 
\[
(\tensor[]{\iota}{_{A}}\tensor[]{)}{_{\BE}} 
    \tensor[]{x}{_{\BE}}\alpha
    = \tensor[]{\begin{psmallmatrix}
        x\\0
    \end{psmallmatrix}}{_{\BE}}\alpha
    = \tensor[]{\begin{psmallmatrix}
        x\\x'
    \end{psmallmatrix}}{_{\BE}}\alpha
    = \tensor[]{\begin{psmallmatrix}
        z\\z'
    \end{psmallmatrix}}{^{\BE}}\rho
    = \tensor[]{\begin{psmallmatrix}
        z\\0
    \end{psmallmatrix}}{^{\BE}}\rho
    = \tensor[]{z}{^{\BE}}
        (\tensor[]{\iota}{_{C}}\tensor[]{)}{^{\BE}}\rho
    = (\tensor[]{\iota}{_{A}}\tensor[]{)}{_{\BE}}
        \tensor[]{z}{^{\BE}}\delta,
\]
where the last equality follows from   \((\tensor[]{\iota}{_{A}},\tensor[]{\iota}{_{C}})
\) being a morphism 
of extensions.
As \((\tensor[]{\iota}{_{A}}\tensor[]{)}{_{\BE}}\) is monic since \(\tensor[]{\iota}{_{A}}\) is a section, this yields
\(\tensor[]{x}{_{\BE}}\alpha = \tensor[]{z}{^{\BE}}\delta\). 
Moreover, 
we 
conclude that \((x,z)\) is the unique morphism \(\alpha \to \delta\) satisfying
\(
\left(
\begin{psmallmatrix}
x\\ x'
\end{psmallmatrix},
\begin{psmallmatrix}
z\\ z'
\end{psmallmatrix}\right)
    = (\tensor[]{\iota}{_{A}},\tensor[]{\iota}{_{C}})
        (x,z),
\) 
so \((\tensor[]{\iota}{_{A}},\tensor[]{\iota}{_{C}})\) is a kernel of \((\tensor[]{\pi}{_{A'}},\tensor[]{\pi}{_{C'}})\). 
Similarly, one can 
verify that 
\((\tensor[]{\pi}{_{A'}},\tensor[]{\pi}{_{C'}})\) 
is a cokernel of 
\((\tensor[]{\iota}{_{A}},\tensor[]{\iota}{_{C}})\). 
\end{proof}

We are now ready to show that \(\CX_{\BE}\) is an exact structure on \(\BE\dExt{\CC}\).

\begin{prop}
\label{prop:ECC-is-a-category}
The pair \((\BE\dExt{\CC},\CX_{\BE})\) is an exact category.
\end{prop}

\begin{proof}
We first verify that \(\BE\dExt{\CC}\) is an additive category. 
Let \(\delta\inn\BE(C,A)\) and \(\rho\inn\BE(D,B)\) be extensions. The collection  
of morphisms of extensions \(\delta\to\rho\) in \(\BE\dExt{\CC}\) is a set, since \(\CC(A,B)\) and \(\CC(C,D)\) are both groups and hence sets.
The addition of morphisms \((a,c)\colon\delta\to\rho\) and \((a',c')\colon\delta\to\rho\) 
is defined by the pair \((a+a',c+c')\), which is a morphism \(\delta\to\rho\) of extensions as \(\BE\) is biadditive. This establishes that \(\BE\dExt{\CC}\) is preadditive. 
By the biadditivity of \(\BE\), we have a natural isomorphism
\(
\BE(C \oplus D, A \oplus B) \iso \BE(C,A) \oplus \BE(C,B) \oplus \BE(D,A) \oplus \BE(D,B).
\)
As in \cite[Def.~2.6]{HerschendLiuNakaoka-n-exangulated-categories-I-definitions-and-fundamental-properties}, we let \(\delta \oplus \rho \inn\BE(C \oplus D, A \oplus B)\) denote the element corresponding to \((\delta,0,0,\rho)\) via this isomorphism. 
It is straightforward to check that this gives a biproduct of extensions making \(\BE\dExt{\CC}\) an additive category; see, for instance, Liu--Tan \cite[Rem.~2]{Liu-Tan-Resolution-Dimension-Relative-to-Resolving-Subcategories-in-Extriangulated-Categories}. 
In particular, note that the biproduct inclusion and projection morphisms are of the form 
\(
\tensor[]{\iota}{_{\delta}} = (\tensor[]{\iota}{_{A}},\tensor[]{\iota}{_{C}})\colon \delta\to \delta\oplus \rho
\) 
and 
\(
\tensor[]{\pi}{_{\delta}} = (\tensor[]{\pi}{_{A}},\tensor[]{\pi}{_{C}})\colon  \delta\oplus \rho\to\delta
\).

Next we show that \(\CX_{\BE}\) is an exact structure on \(\BE\dExt{\CC}\). 
It follows from \cref{lem:every-pair-of-sections-is-an-inflation} that \(\CX_{\BE}\) consists of kernel-cokernel pairs, using that (co)kernels are unique up to isomorphism.
To check the axioms as in B\"{u}hler \cite[Def.\ 2.1]{Buhler-exact-categories} of an exact category, it suffices to consider sequences 
of the form \eqref{eqn:conflation-in-exact-structure-canonical-split} as \(\CX_{\BE}\) is closed under isomorphisms. 
The identity morphism of \(\delta\inn\BE(C,A)\) is \(\iden{\delta} = (\iden{A},\iden{C})\), which is a pair of sections admitting cokernels, so (E0) holds. 
The collection of morphisms of extensions that consist of pairs of sections that admit cokernels is closed under composition, so (E1) follows from \cref{lem:every-pair-of-sections-is-an-inflation}. 
We prove (E2) below, and note that axioms (E\(0^{\op}\)), (E\(1^{\op}\)) and (E\(2^{\op}\)) can be shown dually.

For (E2), 
suppose we have a conflation \eqref{eqn:conflation-in-exact-structure-canonical-split}. 
Let 
\((u,w)\colon \delta \to \beta\) 
be an arbitrary morphism 
where 
\(\beta\inn\BE(W,U)\). 
By the universal property of the product in \(\BE\dExt{\CC}\), 
there exists a unique morphism 
\((e,f)\colon \delta\to \beta\oplus \rho\) 
for which 
\(\tensor[]{\pi}{_{\beta}}(e,f) = (-u,-w)\) and 
\(\tensor[]{\pi}{_{\rho}}(e,f) = (\tensor[]{\iota}{_{A}},\tensor[]{\iota}{_{C}})\).
It follows that 
\((e,f) = \left(\begin{psmallmatrix}
-u \\ \iden{A} \\0
\end{psmallmatrix}, 
\begin{psmallmatrix}
-w \\ \iden{C} \\0
\end{psmallmatrix}\right)\) 
is a pair of sections with 
\(l \deff \cok e = \begin{psmallmatrix}
\iden{U} & u & 0 \\
0 & 0 & \iden{A'}
\end{psmallmatrix}\) 
and
\(
m \deff \cok f = \begin{psmallmatrix}
\iden{W} & w & 0\\
0 & 0 & \iden{C'}
\end{psmallmatrix}
\). 
By \cref{lem:every-pair-of-sections-is-an-inflation}, this implies that \((e,f)\) fits into a conflation  
\(
\begin{tikzcd}
\delta \arrow{r}{(e,f)}& \beta\oplus\rho \arrow{r}{(l,m)}& \gamma
\end{tikzcd}
\) 
with \(\gamma\inn\BE(W\oplus C', U\oplus A')\). 
It is straightforward to check that 
\(\gamma\) equipped with 
\((l,m)\tensor[]{\iota}{_{\beta}}\colon \beta\to\gamma\)
and 
\((l,m)\tensor[]{\iota}{_{\rho}}\colon \rho\to\gamma\)
is a pushout of \((\tensor[]{\iota}{_{A}},\tensor[]{\iota}{_{C}})\) along \((u,w)\). 
Lastly, we note that  
\(
(l,m)\tensor[]{\iota}{_{\beta}}
    = \left( 
    \begin{psmallmatrix}
    \iden{U} \\ 0
    \end{psmallmatrix},
    \begin{psmallmatrix}
    \iden{W} \\ 0
    \end{psmallmatrix}
    \right)
\) 
is a pair of sections admitting cokernels, 
and hence an inflation by \cref{lem:every-pair-of-sections-is-an-inflation}.
\end{proof}

In \cref{example:arrow-category-as-category-of-extensions} we consider the category of extensions given by the \(\Hom\)-bifunctor.

\begin{example}
\label{example:arrow-category-as-category-of-extensions}
Consider the biadditive functor 
\(
\BE(-,-) \deff \CC(-,-) \colon \CC^{\op}\times\CC\to\Ab
\). 
With this choice, the objects in \(\BE\dExt{\CC}\) coincide with morphisms in \(\CC\). 
For \(\delta \inn\BE(C,A)\) and \(\rho \inn\BE(D,B)\), 
a morphism 
\(\delta \to \rho\) in \(\BE\dExt{\CC}\) 
is given by a pair \((a,c)\) with \(a \inn\CC(A,B)\) and \(c \inn\CC(C,D)\) such that \(\tensor[]{a}{_{\BE}}\delta  = \tensor[]{c}{^{\BE}}\rho\), i.e.\ such that the square
\[
\begin{tikzcd}[column sep=1cm]
C \arrow{d}{c} \arrow{r}{\delta} & A \arrow{d}{a}\\
D \arrow{r}{\rho} & B
\end{tikzcd}
\]
commutes in \(\CC\). 
It follows that 
\(\BE\dExt{\CC}\) is the \emph{arrow category of \(\CC\)}. 
Furthermore, note that given \(\delta \inn\BE(C,A)\) and \(\rho \inn\BE(D,B)\), the biproduct
\(\delta\oplus \rho\) in the additive category \(\BE\dExt{\CC}\) is the morphism 
\(
\begin{psmallmatrix}
\delta  & 0 \\
0       & \rho
\end{psmallmatrix}
\colon C \oplus D \to A \oplus B
\).

A sequence \eqref{eqn:conflation-in-exact-structure-canonical-split} in \(\CX_{\BE}\) corresponds to a morphism \(
\begin{psmallmatrix}
\delta  & \alpha \\
0       & \eta
\end{psmallmatrix}
\colon C \oplus C' \to A \oplus A',
\) where 
\(\alpha\colon C'\to A\) can be taken to be arbitrary. 
Such a conflation is trivial if and only if we have \(\alpha=\delta\gamma-\beta\eta\) for some \(\gamma\colon C'\to C\) and some \(\beta\colon A'\to A\). 
For an example of a non-trivial conflation, one may hence take \(\CC=\Ab\), \(\alpha=\iden{\BZ}\) and \(\delta=\eta\) to be the endomorphism of \(\BZ\) given by \(d\mapsto 2d\).
\end{example}

\begin{rem}
\label{rem:lit-review-of-the-category-of-extensions}
The authors are grateful to Thomas Br{\"{u}}stle for pointing out that variants of the category \(\BE\dExt{\CC}\) have been 
studied before. 
Gabriel--Nazarova--Roiter--Sergeichuk--Vossieck \cite[Sec.\ 1]{GabrielNazarovaRoiterSergeichukVossieck-tame-and-wild-subspace-problems} considered a category of \(M\)-\emph{spaces} for a functor \(M\) from an \emph{aggregate} (that is, a skeletally small, \(\Hom\)-finite, Krull--Schmidt category) 
to a category of vector spaces. 
Gabriel--Roiter \cite[p.\ 88, Exam.\ 5]{GabrielRoiter-reps-of-finite-dimensional-algebras} looked more generally at a category defined by a bifunctor on a pair of aggregates, and this context was generalised further by Dr{\"{a}}xler--Reiten--Smal{\o}--Solberg \cite[p.\ 670]{DraxlerReitenSmaloSolberg-exact-categories-and-vector-space-categories}. 
These examples also have analogues, where one restricts the focus to extensions of the form \(\delta\inn\BE(C,A)\) with  \(A=C\); see, for example,  
Crawley-Boevey \cite{Crawley-Boevey-Matrix-problems-and-Drozds-theorem}, 
Tiefenbrunner \cite{Tiefenbrunner-phd-thesis}, 
Gei{\ss} \cite{geiss-Deformations-of-bimodule-problems} 
and 
Br{\"{u}}stle--Hille \cite{Brustle-Hille-Matrices-over-upper-triangular-bimodules}.
\end{rem}

%%%%%%%%%%%%%%%%%%%%%%%%%%%%%%%%%%%%%%%%%%%%%%%%%%%
%%%%%%%%%%%%%%%%%%%%%%%%%%%%%%%%%%%%%%%%%%%%%%%%%%%

\subsection{Functors between categories of extensions}
\label{subsection:3.2}

In this subsection we discuss how functors between categories of extensions relate to functors on the underlying categories. 
This culminates in \cref{prop:functor-on-extensions}, from which 
\cref{thm:characterisation-of-n-exangulated-functors} in the next subsection will follow.

\begin{setup}
\label{setup:3}
For the remainder of this subsection, let \(\SF\colon\CC\to\CC'\) be a functor between additive categories \(\CC\) and \(\CC'\). Suppose also that \(\BE\colon \CC^{\op}\times\CC\to\Ab\) and \(\BE'\colon (\CC')^{\op}\times\CC'\to\Ab\) are biadditive functors. 
\end{setup}

In Setup~\ref{setup:3} we do not assume the functor \(\SF\) to be additive. We explicitly impose this requirement whenever needed in the results that follow. Associated to \(\SF\) is the \emph{opposite functor} \(\SF^{\op} \colon \CC^{\op}\to (\CC')^{\op}\), and we usually abuse notation by writing \(\SF\) instead of \(\SF^{\op}\).

\begin{defn}
\label{def:respects-morphisms-over-SF}
We say that a functor \(\SE\colon\BE\dExt{\CC}\to\BE'\dExt{\CC'}\) \emph{respects morphisms over} \(\SF\) 
if, for every morphism \((a,c)\colon\delta\to\rho\) in \(\BE\dExt{\CC}\), the morphism \(\SE(a,c)\colon \SE(\delta)\to\SE(\rho)\) in \(\BE'\dExt{\CC'}\)
is given by the pair \((\SF a, \SF c)\).
\end{defn}

Building on \cref{example:arrow-category-as-category-of-extensions}, the following shows that a functor \(\SF\) between additive categories always induces a functor between categories of extensions that respects morphisms over \(\SF\).

\begin{example}
\label{example:arrow-category-as-category-of-extensions-part2}
Recall that if we put \(\BE(-,-) = \CC(-,-)\) and \(\BE'(-,-) = \CC'(-,-)\), 
then \(\BE\dExt{\CC}\) and \(\BE'\dExt{\CC'}\) coincide with the arrow categories of \(\CC\) and \(\CC'\), respectively; see \cref{example:arrow-category-as-category-of-extensions}. 
Since functors preserve commutative squares, any functor \(\SF\colon \CC \to \CC'\) induces a functor 
\(\SE\colon \BE\dExt{\CC} \to \BE'\dExt{\CC'}\). 
This functor is defined by 
\(\SE(\delta) = \SF\delta\) for \(\delta\inn \BE(C,A)\), and by \(\SE(a,c) = (\SF a,\SF c) \colon \SE(\delta) \to \SE(\rho)\) for each morphism \((a,c)\colon \delta \to \rho\) in \(\BE\dExt{\CC}\). Note that \(\SE\) respects morphisms over \(\SF\) by construction.

Note that \cref{lem:functor-respects-morphisms-plays-well-with-domains-codomains-functions} below holds trivially in the setup of this example. In particular, the equations in 
\cref{lem:functor-respects-morphisms-plays-well-with-domains-codomains-functions}(ii) say that \(\SF\) respects composition of morphisms. It is clear 
that if \(\SF\) is additive, then \(\SE\) is additive. The converse also holds, but involves a trick; see the proof of \cref{prop:F-additive-iff-E-additive-iff-E-exact}. 
In this case, we 
have \(\SE(\tensor[]{\delta}{_{1}} + \tensor[]{\delta}{_{2}}) = \SE(\tensor[]{\delta}{_{1}})+\SE(\tensor[]{\delta}{_{2}})\) for \(\tensor[]{\delta}{_{1}},\tensor[]{\delta}{_{2}} \inn\BE(C,A)\). 
These statements hold more generally; see \cref{prop:F-additive-iff-E-additive-iff-E-exact} and \cref{prop:respecting-extensions-preserves-additivity}.
\end{example}

\begin{rem}
\label{rem:subtle-remark}
Note that even though a functor \(\SE\colon\BE\dExt{\CC}\to\BE'\dExt{\CC'}\) that respects morphisms over \(\SF\) sends a pair \((a,c)\) of morphisms 
of
\(\CC\) to the pair \((\SF a,\SF c)\) of morphisms 
of 
\(\CC'\), 
this does \emph{not} mean that \(\SE\) is determined on all morphisms of \(\BE\dExt{\CC}\). 
For instance, in \cref{example:arrow-category-as-category-of-extensions-part2} one could also consider a functor \(\wt{\SE}\) that respects morphisms over \(\SF\), but is defined by \(\wt{\SE}(\delta)= -\SF \delta\) on objects. 
Despite \(\SE(a,c)\) and \(\wt{\SE}(a,c)\) both being equal to \((\SF a, \SF c)\) as pairs of morphisms 
of
\(\CC'\), we might have \(\SE(a,c) \neq \wt{\SE}(a,c)\)
as morphisms in \(\BE'\dExt{\CC'}\), since the domains 
or the codomains may not agree. 
Indeed, if \(\SF\delta\neq -\SF\delta\) in \(\BE'(\SF C,\SF A)\), then 
the morphisms 
\(\SE(\iden{\delta})
    = \iden{\SF \delta}
\) 
and 
\(\wt{\SE}(\iden{\delta}) 
    = \iden{-\SF \delta} 
\)
are not the same identity morphisms in \(\BE'\dExt{\CC'}\), even though they are both given by the pair  
\((\iden{\SF A},\iden{\SF C})\). 
Hence, the notion of respecting morphisms is not as strict as it may seem. In particular, a functor that is determined on all morphisms is automatically also determined on all objects, but this is not necessarily the case for functors respecting morphisms. 
\end{rem}

As the point made in \cref{rem:subtle-remark} is subtle, we now spell out explicitly what it means for 
two morphisms 
in 
the category \(\BE\dExt{\CC}\) to be equal. 

\begin{rem}
\label{rem:equality-of-morphisms-of-extensions}
Let \(\delta\inn\BE(C,A)\), \(\delta'\inn\BE(C',A')\), \(\rho\inn\BE(D,B)\) and \(\rho'\inn\BE(D',B')\) be extensions. Suppose one fixes morphisms \((a,c)\colon \delta\to \rho\) and \((a',c')\colon \delta'\to \rho'\) in the category \(\BE\dExt{\CC}\). These morphisms are equal in \(\BE\dExt{\CC}\) if and only if we have the equalities:
\begin{enumerate}[label=\textup{(\roman*)}]
    \item\label{item:equal-objects}
    \(A=A'\), \(B=B'\), \(C=C'\) and \(D=D'\) as objects in \(\CC\);
    
    \item\label{item:equal-domain-codomain}
    \(\delta=\delta'\) as elements in \(\BE(C,A)\) and \(\rho=\rho'\) as elements in \(\BE(D,B)\); and 
    
    \item\label{item:equal-morphisms}
    \(a=a'\) and \(c=c'\) as morphisms in \(\CC\).
\end{enumerate}
When 
checking that two morphisms \((a,c)\colon\delta\to\rho\) and \((a',c')\colon\delta'\to\rho'\) as above are equal, it is 
usually straightforward---but nonetheless essential---to verify requirements \ref{item:equal-objects} and \ref{item:equal-domain-codomain}. 
The verification of \ref{item:equal-morphisms} is 
typically less straightforward and often involves \cref{def:respects-morphisms-over-SF}.
\end{rem}

Our next lemma, which we use to prove \cref{prop:respecting-extensions-preserves-additivity,prop:functor-on-extensions}, describes natural compatibility properties for a functor that respects morphisms over \(\SF\).

\begin{lem}
\label{lem:functor-respects-morphisms-plays-well-with-domains-codomains-functions}
Let \(\SE\colon\BE\dExt{\CC}\to\BE'\dExt{\CC'}\) be a functor that respects morphisms over \(\SF\).
The following statements hold for \(\delta\inn\BE(C,A)\).
\begin{enumerate}[label=\textup{(\roman*)}]

    \item The extension \(\SE(\delta)\) lies in \(\BE'(\SF C,\SF A)\).

    \item If \(x\colon A\to X\) and \(z\colon Z\to C\) are morphisms in \(\CC\), then
        \begin{center}
        \(
        \SE(\tensor[]{x}{_{\BE}}\delta)
            =(\SF x\tensor[]{)}{_{\BE'}}\SE(\delta)
        \hspace{1cm}\text{and}\hspace{1cm}
        \SE(\tensor[]{z}{^{\BE}}\delta)
            =(\SF z\tensor[]{)}{^{\BE'}}\SE(\delta).
        \) 
        \end{center}
\end{enumerate}
\end{lem}

\begin{proof}
(i)\;
As \(\SE\) is a functor, we have 
\(
\iden{\SE(\delta)}
    = \SE(\iden{A},\iden{C})
\). 
Since \(\SE\) respects morphisms over \(\SF\), 
this equals 
\(
(\SF \iden{A}, \SF \iden{C}) 
    = (\iden{\SF A}, \iden{\SF C})
\), 
so \(\SE(\delta)\inn\BE'(\SF C,\SF A)\). 

(ii)\;
We only demonstrate the first identity, as the second is dual. 
By \cite[Rem.~2.4]{HerschendLiuNakaoka-n-exangulated-categories-I-definitions-and-fundamental-properties}, the pair
\(
(x,\iden{C})\colon \delta\to \tensor[]{x}{_{\BE}}\delta
\)
is a morphism in \(\BE\dExt{\CC}\). 
Since \(\SE\) is a functor and respects morphisms over \(\SF\), this implies that \(\SE(x,\iden{C})=(\SF x,\iden{\SF C})\) is a morphism 
\(\SE(\delta) \to \SE(\tensor[]{x}{_{\BE}}\delta)\) in \(\BE'\dExt{\CC'}\). Consequently, we have \(
(\SF x\tensor[]{)}{_{\BE'}}\SE(\delta)
    = (\iden{\SF C}\tensor[]{)}{^{\BE'}}\SE(\tensor[]{x}{_{\BE}}\delta)
    = \SE(\tensor[]{x}{_{\BE}}\delta)
\).
\end{proof}

The following proposition establishes the connection between the additivity of a functor \(\SF\) and that of, or the exactness of, a functor which respects morphisms over \(\SF\).

\begin{prop}
\label{prop:F-additive-iff-E-additive-iff-E-exact}
Let \(\SE\colon\BE\dExt{\CC}\to\BE'\dExt{\CC'}\) be a functor that respects morphisms over \(\SF\). 
The following statements are equivalent.
\begin{enumerate}[\textup{(\roman*)}]
    \item\label{item:F-additive} The functor \(\SF\colon\CC\to\CC'\) is additive.
    \item\label{item:E-additive} The functor \(\SE\colon\BE\dExt{\CC}\to\BE'\dExt{\CC'}\) is additive.
    \item\label{item:E-exact} The functor \(\SE\colon (\BE\dExt{\CC},\CX_{\BE}) \to (\BE'\dExt{\CC'},\CX_{\BE'})\) is exact.
\end{enumerate}
\end{prop}

\begin{proof}
\(\ref{item:F-additive}\Rightarrow\ref{item:E-exact}\)\;
For a conflation \eqref{eqn:conflation-in-exact-structure} in \(\CX_{\BE}\), 
the underlying sequences 
\(
\begin{tikzcd}[column sep=0.5cm]
A \arrow{r}& B\arrow{r} & A'
\end{tikzcd}
\) 
and 
\(
\begin{tikzcd}[column sep=0.5cm]
C \arrow{r}& D\arrow{r} & C'
\end{tikzcd}
\) are split exact in \(\CC\). 
As \(\SF\) is additive, their images 
\(
\begin{tikzcd}[column sep=0.5cm]
\SF A \arrow{r}& \SF B\arrow{r} & \SF A'
\end{tikzcd}
\) 
and 
\(
\begin{tikzcd}[column sep=0.5cm]
\SF C \arrow{r}& \SF D\arrow{r} & \SF C'
\end{tikzcd}
\) 
under \(\SF\) are split exact in \(\CC'\). 
Since \(\SE\) respects morphisms over \(\SF\), the image
\(
\begin{tikzcd}[column sep=0.4cm]
\SE(\delta) \arrow{r} & \SE(\rho') \arrow{r} & \SE(\eta)
\end{tikzcd}
\)
of \eqref{eqn:conflation-in-exact-structure} under \(\SE\) is a sequence in \(\CX_{\BE'}\).

\(\ref{item:E-exact}\Rightarrow\ref{item:E-additive}\)\;
An exact functor 
preserves finite biproducts and is hence additive.

\(\ref{item:E-additive}\Rightarrow\ref{item:F-additive}\)\;
Recall that for any morphism \(a\colon A\to B\) in \(\CC\) and for the unique element \(\delta\inn\BE(0,A)\), there is a morphism \((a,\iden{0})\colon \delta\to \tensor[]{a}{_{\BE}}\delta\) in \(\BE\dExt{\CC}\). 
Let \(a,b\colon A\to B\) be morphisms in \(\CC\). Note that the morphisms \((a,\iden{0}) \colon \delta \to a_{\BE}\delta\) and \((b,\iden{0}) \colon \delta \to b_{\BE}\delta\) in \(\BE\dExt{\CC}\) have the same codomain, as \(a_{\BE}\delta = b_{\BE}\delta\) in the trivial abelian group \(\BE(0,B)\). 
Consequently, we can add 
\((a,\iden{0})\) and \((b,\iden{0})\)
as morphisms \(\delta\to a_{\BE}\delta\), and so the sum 
\(
\SE(a,\iden{0}) + \SE(b,\iden{0})
\)
also makes sense. 
We then have
\begin{align*}
(\SF(a+b),\SF \iden{0})
    &= \SE(a+b, \iden{0}) && \text{as }\SE\text{ respects morphisms over }\SF\\
    &= \SE(a+b, \iden{0} + \iden{0}) && \text{as }\CC(0,0)=\set{\iden{0}}\text{ is trivial}\\
    &= \SE(a,\iden{0})+\SE(b,\iden{0}) && \text{as }\SE\text{ is additive by assumption}\\
    &= (\SF a, \SF\iden{0}) + (\SF b, \SF \iden{0})&& \text{as }\SE\text{ respects morphisms over }\SF\\
    &= (\SF a+\SF b,\SF \iden{0} + \SF \iden{0}).
\end{align*}
The computation 
yields \(\SF(a+b) = \SF a + \SF b\), and so \(\SF\) is additive.
\end{proof}

In the proof of \cref{prop:functor-on-extensions}, we use that functors which respect morphisms over additive functors preserve the additivity of extensions. This is shown in \cref{prop:respecting-extensions-preserves-additivity} below, for which we first recall some notation.

\begin{notn}\label{not:delta,nabla} (See \cite[Def.~2.6]{HerschendLiuNakaoka-n-exangulated-categories-I-definitions-and-fundamental-properties}.)
Given an object \(X\) in an additive category, we use the notation \(\tensor[]{\iota}{_{X,1}}\) and \(\tensor[]{\iota}{_{X,2}}\) for the canonical inclusion morphisms \(X\to X\oplus X\) into the first and second summand of the biproduct, respectively. Similarly, we write \(\tensor[]{\pi}{_{X,1}}\) and \(\tensor[]{\pi}{_{X,2}}\) for the canonical projection morphisms \(X\oplus X\to X\). 
Let \(\tensor[]{\Delta}{_{X}}\colon X\to X\oplus X\) (resp.\ \(\tensor[]{\nabla}{_{X}}\colon X\oplus X\to X\)) denote the diagonal (resp.\ codiagonal) morphism of \(X\), i.e.\ the unique morphism such that \(\tensor[]{\pi}{_{X,i}}\circ\tensor[]{\Delta}{_{X}}=\iden{X}\) (resp.\ \(\tensor[]{\nabla}{_{X}}\circ\tensor[]{\iota}{_{X,i}}=\iden{X}\)) for 
\(i=1,2\). Note that for any \(A,C\inn\CC\) and any extensions \(\tensor[]{\delta}{_{1}},\tensor[]{\delta}{_{2}}\inn\BE(C,A)\), the addition in the abelian group \(\BE(C,A)\) relates to the
biproduct in \(\BE\dExt{\CC}\) via the equation
\begin{equation}
\label{eqn:internal-sum-of-BE-and-the-Baer-sum}
\tensor[]{\delta}{_{1}}+\tensor[]{\delta}{_{2}}
    =\BE(\tensor[]{\Delta}{_{C}},\tensor[]{\nabla}{_{A}})(\tensor[]{\delta}{_{1}}\oplus\tensor[]{\delta}{_{2}}).
\end{equation}
\end{notn}

\begin{prop}
\label{prop:respecting-extensions-preserves-additivity} 
Suppose that \(\SF\) is additive and let  
\(\SE\colon\BE\dExt{\CC}\to\BE'\dExt{\CC'}\) be a functor that respects morphisms over \(\SF\). 
For all \(A,C\inn\CC\) and for all extensions \(\tensor[]{\delta}{_{1}},\tensor[]{\delta}{_{2}}\inn\BE(C,A)\), 
we have that 
\(\SE(\tensor[]{\delta}{_{1}} + \tensor[]{\delta}{_{2}})=\SE(\tensor[]{\delta}{_{1}})+\SE(\tensor[]{\delta}{_{2}})\).
\end{prop}

\begin{proof} 
Since \(\SF\) is additive, there exists an isomorphism 
\(\tensor[]{f}{_{X}}\colon \SF(X\oplus X) \to \SF X\oplus \SF X\) for each \(X\inn\CC\), such that for each \(i=1,2\) the diagram
\begin{equation}
\label{eqn:morphism-f-X-for-biproduct}
\begin{tikzcd}[column sep=2.5cm, row sep=0.3cm]
        &\SF(X \oplus X)
            \arrow{dd}{\iso}[swap]{\tensor[]{f}{_{X}}}   & \\
\SF X  
    \arrow{ur}{\SF\tensor[]{\iota}{_{X,i}}}
    \arrow{dr}[swap]{\tensor[]{\iota}{_{\SF X,i}}}
        &&      \SF X
                \arrow
                    [leftarrow]{ul}[swap]{\SF \tensor[]{\pi}{_{X,i}}} 
                \arrow
                    [leftarrow]{dl}{\tensor[]{\pi}{_{\SF X,i}}} \\
        &\SF X\oplus \SF X&
\end{tikzcd}
\end{equation}
in \(\CC'\) commutes. 
For \(1\leq i,j \leq 2\), it follows that 
\begin{equation}
\label{eqn:isomorphism-of-biproducts-and-inclusion-projection-maps}
\BE'(\tensor[]{\iota}{_{\SF C,i}}, \tensor[]{\pi}{_{\SF A,j}}) 
= \BE'(\SF\tensor[]{\iota}{_{C,i}}, \SF\tensor[]{\pi}{_{A,j}})\circ \BE'(\tensor[]{f}{_{C}}, \tensor*[]{f}{_{A}^{-1}}).
\end{equation}

Recall from \cite[Def.~2.6]{HerschendLiuNakaoka-n-exangulated-categories-I-definitions-and-fundamental-properties} 
that \(\tensor[]{\delta}{_{1}}\oplus\tensor[]{\delta}{_{2}}\) is the unique element in \(\BE(C\oplus C,A\oplus A)\) satisfying 
\begin{equation}
\label{eqn:equations-for-rho-oplus-rho-prime}
\BE(\tensor[]{\iota}{_{C,i}},\tensor[]{\pi}{_{A,j}})(\tensor[]{\delta}{_{1}}\oplus\tensor[]{\delta}{_{2}}) =
\begin{cases}
\delta_{i} & \text{if } i=j\\
\tensor[_{A}]{0}{_{C}} & \text{if } i\neq j,
\end{cases}
\end{equation} 
and that \(\SE(\tensor[]{\delta}{_{1}})\oplus\SE(\tensor[]{\delta}{_{2}})\) is the unique element in \(\BE'(\SF C\oplus \SF C,\SF A\oplus \SF A)\) satisfying 
\begin{equation}\label{eqn:equations-for-E-rho-oplus-E-rho-prime}
\BE'(\tensor[]{\iota}{_{\SF C,i}},\tensor[]{\pi}{_{\SF A,j}})(\SE(\tensor[]{\delta}{_{1}})\oplus\SE(\tensor[]{\delta}{_{2}})) =
\begin{cases}
\SE(\delta_{i}) & \text{if } i=j\\
\tensor[_{\SF A}]{0}{_{\SF C}} & \text{if } i\neq j.
\end{cases}
\end{equation}

Consider the element 
\(
\eta \deff \BE'(\tensor[]{f}{_{C}},\tensor*[]{f}{_{A}^{-1}}\tensor[]{)}{^{-1}}(\SE(\tensor[]{\delta}{_{1}}\oplus\tensor[]{\delta}{_{2}}))
\) 
in \(\BE'(\SF C\oplus \SF C,\SF A\oplus \SF A)\). 
We claim that \(\eta\) and \(\SE(\tensor[]{\delta}{_{1}})\oplus\SE(\tensor[]{\delta}{_{2}})\) are equal. Note that it suffices to show that \(\eta\) satisfies the defining equations \eqref{eqn:equations-for-E-rho-oplus-E-rho-prime} of \(\SE(\tensor[]{\delta}{_{1}})\oplus\SE(\tensor[]{\delta}{_{2}})\). 
In order to verify this, let us first consider the zero morphism \(\tensor[]{0}{_{A}}\colon A\to A\). 
The induced homomorphism 
\(
(\tensor[]{0}{_{A}}\tensor[]{)}{_{\BE}}
    = \BE(C,\tensor[]{0}{_{A}})
    \colon \BE(C,A)\to \BE(C,A)
\) is 
trivial, and hence 
\(
(\tensor[]{0}{_{A}}\tensor[]{)}{_{\BE}}(\tensor[_{A}]{0}{_{C}})
    = \tensor[_{A}]{0}{_{C}}
\). 
Similarly, the homomorphism \((\tensor[]{0}{_{\SF A}}\tensor[]{)}{_{\BE'}}\) is the zero map, so 
\(
(\tensor[]{0}{_{\SF A}}\tensor[]{)}{_{\BE'}}(\SE(\tensor[_{A}]{0}{_{C}}))
    =\tensor[_{\SF A}]{0}{_{\SF C}}
\).
Using these observations, we have that 
\begin{align*}
\SE(\tensor[_{A}]{0}{_{C}})
    &= \SE((\tensor[]{0}{_{A}}\tensor[]{)}{_{\BE}}(\tensor[_{A}]{0}{_{C}}))\\
    &= (\tensor[]{\SF 0}{_{A}}\tensor[]{)}{_{\BE'}}(\SE(\tensor[_{A}]{0}{_{C}}))
        &&\text{by \cref{lem:functor-respects-morphisms-plays-well-with-domains-codomains-functions}(ii)}\\
    &= (\tensor[]{0}{_{\SF A}}\tensor[]{)}{_{\BE'}}(\SE(\tensor[_{A}]{0}{_{C}}))
        &&\text{since }\SF\text{ is additive}\\
    &= \tensor[_{\SF A}]{0}{_{\SF C}}.
\end{align*}
This gives 
\begin{equation}
\label{eqn:equations-for-rho-oplus-rho-prime-after-applying-E}
\BE'(\SF\tensor[]{\iota}{_{C,i}},\SF\tensor[]{\pi}{_{A,j}})(\SE(\tensor[]{\delta}{_{1}}\oplus\tensor[]{\delta}{_{2}})) 
= \SE (\BE(\tensor[]{\iota}{_{C,i}},\tensor[]{\pi}{_{A,j}})(\tensor[]{\delta}{_{1}}
\oplus\tensor[]{\delta}{_{2}})) =
\begin{cases}
\SE(\delta_{i}) & \text{if } i=j\\
\tensor[_{\SF A}]{0}{_{\SF C}} & \text{if } i\neq j,
\end{cases}
\end{equation} 
where the first equality is by \cref{lem:functor-respects-morphisms-plays-well-with-domains-codomains-functions}(ii) and the second holds by \eqref{eqn:equations-for-rho-oplus-rho-prime} and the observations above. We next see that 
\begin{align*}
\BE'(\tensor[]{\iota}{_{\SF C,i}},\tensor[]{\pi}{_{\SF A,j}})(\eta) 
    &= \BE'(\SF\tensor[]{\iota}{_{C,i}}, \SF\tensor[]{\pi}{_{A,j}})\circ \BE'(\tensor[]{f}{_{C}}, \tensor*[]{f}{_{A}^{-1}}) (\eta) 
        && \text{using } \eqref{eqn:isomorphism-of-biproducts-and-inclusion-projection-maps}\\
    &= \BE'(\SF\tensor[]{\iota}{_{C,i}}, \SF\tensor[]{\pi}{_{A,j}})(\SE(\tensor[]{\delta}{_{1}}\oplus\tensor[]{\delta}{_{2}})) 
        && \text{using the definition of }\eta\\
    &= \begin{cases}
            \SE(\delta_{i}) & \text{if } i=j\\
            \tensor[_{\SF A}]{0}{_{\SF C}} & \text{if } i\neq j
        \end{cases}
        && \text{by }\eqref{eqn:equations-for-rho-oplus-rho-prime-after-applying-E}.
\end{align*}
Uniqueness hence yields
\(\BE'(\tensor[]{f}{_{C}},\tensor*[]{f}{_{A}^{-1}}\tensor[]{)}{^{-1}}(\SE(\tensor[]{\delta}{_{1}}\oplus\tensor[]{\delta}{_{2}}))
    = \eta 
    = \SE(\tensor[]{\delta}{_{1}})\oplus\SE(\tensor[]{\delta}{_{2}})\), which implies that
\begin{equation}
\label{eqn:description-SE-of-rho-oplus-rho-prime}
\SE(\tensor[]{\delta}{_{1}}\oplus\tensor[]{\delta}{_{2}}) = \BE'(\tensor[]{f}{_{C}},\tensor*[]{f}{_{A}^{-1}})(\SE(\tensor[]{\delta}{_{1}})\oplus\SE(\tensor[]{\delta}{_{2}})).
\end{equation}
The commutativity of \eqref{eqn:morphism-f-X-for-biproduct} combined with uniqueness statements from the universal properties of the diagonal \(\tensor{\Delta}{_{\SF C}}\) and the codiagonal \(\tensor[]{\nabla}{_{\SF A}}\) gives 
\begin{equation}
\label{eqn:Delta-nabla-compatibility}
\tensor{\Delta}{_{\SF C}} = \tensor[]{f}{_{C}}\circ \SF \tensor[]{\Delta}{_{C}} \hspace{1cm}\text{and}\hspace{1cm}
\tensor[]{\nabla}{_{\SF A}} = (\SF\tensor[]{\nabla}{_{A}})\circ \tensor*[]{f}{_{A}^{-1}}.
\end{equation}
Altogether, we conclude that 
\begin{align*}
\SE(\BE(\tensor[]{\Delta}{_{C}},\tensor[]{\nabla}{_{A}})(\tensor[]{\delta}{_{1}}\oplus\tensor[]{\delta}{_{2}})) 
    &= \BE'(\SF\tensor[]{\Delta}{_{C}},\SF\tensor[]{\nabla}{_{A}})(\SE(\tensor[]{\delta}{_{1}}\oplus\tensor[]{\delta}{_{2}})) 
        && \text{by \cref{lem:functor-respects-morphisms-plays-well-with-domains-codomains-functions}(ii)}\\
    &= \BE'(\SF\tensor[]{\Delta}{_{C}},\SF\tensor[]{\nabla}{_{A}})\BE'(\tensor[]{f}{_{C}},\tensor*[]{f}{_{A}^{-1}})(\SE(\tensor[]{\delta}{_{1}})\oplus\SE(\tensor[]{\delta}{_{2}}))
        && \text{by \eqref{eqn:description-SE-of-rho-oplus-rho-prime}}\\
    &= \BE'(\tensor{\Delta}{_{\SF C}},\tensor[]{\nabla}{_{\SF A}})(\SE(\tensor[]{\delta}{_{1}})\oplus\SE(\tensor[]{\delta}{_{2}}))
        && \text{by \eqref{eqn:Delta-nabla-compatibility}.}
\end{align*}
Thus, using the description in \eqref{eqn:internal-sum-of-BE-and-the-Baer-sum}, we see that 
\(\SE(\tensor[]{\delta}{_{1}} +\tensor[]{\delta}{_{2}})=\SE(\tensor[]{\delta}{_{1}}) + \SE(\tensor[]{\delta}{_{2}})\), as required.
\end{proof}

We are now ready to prove the main result of this subsection, 
characterising the existence of an additive functor between categories of extensions that respects morphisms. This is the key ingredient in the proof of 
\cref{thm:characterisation-of-n-exangulated-functors}. 
We remark that, by \cref{prop:F-additive-iff-E-additive-iff-E-exact}, the functors \(\SE\) appearing in the right-hand side of the statement of \cref{prop:functor-on-extensions} are exact functors of the form 
\((\BE\dExt{\CC},\CX_{\BE})\to(\BE'\dExt{\CC'},\CX_{\BE'})\). 
Moreover, we note that the assignment from left to right in \cref{prop:functor-on-extensions} has been proven independently by B{\o{}}rve--Trygsland; 
see \cite[Lem.\ 4.2]{Borve-Trygsland-factorization-extensions}. 
We include the argument for completeness.

\begin{prop}
\label{prop:functor-on-extensions}
For an additive functor \(\SF\colon\CC\to\CC'\), there is a one-to-one correspondence
\[
\begin{adjustbox}{scale=1,center}
$\displaystyle
\begin{aligned}[t]
    \Set{
        \begin{array}{c}
            \textit{natural transformations}\\
            \Gamma\colon\BE(-,-) \Rightarrow\BE'(\SF-,\SF-)
        \end{array}
    }
& \longleftrightarrow 
    \Set{
        \begin{array}{c}
            \textit{additive functors } \SE\colon\BE\dExt{\CC}\to\BE'\dExt{\CC'}\\
            \textit{that respect morphisms over } \SF
        \end{array}
    }\\[5pt]
    \Gamma 
    \hspace{6pt}  
    &
    \xmapsto{\hspace{15.5pt}}\hspace{6pt}
    \tensor[]{\SE}{_{(\SF,\Gamma)}}\\
\tensor[]{\Gamma}{_{(\SF,\SE)}}
\hspace{5pt}
    &\xmapsfrom{\hspace{15.5pt}}\hspace{6pt}
    \SE,
\end{aligned}$
\end{adjustbox}
\]
where 
\(\tensor[]{\SE}{_{(\SF,\Gamma)}}(\delta)\deff\tensor[]{\Gamma}{_{(C,A)}}(\delta)\) 
and 
\((\tensor[]{\Gamma}{_{(\SF,\SE)}}\tensor[]{)}{_{(C,A)}}(\delta)\deff\SE(\delta)\) 
for each \(\delta\inn\BE(C,A)\). 
\end{prop}

\begin{proof}
Suppose that 
\(\Gamma\colon\BE(-,-) \Rightarrow\BE'(\SF-,\SF-)\) is a natural transformation. We define \(\SE{_{(\SF,\Gamma)}}\) by setting \(\tensor[]{\SE}{_{(\SF,\Gamma)}}(\delta)\deff\tensor[]{\Gamma}{_{(C,A)}}(\delta)\) for objects \(\delta\inn\BE(C,A)\) and \(\tensor[]{\SE}{_{(\SF,\Gamma)}}(a,c)\coloneqq(\SF a,\SF c)\) for morphisms \((a,c)\colon\delta\to\rho\) in \(\BE\dExt{\CC}\). We need to show that this gives an additive functor \(\BE\dExt{\CC}\to\BE'\dExt{\CC'}\) which respects morphisms over \(\SF\).

Let us first check that \(\tensor[]{\SE}{_{(\SF,\Gamma)}}(a,c)=(\SF a,\SF c)\) is a morphism from \(\tensor[]{\SE}{_{(\SF,\Gamma)}}(\delta) = \tensor[]{\Gamma}{_{(C,A)}}(\delta)\) to \(\tensor[]{\SE}{_{(\SF,\Gamma)}}(\rho) = \tensor[]{\Gamma}{_{(D,B)}}(\rho)\) in \(\BE'\dExt{\CC'}\) whenever \((a,c)\) is a morphism from \(\delta \inn \BE(C,A)\) to \(\rho \inn \BE(D,B)\) in \(\BE\dExt{\CC}\). 
Note that by the naturality of \(\Gamma\), for any pair of morphisms \(a\colon A\to B\) and \(c\colon C\to D\) 
of 
\(\CC\), the 
diagram 
\[
\begin{tikzcd}[column sep=2.2cm]
\BE(C,A) \arrow{r}{\BE(C,a)} \arrow{d}[swap]{\tensor[]{\Gamma}{_{(C,A)}}} & 
\BE(C,B)\arrow{d}{\tensor[]{\Gamma}{_{(C,B)}}} & 
\BE(D,B)\arrow{d}{\tensor[]{\Gamma}{_{(D,B)}}}\arrow{l}[swap]{\BE(c,B)}\\
\BE'(\SF C,\SF A ) \arrow{r}{\BE'(\SF C,\SF a)}&
\BE'(\SF C,\SF B)&
\BE'(\SF D,\SF B) \arrow{l}[swap]{\BE'(\SF c,\SF B)}
\end{tikzcd}
\]
commutes.
Suppose furthermore that the pair \((a,c)\) defines a morphism \(\delta\to\rho\) in \(\BE\dExt{\CC}\), where \(\delta\inn\BE(C,A)\) and \(\rho\inn\BE(D,B)\). 
Assuming additionally \((a,c) \colon \delta \to \rho\) is a morphism in \(\BE\dExt{\CC}\), we have
\(\BE(C,a)(\delta)=\BE(c,B)(\rho)\). 
Applying \(\tensor[]{\Gamma}{_{(C,B)}}\) to this equality, the commutativity 
above gives
\[
\BE'(\SF C,\SF a)(\tensor[]{\Gamma}{_{(C,A)}}(\delta))=\BE'(\SF c,\SF B)(\tensor[]{\Gamma}{_{(D,B)}}(\rho)),
\]
so \ \((\SF a, \SF c)\colon\tensor[]{\Gamma}{_{(C,A)}}(\delta)\to\tensor[]{\Gamma}{_{(D,B)}}(\rho)\) is a morphism in \(\BE'\dExt{\CC'}\). 
It is clear that the assignment \(\SE{_{(\SF,\Gamma)}}\)
respects identity morphisms and composition, and thus defines a functor 
\(\BE\dExt{\CC}\to\BE'\dExt{\CC'}\).  
Notice that \(\tensor[]{\SE}{_{(\SF,\Gamma)}}\) respects morphisms over \(\SF\) by construction. 
The additivity of \(\tensor[]{\SE}{_{(\SF,\Gamma)}}\) follows from the additivity of \(\SF\) by  \cref{prop:F-additive-iff-E-additive-iff-E-exact}.

Conversely,
suppose we are given an additive functor 
\(
\SE\colon\BE\dExt{\CC}\to\BE'\dExt{\CC'}
\) that respects morphisms over \(\SF\). 
Note that 
by \cref{lem:functor-respects-morphisms-plays-well-with-domains-codomains-functions}(i), 
we have
\(\SE(\delta)\inn\BE'(\SF C,\SF A)\) whenever \(\delta\inn\BE(C,A)\). 
For each pair \(A,C\inn\CC\), we can hence write \((\tensor[]{\Gamma}{_{(\SF,\SE)}}\tensor[]{)}{_{(C,A)}}(\delta)\deff\SE(\delta)\) to define a function \((\tensor[]{\Gamma}{_{(\SF,\SE)}}\tensor[]{)}{_{(C,A)}}\colon \BE(C,A)\to\BE'(\SF C,\SF A)\). 
It follows from \cref{prop:respecting-extensions-preserves-additivity} that the functions \((\tensor[]{\Gamma}{_{(\SF,\SE)}}\tensor[]{)}{_{(C,A)}}\) 
are group homomorphisms. 
The diagram
\[
\begin{tikzcd}[column sep=2.2cm, row sep=1.1cm]
 \BE(Z,A)\arrow{d}[swap]{(\tensor[]{\Gamma}{_{(\SF,\SE)}}\tensor[]{)}{_{(Z,A)}}} & \BE(C,A)\arrow{l}[swap]{\BE(z,A)} \arrow{r}{\BE(C,x)} \arrow{d}{(\tensor[]{\Gamma}{_{(\SF,\SE)}}\tensor[]{)}{_{(C,A)}}} & 
\BE(C,X)\arrow{d}{(\tensor[]{\Gamma}{_{(\SF,\SE)}}\tensor[]{)}{_{(C,X)}}}
\\
\BE'(\SF Z,\SF A ) &
\BE'(\SF C,\SF A)\arrow{r}{\BE'(\SF C,\SF x)}\arrow{l}[swap]{\BE'(\SF z,\SF A)}&
\BE'(\SF C,\SF X)
\end{tikzcd}
\]
commutes for any pair of morphisms \(x\colon A\to X\) and \(z\colon Z\to C\) 
of  
\(\CC\) 
by \cref{lem:functor-respects-morphisms-plays-well-with-domains-codomains-functions}(ii). 
Commutativity of diagrams of the above form imply the naturality of the transformation 
\(
\tensor[]{\Gamma}{_{(\SF,\SE)}}
    = \{ 
        (\tensor[]{\Gamma}{_{(\SF,\SE)}}
      \tensor[]{)}{_{(C,A)}}\tensor[]{\}}{_{(C,A)\inn\CC^{\scalebox{0.7}{\op}}\times\CC}}
        \colon 
    \BE(-,-) \Longrightarrow \BE'(\SF-,\SF-).
\)

By the arguments above, we see that the assignments \(\Gamma \mapsto \tensor[]{\SE}{_{(\SF,\Gamma)}}\) and \(\SE \mapsto \tensor[]{\Gamma}{_{(\SF,\SE)}}\) from the statement of the proposition are well-defined. 
It is straightforward to check that they are mutually inverse, and hence define a one-to-one correspondence.
\end{proof}

%%%%%%%%%%%%%%%%%%%%%%%%%%%%%%%%%%%%%%%%%%%%%%%%%%%%%%%%%%
%%%%%%%%%%%%%%%%%%%%%%%%%%%%%%%%%%%%%%%%%%%%%%%%%%%%%%%%%%

\subsection{A characterisation of \texorpdfstring{\(n\)}{n}-exangulated functors}
\label{subsection:3.3}

In this subsection we first recall the definition of an \(n\)-exangulated functor as introduced in \cite{Bennett-TennenhausShah-transport-of-structure-in-higher-homological-algebra}. This notion captures what it means for a functor between \(n\)-exangulated categories to be structure-preserving. 
We then prove the main result of \cref{section:3}, namely \cref{thm:characterisation-of-n-exangulated-functors}, which gives a characterisation of \(n\)-exangulated functors in terms of functors on the associated categories of extensions. We conclude with a lemma pertaining to the composition of these kinds of functors, which will allow us to take a 2-categorical perspective on \(n\)-exangulated categories in \cref{section:4}.

Recall that given an additive functor \(\SF\colon\CC\to\CC'\) between additive categories, there is an induced functor  \(\SF_{\com}\colon\com_{\CC}\to\com_{\CC'}\) between the corresponding categories of complexes. 
For 
\(X^{\combul}\inn\com_{\CC}\), the object \(\SF_{\com}X^{\combul}\inn\com_{\CC'}\) has \((\SF_{\com}X^{\combul})^{i}=\SF (X^{i})\) in degree \(i\inn\BZ\). The differential of \(\SF_{\com}X^{\combul}\) is given by \(d_{\SF_{\com} X }^{i}=\SF (d_{X}^{i})\), where \(d_X\) denotes the differential of \(X^{\combul}\). 
For a morphism \(f^{\combul}\) in \(\tensor[]{\com}{_{\CC}}\), 
we have 
\(\SF_{\com}f^{\combul} = (\ldots,\SF (f^{i-1}),\SF (f^{i}),\SF (f^{i+1}),\ldots)\). 
For the remainder of this section, let \((\CC,\BE,\fs)\) and \((\CC',\BE',\fs')\) 
be \(n\)-exangulated categories.

\begin{defn}
\label{def:n-exangulated-functor}(See \cite[Def.~2.32]{Bennett-TennenhausShah-transport-of-structure-in-higher-homological-algebra}.) 
Let \(\SF\colon \CC\to\CC'\) be an additive functor and let 
\[
\Gamma=\{\tensor[]{\Gamma}{_{(C,A)}}\tensor[]{\}}{_{(C,A)\inn\CC^{\scalebox{0.7}{\op}}\times\CC}}\colon \BE(-,-) \Longrightarrow \BE'(\SF^{\op}-,\SF-)
\]
be a natural transformation. We call the pair \((\SF,\Gamma) \colon (\CC,\BE,\fs) \to (\CC',\BE',\fs')\) an \emph{\(n\)-exangulated functor} if,
for all \(X^{0},X^{n+1}\inn\CC\) and each \(\delta\inn\BE(X^{n+1},X^{0})\), we have that 
\(\fs(\delta)=[X^{\combul}]\) implies 
\(\fs'(\tensor[]{\Gamma}{_{(X^{n+1},X^{0})}}(\delta))=[\SF_{\com} X^{\combul}]\). 
\end{defn}

A similar structure-preservation condition exists for functors on categories of extensions.

\begin{defn}
\label{def:SE-respects-distinguished-n-exangle-over-SF}
Suppose \(\SF\colon \CC \to \CC'\) is an additive functor. We say that a functor 
\(\SE\colon\BE\dExt{\CC}\to\BE'\dExt{\CC'}\) \emph{respects distinguished \(n\)-exangles over \(\SF\)} if 
\(\fs(\delta)=[X^{\combul}]\) implies \(\fs'(\SE(\delta))=[\SF_{\com} X^{\combul}]\).
\end{defn}

We are now ready to prove \cref{thmx:characterisation-of-n-exangulated-functors} from \cref{sec:introduction}. 
Again, by \cref{prop:F-additive-iff-E-additive-iff-E-exact}, each \(\SE\) in the statement below is an exact functor 
\((\BE\dExt{\CC},\CX_{\BE})\to(\BE'\dExt{\CC'},\CX_{\BE'})\).

\begin{thm}
\label{thm:characterisation-of-n-exangulated-functors} 
There is a one-to-one correspondence
\[
\begin{adjustbox}{
scale=0.93,
center}
$\displaystyle
\begin{aligned}[t]
    \Set{
        \begin{array}{c}
            n\textit{-exangulated functors}\\
            (\SF,\Gamma)\colon (\CC,\BE,\fs) \to (\CC',\BE',\fs')
        \end{array}
    }
& \longleftrightarrow 
    \Set{
        \begin{array}{c}
            \textit{pairs } (\SF,\SE) \textit{ of additive functors }\SF \colon \CC \to \CC'\\\textit{and } \SE\colon\BE\dExt{\CC}\to\BE'\dExt{\CC'}\textit{, where } \SE \textit{ respects}\\\textit{ morphisms and distinguished }
            n\textit{-exangles over } \SF
        \end{array}
    }\\[5pt]
(\SF,\Gamma) 
\hspace{6pt}
    &
    \xmapsto{\hspace{15.5pt}}\hspace{6pt}
    (\SF,\tensor[]{\SE}{_{(\SF,\Gamma)}})\\
(\SF,\tensor[]{\Gamma}{_{(\SF,\SE)}})
\hspace{6pt}
    &
    \xmapsfrom{\hspace{15.5pt}}\hspace{6pt}
    (\SF,\SE),
\end{aligned}$
\end{adjustbox}
\]
where 
\(\tensor[]{\SE}{_{(\SF,\Gamma)}}(\delta)\deff\tensor[]{\Gamma}{_{(C,A)}}(\delta)\) 
and 
\((\tensor[]{\Gamma}{_{(\SF,\SE)}}\tensor[]{)}{_{(C,A)}}(\delta)\deff\SE(\delta)\) 
for \(\delta\inn\BE(C,A)\). 
\end{thm}

\begin{proof}
Suppose first that 
\((\SF,\Gamma)\colon (\CC,\BE,\fs) \to (\CC',\BE',\fs')\) is an \(n\)-exangulated functor. The functor \(\SF \colon \CC \to \CC'\) is hence additive, and \cref{prop:functor-on-extensions} yields that \(\tensor[]{\SE}{_{(\SF,\Gamma)}}\) as defined above is an additive functor \(\BE\dExt{\CC}\to\BE'\dExt{\CC'}\) that respects morphisms over \(\SF\). As \((\SF,\Gamma)\) is \(n\)-exangulated, we have that \(\fs(\delta)=[X^{\combul}]\) implies \(\fs'(\tensor[]{\SE}{_{(\SF,\Gamma)}}(\delta))=\fs'(\tensor[]{\Gamma}{_{(X^{n+1},X^0)}}(\delta))=[\SF_{\com} X^{\combul}]\), so \(\tensor[]{\SE}{_{(\SF,\Gamma)}}\) respects distinguished \(n\)-exangles over \(\SF\).

On the other hand, consider a pair \((\SF,\SE)\) 
from 
the right-hand side of the claimed correspondence. 
\cref{prop:functor-on-extensions} yields that \(\tensor[]{\Gamma}{_{(\SF,\SE)}}\) as defined above is a natural transformation \(\BE(-,-) \Rightarrow\BE'(\SF-,\SF-)\). 
If  
\(\fs(\delta)=[X^{\combul}]\), 
then we have \mbox{\(\fs'((\tensor[]{\Gamma}{_{(\SF,\SE)}}\tensor[]{)}{_{(X^{n+1},X^0)}}(\delta))=[\SF_{\com} X^{\combul}]\)} by assumption, so \((\SF,\tensor[]{\Gamma}{_{(\SF,\SE)}}) \colon (\CC,\BE,\fs) \to (\CC',\BE',\fs')\) is an \(n\)-exangulated functor.

Consequently, the assignments \((\SF,\Gamma) \mapsto (\SF,\tensor[]{\SE}{_{(\SF,\Gamma)}})\) and \((\SF,\SE) \mapsto (\SF,\tensor[]{\Gamma}{_{(\SF,\SE)}})\) from the statement of theorem are well-defined. It is straightforward to check that these two assignments are mutually inverse, and hence define a one-to-one correspondence.
\end{proof}

Recall that \cref{corx:criterion-for-n-exangulated-functor} of \cref{sec:introduction} interprets, in terms of functors between categories of extensions, what it means for an additive functor between \(n\)-exangulated categories to be structure-preserving. 
This corollary is an immediate consequence of \cref{thm:characterisation-of-n-exangulated-functors}.

For the remainder of this section, suppose also that \((\CC'',\BE'',\fs'')\) is an \(n\)-exangulated category and that \((\SF,\Gamma) \colon (\CC,\BE,\fs) \to (\CC',\BE',\fs')\) and 
\((\SL,\Phi) \colon (\CC',\BE',\fs') \to (\CC'',\BE'',\fs'')\) are \mbox{\(n\)-exangulated} functors. 
There is then a natural transformation
\[
\tensor[]{\Phi}{_{\SF\times\SF}}
    = 
    \{\Phi_{(\SF C,\SF A)}\tensor[]{\}}{_{(C,A)\inn\CC^{\scalebox{0.7}{\op}}\times\CC}}
        \colon 
    \BE'(\SF-,\SF-) 
        \Longrightarrow
    \BE''(\SL\SF-,\SL\SF-).
\] 
This is known as the \emph{whiskering of \(\SF\times\SF\) and \(\Phi\)}. 
We also use whiskerings in \cref{section:4}; see \Cref{notn:whiskers-and-identity-nat-trans} and onward.

Whiskerings enable us to define the composition of \(n\)-exangulated functors. This is a higher analogue of the composition of extriangulated functors from \cite[Def.~2.11(2)]{NakaokaOgawaSakai-localization-of-extriangulated-categories}.

\begin{defn}
\label{def:composing-n-exangulated-functors}
\begin{enumerate}[label=\textup{(\roman*)}]

    \item 
    The \emph{identity \(n\)-exangulated functor} of \((\CC,\BE,\fs)\) is the pair \((\idfunc{\CC},\iden{\BE})\).
    
    \item The \emph{composite} of \((\SF,\Gamma)\) and \((\SL,\Phi)\) is \((\SL,\Phi) \circ (\SF,\Gamma)\deff(\SL \circ \SF, \tensor[]{\Phi}{_{\SF\times\SF}} \circ\Gamma)\). 
    \end{enumerate}
\end{defn}

We conclude the section by justifying our terminology in \cref{def:composing-n-exangulated-functors} and showing 
that the left-to-right assignment in \cref{thm:characterisation-of-n-exangulated-functors} is compatible with identity and composition of \(n\)-exangulated functors.
\begin{lem}
\label{lem:identity-composition-n-exangulated-functors-and-extension-functors-respect-compositions}
The following statements hold. 
\begin{enumerate}[label=\textup{(\roman*)}]
    \item\label{item:identity-is-n-exangulated}
        The pair \((\idfunc{\CC},\iden{\BE})\) is an \(n\)-exangulated functor \((\CC,\BE,\fs) \to (\CC,\BE,\fs)\).
    \item\label{item:1-cell-composite-is-n-exangulated} 
        The composite 
\((\SL,\Phi) \circ (\SF,\Gamma)\colon (\CC,\BE,\fs) \to (\CC'',\BE'',\fs'')\)
is \(n\)-exangulated. This composition 
is associative and unital with respect to identity \(n\)-exangulated functors. 

    \item\label{item:induced-scriptE-functors-respect-identity-and-composition} 
        There are equalities 
        \(\tensor[]{\SE}{_{(\iden{\CC},\idfunc{\BE})}}=\iden{\BE\dExt{\CC}}\) 
        and 
        \(\tensor[]{\SE}{_{(\SL,\Phi)\circ(\SF,\Gamma)}}
            = \tensor[]{\SE}{_{(\SL,\Phi)}}\circ \tensor[]{\SE}{_{(\SF,\Gamma)}}\).
\end{enumerate}
\end{lem}

\begin{proof}
Checking \ref{item:identity-is-n-exangulated}, \ref{item:1-cell-composite-is-n-exangulated} and the first part of \ref{item:induced-scriptE-functors-respect-identity-and-composition} is straightforward. 
For the second claim of \ref{item:induced-scriptE-functors-respect-identity-and-composition}, 
note that
\(
\tensor[]{\SE}{_{(\SL,\Phi)\circ(\SF,\Gamma)}}
    = \tensor[]{\SE}{_{(\SL\SF,\tensor[]{\Phi}{_{\SF\times\SF}}\Gamma)}}
\) 
and 
\(
\tensor[]{\SE}{_{(\SL,\Phi)}} \circ \tensor[]{\SE}{_{(\SF,\Gamma)}}
\) 
agree on objects of \(\BE\dExt{\CC}\). 
By \cref{thm:characterisation-of-n-exangulated-functors}, the functors \(\tensor[]{\SE}{_{(\SF,\Gamma)}}\), \(\tensor[]{\SE}{_{(\SL,\Phi)}}\) and \(
\tensor[]{\SE}{_{(\SL\SF,\tensor[]{\Phi}{_{\SF\times\SF}}\Gamma)}}
\) respect morphisms over \(\SF\), \(\SL\) and \(\SL\SF\), respectively. 
Since 
\(
\tensor[]{\SE}{_{(\SL,\Phi)}}(\tensor[]{\SE}{_{(\SF,\Gamma)}}(a,c))
    = (\SL \SF a, \SL \SF c)
    = \tensor[]{\SE}{_{(\SL\SF,\tensor[]{\Phi}{_{\SF\times\SF}}\Gamma)}}(a,c)
\)
for each morphism \((a,c)\) in \(\BE\dExt{\CC}\), it follows from \cref{rem:equality-of-morphisms-of-extensions} 
that \(\tensor[]{\SE}{_{(\SL\SF,\tensor[]{\Phi}{_{\SF\times\SF}}\Gamma)}}\) and \(\tensor[]{\SE}{_{(\SL,\Phi)}}\circ \tensor[]{\SE}{_{(\SF,\Gamma)}}\) also agree on morphisms of \(\BE\dExt{\CC}\), and hence are equal as functors.
\end{proof}

%%%%%%%%%%%%%%%%%%%%%%%%%%%%%%%%%%%%%%%%%%%%%%%%%%%
%%%%%%%%%%%%%%%%%%%%%%%%%%%%%%%%%%%%%%%%%%%%%%%%%%%

\section{A \texorpdfstring{\(2\)}{2}-categorical perspective on \texorpdfstring{\(n\)}{n}-exangulated categories}
\label{section:4}

The authors are very grateful to Hiroyuki Nakaoka for encouraging them to think about morphisms between \(n\)-exangulated functors, which prompted the results in this section. 
In particular, we introduce the notion of \(n\)-exangulated natural transformations, which recovers \cite[Def.~2.11(3)]{NakaokaOgawaSakai-localization-of-extriangulated-categories} in the case \(n=1\). 
This enables us to make considerations that are \mbox{\(2\)-category}-theoretic in the sense of \cite[Sec.\ XII.3]{MacLane-categories-for-the-working-mathematician}. 
Some definitions have been developed independently in He--He--Zhou \cite{HeHeZhou-localization-of-n-exangulated-categories} and in Enomoto--Saito \cite{enomoto2022grothendieck}. 
The \(2\)-category of small abelian categories has been studied before; see, for example, work of Prest--Rajani \cite{Prest-Rajani-Structure-sheaves}.

A \emph{\(2\)-category} consists of \(0\)-cells, \(1\)-cells and \(2\)-cells, which should be thought of as objects, morphisms between objects, 
and morphisms between morphisms, respectively, 
satisfying some axioms; 
see e.g.\ \cite[p.~273]{MacLane-categories-for-the-working-mathematician}. 
We show that the category \(\exang{n}\) of small \(n\)-exangulated categories is a \(2\)-category, 
with \(n\)-exangulated functors as \(1\)-cells 
and 
\(n\)-exangulated natural transformations as \(2\)-cells; 
see \cref{cor:2-cats}. Recall that a category is said to be \emph{small} if the class of objects and the class of morphisms are sets. 
More generally, we prove that similar properties hold for the category \(\Exang{n}\) of all \(n\)-exangulated categories; see \cref{prop:hom-category-of-n-exangulated-categories}.

We start this section by giving the definition of \(n\)-exangulated natural transformations, before considering their compositions in Subsection~\ref{subsec:composing-n-exangulated-natural-transformations}. 
Having established a notion of morphisms between \(n\)-exangulated functors, 
we 
will be in position 
to introduce and study \mbox{\(n\)-exangulated} adjoints and equivalences 
in Subsection~\ref{subsec:n-exangulated-adjoints}. 
In Subsection~\ref{subsec:2-categorical-viewpoint} we continue our \(2\)-categorical approach, 
leading to the construction of a \(2\)-functor \(\updave \colon \exang{n} \to \exactcat\) to the \(2\)-category of small  exact categories; see \cref{cor:finalcorollary}, which 
yields 
\cref{thmx:finalcorollary} from \cref{sec:introduction}. 
The proof of this statement goes via a more general result, namely \cref{thm:2-functor}, where we establish a functor \(\updave \colon \Exang{n} \to \Exactcat\) with similar properties, 
but without 
any smallness assumptions. 
If one 
ignores the set-theoretic issue described in \cref{rem:not-2-categories}, one can interpret our work in Subsection~\ref{subsec:2-categorical-viewpoint} as constructing a \(2\)-functor \(\updave\) from the category \(\Exang{n}\) to the category \(\Exactcat\) of all exact categories. 
A fundamental step in defining 
\(\updave \colon \Exang{n} \to \Exactcat\) is the characterisation of \(n\)-exangulated natural transformations given in \cref{thm:characterisation-of-n-exangulated-natural-transformations}. 
A full definition of \(\updave\) is given in \cref{def:updave}.

For \(n\)-exangulated natural transformations we use Hebrew letters \(\beth\) (beth) and \(\daleth\) (daleth).

\begin{defn}
\label{def:n-exangulated-natural-transformation}
Let \((\SF,\Gamma), (\SG,\Lambda)\colon(\CC,\BE,\fs)\to(\CC',\BE',\fs')\) be \(n\)-exangulated functors and 
\(\beth\colon\SF\Rightarrow\SG\) 
a natural transformation. 
We call  \(\beth\colon(\SF,\Gamma)\Rightarrow(\SG,\Lambda)\) an \emph{\(n\)-exangulated natural transformation} if, for all \(A,C\inn\CC\) and each 
\(\delta\inn\BE(C,A)\), 
the pair \((\tensor[]{\beth}{_{A}},\tensor[]{\beth}{_{C}})\) is a morphism \(\tensor[]{\Gamma}{_{(C,A)}}(\delta) \to \tensor[]{\Lambda}{_{(C,A)}}(\delta)\) in \(\BE'\dExt{\CC'}\), 
i.e.\ 
\begin{equation}
\label{eqn:n-exangulated-natural-transformation-property}
    (\tensor*[]{\beth}{_{A}}\tensor*[]{)}{_{\BE'}}\tensor[]{\Gamma}{_{(C,A)}}(\delta)
    = (\tensor[]{\beth}{_{C}}\tensor[]{)}{^{\BE'}}\tensor[]{\Lambda}{_{(C,A)}}(\delta).
\end{equation}
\end{defn}

See Examples~\ref{example:extriangulated-is-1-exangulated}, 
\ref{example:n+2-angulated-category-is-n-exangulated} and 
\ref{example:n-exact-category-is-n-exangulated} 
for discussions on the notion of an \(n\)-exangulated natural transformation in some familiar settings.

\begin{setup}
\label{setup:section4} 
For the remainder of this section, we use the standing assumptions and notation as indicated in the diagram below. Fix \(n\geq 1\). 
We assume that the categories \((\CC,\BE,\fs)\), \((\CC',\BE',\fs')\) and \((\CC'',\BE'',\fs'')\) are \(n\)-exangulated. 
The seven functors between these categories, drawn horizontally, and the four natural transformations, draw vertically, are assumed to be \(n\)-exangulated functors and \(n\)-exangulated natural transformations, respectively. 
\[
\begin{tikzcd}
&{}
    \arrow[Rightarrow,
            shorten <= 6pt,
            shorten >= 6pt
            ]{d}
                [xshift=1pt, yshift=-1pt]{\beth}
&
&{}
    \arrow[Rightarrow,
            shorten <= 11.5pt,
            shorten >= 6pt,
            xshift=1pt
            ]{d}
                [xshift=1pt, yshift=-1pt]{\daleth}
&
\\
(\CC,\BE,\fs)
    \arrow[bend left=40]{rr}
            [description, xshift=-1.5pt, yshift=-1pt]{(\SF,\Gamma)}
    \arrow{rr}
            [description]{(\SG,\Lambda)}
    \arrow[bend right=40]{rr}
            [description, xshift=-1.5pt, yshift=1pt]{(\SH,\Theta)}
&{}
    \arrow[Rightarrow,
            shorten <= 6pt,
            shorten >= 6pt,
            yshift=5pt
            ]{d}
                [xshift=1pt, yshift=3pt]{\beth'}
&(\CC',\BE',\fs')
    \arrow[bend left=40]{rr}
            [description, xshift=-2pt, yshift=-2pt]{(\SL,\Phi)}
    \arrow{rr}
            [description]{(\SM,\Psi)}
    \arrow[bend right=40]{rr}
            [description, xshift=-3pt, yshift=2pt]{(\SN,\Omega)}
    \arrow[bend left=80, 
            shorten <= -1cm,
            shorten >= -1cm,
            shift left=1cm, 
            start anchor={[xshift=-3pt]},
            end anchor={[xshift=3pt]}
            ]{ll}
                [description, yshift=9pt]{(\SA,\NT)}
&{}
    \arrow[Rightarrow,
            shorten <= 6pt,
            shorten >= 6pt,
            xshift=1pt, 
            yshift=5pt
            ]{d}
                [xshift=1pt, yshift=3pt]{\daleth'}
&(\CC'',\BE'',\fs'')
\\
&{}&&{}&
\end{tikzcd}
\]
\end{setup}

%%%%%%%%%%%%%%%%%%%%%%%%%%%%%%%%%%%%%%%%%%%%%%%%%%%
%%%%%%%%%%%%%%%%%%%%%%%%%%%%%%%%%%%%%%%%%%%%%%%%%%%

\subsection{Composing \texorpdfstring{\(n\)}{n}-exangulated natural transformations}
\label{subsec:composing-n-exangulated-natural-transformations}

In a \(2\)-category, one should be able to compose \(2\)-cells in two ways that are associative and unital \cite[p.~273]{MacLane-categories-for-the-working-mathematician}. 
Our aim in this section is hence to consider two notions of composition of \(n\)-exangulated natural transformations. 
These are defined by using the classical notions of vertical and horizontal compositions, which apply to natural transformations in general \cite[pp.~40, 42]{MacLane-categories-for-the-working-mathematician}.

\begin{defn}
\label{notn:vertical-composition-and-horizontal-composition-of-2-cells} 
\begin{enumerate}[label=\textup{(\roman*)}]
    \item\label{item:identity-n-exangulated-transformation} 
    We define the \emph{identity} 
    \(\iden{(\SF, \Gamma)}\)
    of \((\SF,\Gamma)\) to be 
    \(
        \iden{\SF}
        \colon 
        \SF\Rightarrow\SF
    \).
    
    \item The \emph{vertical composition} \(\beth'\circ_{v}\beth\) 
    is given by  
    \(
    (\beth'\circ_{v}\beth\tensor[]{)}{_{X}}\deff\beth'_{X}\tensor[]{\beth}{_{X}}
    \) 
    for each \(X\inn\CC\).
    
    \item The $n$-exangulated natural transformation \(\beth\) is said to be
an \emph{\(n\)-exangulated natural isomorphism} if it has an \(n\)-exangulated inverse under vertical composition. 
Note that this is equivalent to 
\(\tensor[]{\beth}{_{X}}\) being an isomorphism for every \(X\inn\CC\).

    \item 
    The \emph{horizontal composition} \(\daleth\circ_{h}\beth\) 
    is given by 
    \(
    (\daleth\circ_{h}\beth\tensor[]{)}{_{X}}\deff \tensor[]{\daleth}{_{\SG X}}(\SL \tensor[]{\beth}{_{X}})
    \) 
    for each \(X\inn\CC\).
\end{enumerate} 
\end{defn}

It follows from classical theory that the vertical and horizontal compositions of \(n\)-exan\-gu\-lat\-ed natural transformations are again natural \cite[pp.~40, 42--43]{MacLane-categories-for-the-working-mathematician}. 
However, it is not clear that these compositions are \(n\)-exangulated. 
This is checked in \cref{prop:compositions-are-n-exangulated}.

First, we verify associativity and unitality, as well as a useful commutativity property known as the \emph{interchange law} (or \emph{middle-four exchange}).

\begin{lem}
\label{prop:identity-is-n-exangulated-associativity-unitality-and-middle-four-exchange}
The following statements hold.
\begin{enumerate}[label=\textup{(\roman*)}]
    \item\label{item:identity-n-exangulated-transformation-is-n-exangulated} 
        The identity 
        \(\iden{(\SF,\Gamma)}\colon (\SF,\Gamma)\Rightarrow(\SF,\Gamma)\) of \((\SF,\Gamma)\) 
        is \(n\)-exangulated.
    
    \item\label{item:assoc-unital-n-exangulated-nat-trans} 
        Both \(\circ_{v}\) and \(\circ_{h}\) are associative and unital on \(n\)-exangulated natural transformations. 
    
    \item\label{item:middle-4-exchange}
        The interchange law 
        \(
        (\daleth' \circ_{h}\beth' )\circ_{v} (\daleth \circ_{h} \beth)=(\daleth' \circ_{v}\daleth)\circ_{h}(\beth' \circ_{v} \beth)
        \) 
        holds.    
\end{enumerate}
\end{lem}

\begin{proof}
Equation 
\eqref{eqn:n-exangulated-natural-transformation-property} 
is trivially satisfied 
when \(\Lambda=\Gamma\) and \(\beth=\idfunc{\SF}\), which yields \ref{item:identity-n-exangulated-transformation-is-n-exangulated}. 
The claims of \ref{item:assoc-unital-n-exangulated-nat-trans} and \ref{item:middle-4-exchange} 
hold for natural transformations in general; 
see \cite[pp.~40, 43]{MacLane-categories-for-the-working-mathematician}. 
\end{proof}

To show that compositions of \(n\)-exangulated natural transformations are again \(n\)-exan\-gu\-lat\-ed, 
we use 
whiskering \cite[p.~275]{MacLane-categories-for-the-working-mathematician}, 
which is 
a special case of horizontal composition.

\begin{notn}
\label{notn:whiskers-and-identity-nat-trans}
\begin{enumerate}[label=\textup{(\roman*)}]

    \item\label{item:whisker-G-and-daleth} 
    The \emph{whiskering} \(\tensor[]{\daleth}{_{\SG}} \colon \SL\SG \Rightarrow \SM\SG\) \emph{of} \(\SG\) \emph{and} \(\daleth\) is the natural transformation given by \((\tensor[]{\daleth}{_{\SG}}\tensor[]{)}{_{X}}\deff\tensor[]{\daleth}{_{\SG X}}\) 
    for \(X\inn \CC\).
    
    \item\label{item:whisker-daleth-and-X} 
    The \emph{whiskering} 
    \(
    \SL\beth \colon \SL\SF \Rightarrow \SL\SG
    \) \emph{of} \(\beth\)
    \emph{and} \(\SL\) is the natural transformation given by \((\SL\beth\tensor[]{)}{_{X}}\deff\SL\tensor[]{\beth}{_{X}}\) for 
    \(X\inn\CC\).
\end{enumerate}
\end{notn}

\begin{rem}
\label{rem:whiskering-combined-with-vertical-equivalent-to-horizontal}
One 
can view horizontal composition as the vertical composition of some whiskerings \cite[p.~43]{MacLane-categories-for-the-working-mathematician}. 
Note that 
\(
\tensor[]{\daleth}{_{\SG}}=\daleth\circ_{h}\idfunc{\SG}
\) 
and 
\(
\SL\beth = \idfunc{\SL}\circ_{h}\beth
\). 
Hence, unitality of vertical composition combined with the interchange law (see \ref{item:assoc-unital-n-exangulated-nat-trans} and \ref{item:middle-4-exchange} of  \cref{prop:identity-is-n-exangulated-associativity-unitality-and-middle-four-exchange}) yields 
\begin{equation}
\label{eqn:horizontal-comp-is-vertical-comp-of-whiskers}
\daleth\circ_{h}\beth 
    = (\daleth \circ_{v}\idfunc{\SL})\circ_{h}(\idfunc{\SG} \circ_{v} \beth) 
    = (\daleth \circ_{h}\idfunc{\SG} )\circ_{v} (\idfunc{\SL} \circ_{h} \beth)
    = \tensor[]{\daleth}{_{\SG}}\circ_{v}\SL \beth.
\end{equation}
\end{rem}

Recall from \cref{thm:characterisation-of-n-exangulated-functors} the exact functor 
\(
\tensor[]{\SE}{_{(\SF,\Gamma)}}
        \colon
    (\BE\dExt{\CC},\CX_{\BE})
        \to
    (\BE'\dExt{\CC'},\CX_{\BE'})
\) given by \(\tensor[]{\SE}{_{(\SF,\Gamma)}}(\delta) = \tensor[]{\Gamma}{_{(C,A)}}(\delta)\) on objects \(\delta\inn\BE(C,A)\), and recall also that \(\tensor[]{\SE}{_{(\SF,\Gamma)}}\) respects morphisms and distinguished \(n\)-exangles over \(\SF\).

\begin{prop}
\label{prop:whiskerings-are-n-exangulated}
The following statements hold.
\begin{enumerate}[label=\textup{(\roman*)}]

    \item\label{item:whisker-G-and-daleth-is-n-exangulated} 
    The whiskering 
    \(
        \tensor[]{\daleth}{_{\SG}}\colon (\SL,\Phi)\circ(\SG,\Lambda)\Rightarrow(\SM,\Psi)\circ(\SG,\Lambda)
    \) 
    is \(n\)-exangulated.
    
    \item\label{item:whisker-daleth-and-X-is-n-exangulated} 
    The whiskering 
    \(
        \SL\beth\colon (\SL,\Phi)\circ(\SF,\Gamma)\Rightarrow(\SL,\Phi)\circ(\SG,\Lambda)
    \) 
    is 
    \(n\)-exangulated. 
\end{enumerate}
\end{prop}

\begin{proof}
\ref{item:whisker-G-and-daleth-is-n-exangulated}\; 
By \cref{lem:identity-composition-n-exangulated-functors-and-extension-functors-respect-compositions}\ref{item:1-cell-composite-is-n-exangulated}, we have that the composites 
\(
(\SL,\Phi)\circ(\SG,\Lambda)=(\SL\SG,\tensor[]{\Phi}{_{\SG\times\SG}}\Lambda)
\) 
and 
\(
(\SM,\Psi)\circ(\SG,\Lambda)=(\SM\SG,\tensor[]{\Psi}{_{\SG\times\SG}}\Lambda)
\)
are \(n\)-exangulated functors from \((\CC,\BE,\fs)\) to \((\CC'',\BE'',\fs'')\). 
Given an extension \(\delta\inn\BE(C,A)\), we have 
\(\tensor[]{\Lambda}{_{(C,A)}}(\delta)\inn\BE'(\SG C, \SG A)\). 
As
\(
\daleth \colon (\SL,\Phi)\Rightarrow(\SM,\Psi)
\) 
is an \(n\)-exangulated natural transformation, we thus obtain
\[
(\tensor[]{\daleth}{_{\SG A}}
\tensor[]{)}{_{\BE''}}\tensor[]{\Phi}{_{(\SG C, \SG A)}}(\tensor[]{\Lambda}{_{(C,A)}}(\delta)) 
    =   (\tensor[]{\daleth}{_{\SG C}}
        \tensor[]{)}{^{\BE''}}
        \tensor[]{\Psi}{_{(\SG C, \SG A)}}(
        \tensor[]{\Lambda}{_{(C,A)}}(\delta)
        ),
\]
which verifies that \(\tensor[]{\daleth}{_{\SG}}\) is an \(n\)-exangulated natural transformation.

\ref{item:whisker-daleth-and-X-is-n-exangulated}\; 
The pairs \(
(\SL,\Phi)\circ(\SF,\Gamma)=(\SL\SF,\tensor[]{\Phi}{_{\SF\times\SF}}\Gamma)
\) 
and
\(
(\SL,\Phi)\circ(\SG,\Lambda)=(\SL\SG,\tensor[]{\Phi}{_{\SG\times\SG}}\Lambda)
\) 
are \(n\)-exangulated functors from 
\((\CC,\BE,\fs)\) to 
\((\CC'',\BE'',\fs'')\) 
by \cref{lem:identity-composition-n-exangulated-functors-and-extension-functors-respect-compositions}\ref{item:1-cell-composite-is-n-exangulated}. 
Consider the functor 
\(
\tensor[]{\SE}{_{(\SL,\Phi)}}
\colon\BE'\dExt{\CC'}\to\BE''\dExt{\CC''}
\) from  \cref{thm:characterisation-of-n-exangulated-functors}. 
For \(\delta\inn\BE(C,A)\), we have
\begin{align*}
    (\tensor*[]{\SL\beth}{_{A}}\tensor*[]{)}{_{\BE''}}\tensor*[]{\Phi}{_{(\SF C,\SF A)}}(\Gamma_{(C,A)}(\delta)) 
    &= (\tensor*[]{\SL\beth}{_{A}}\tensor*[]{)}{_{\BE''}}\tensor[]{\SE}{_{(\SL,\Phi)}}(\Gamma_{(C,A)}(\delta))
        &&\text{using the definition of \(\tensor[]{\SE}{_{(\SL,\Phi)}}\)}\\
    &= \tensor[]{\SE}{_{(\SL,\Phi)}} ((\tensor[]{\beth}{_{A}}\tensor[]{)}{_{\BE'}}\tensor[]{\Gamma}{_{(C,A)}}(\delta))&&\text{by \cref{lem:functor-respects-morphisms-plays-well-with-domains-codomains-functions}(ii)}\\
    &= \tensor[]{\SE}{_{(\SL,\Phi)}} ((\tensor[]{\beth}{_{C}}\tensor[]{)}{^{\BE'}}\tensor[]{\Lambda}{_{(C,A)}}(\delta)) &&\text{as }\beth\text{ is }n\text{-exangulated}\\
    &= (\tensor*[]{\SL\beth}{_{C}}\tensor*[]{)}{^{\BE''}}\tensor[]{\SE}{_{(\SL,\Phi)}}(\Lambda_{(C,A)}(\delta))&&\text{by \cref{lem:functor-respects-morphisms-plays-well-with-domains-codomains-functions}(ii)}\\
    &= (\tensor*[]{\SL\beth}{_{C}}\tensor*[]{)}{^{\BE''}} \tensor*[]{\Phi}{_{(\SG C,\SG A)}}(\Lambda_{(C,A)}(\delta))&&\text{using the definition of \(\tensor[]{\SE}{_{(\SL,\Phi)}}\).}
\end{align*}
This verifies that \(\SL\beth\) is an \(n\)-exangulated natural transformation.
\end{proof}

We are now in position to show that the collection of \(n\)-exangulated natural transformations is closed under vertical and horizontal composition.
\begin{prop}
\label{prop:compositions-are-n-exangulated}
The following statements hold.
\begin{enumerate}[label=\textup{(\roman*)}]
    
    \item\label{item:vertical-composition-is-n-exangulated} The vertical composition \(\beth'\circ_{v}\beth\colon (\SF,\Gamma)\Rightarrow(\SH,\Theta)\) is \(n\)-exangulated.
    
    \item\label{item:horizontal-composition-is-n-exangulated} The horizontal composition \(\daleth\circ_{h}\beth\colon (\SL,\Phi)\circ(\SF,\Gamma)\Rightarrow(\SM,\Psi)\circ(\SG,\Lambda)\) is \(n\)-exangulated.
\end{enumerate}
\end{prop}

\begin{proof}
\ref{item:vertical-composition-is-n-exangulated}\; 
Let \(\delta\inn\BE(C,A)\) be arbitrary. 
We have 
\begin{align*}
((\beth'\circ_{v}\beth\tensor[]{)}{_{A}}\tensor[]{)}{_{\BE'}}\tensor*[]{\Gamma}{_{(C,A)}}(\delta)
        &=(\tensor*[]{\beth}{^{\prime}_{A}}\tensor*[]{)}{_{\BE'}}(\tensor*[]{\beth}{_{A}}
        \tensor*[]{)}{_{\BE'}}
        \tensor*[]{\Gamma}{_{(C,A)}}(\delta)
            &&\text{using the definition of \(\circ_{v}\)}\\
        &=(\tensor*[]{\beth}{^{\prime}_{A}}\tensor*[]{)}{_{\BE'}}(\tensor*[]{\beth}{_{C}}\tensor[]{)}{^{\BE'}}
        \tensor*[]{\Lambda}{_{(C,A)}}(\delta)
            &&\text{as \(\beth\) is \(n\)-exangulated}\\
        &=(\tensor*[]{\beth}{_{C}}\tensor[]{)}{^{\BE'}}(\tensor*[]{\beth}{^{\prime}_{C}}\tensor[]{)}{^{\BE'}}
        \tensor*[]{\Theta}{_{(C,A)}}(\delta)
            &&\text{as \(\beth'\) is \(n\)-exangulated}\\
        &=
        ((\beth'\circ_{v}\beth\tensor[]{)}{_{C}}\tensor[]{)}{^{\BE'}}
        \tensor*[]{\Theta}{_{(C,A)}}(\delta)
            &&\text{using the definition of \(\circ_{v}\)},
\end{align*}
which verifies that \(\beth'\circ_{v}\beth\) is \(n\)-exangulated.

\ref{item:horizontal-composition-is-n-exangulated}\; 
This follows from \eqref{eqn:horizontal-comp-is-vertical-comp-of-whiskers}, \cref{prop:whiskerings-are-n-exangulated} and part \ref{item:vertical-composition-is-n-exangulated} above. 
\end{proof}

%%%%%%%%%%%%%%%%%%%%%%%%%%%%%%%%%%%%%%%%%%%%%%%%%%%
%%%%%%%%%%%%%%%%%%%%%%%%%%%%%%%%%%%%%%%%%%%%%%%%%%%

\subsection{\texorpdfstring{\(n\)}{n}-exangulated adjoints and equivalences}
\label{subsec:n-exangulated-adjoints}

In this subsection we discuss adjunctions and equivalences between \(n\)-exangulated categories. 
An important property of the functor \(\updave \colon \Exang{n} \to \Exactcat\), which will be defined in 
\cref{def:updave}, 
is that it preserves adjoint pairs and equivalences; see \cref{cor:n-exangulated-adjoint-pair-gives-adjoint-pair-of-functors-on-Ext-categories}.

Recall from Setup~\ref{setup:section4} that we consider \(n\)-exangulated functors 
\((\SF,\Gamma) \colon \begin{tikzcd}[cramped, column sep=0.3cm](\CC,\BE,\fs) \arrow{r} & (\CC',\BE',\fs')\end{tikzcd}\) 
and \((\SA,\NT) \colon (\CC',\BE',\fs') \to (\CC,\BE,\fs)\). 
In the case \(n=1\), part (ii) in \cref{def:n-exangulated-adjoints-and-equivalences} recovers the notion of an equivalence of extriangulated categories as defined in \cite[Prop.~2.13]{NakaokaOgawaSakai-localization-of-extriangulated-categories}. 
In the following we use the Hebrew letters \(\tsadi\) (tsadi) and \(\mem\) (mem).

\begin{defn}
\label{def:n-exangulated-adjoints-and-equivalences}
\begin{enumerate}[label=\textup{(\roman*)}]
    \item\label{item:n-exangulated-adjunction} 
    We call \(((\SF,\Gamma),(\SA,\NT))\) an \emph{\(n\)-exangulated adjoint pair} if \((\SF,\SA)\) is an adjoint pair for which the unit \(\tsadi\) and counit \(\mem\) give \(n\)-exangulated natural transformations \(\tsadi\colon(\idfunc{\CC}, \iden{\BE}) \Rightarrow (\SA,\NT) \circ (\SF,\Gamma)\) and \(\mem\colon(\SF,\Gamma)\circ(\SA,\NT)\Rightarrow(\idfunc{\CC'}, \iden{\BE'})\).

    \item\label{item:n-exangulated-equivalence} 
    We call \((\SF,\Gamma)\) an 
    \emph{\(n\)-exangulated equivalence} if there is an \(n\)-exangulated adjoint pair \(((\SF,\Gamma),(\SA,\NT))\) whose unit and counit give \(n\)-exangulated natural isomorphisms.

\end{enumerate}
\end{defn}

Recall that if \((\SF,\SA)\) is an adjoint pair with unit \(\tsadi\colon\idfunc{\CC}\Rightarrow\SA\SF\) and counit \(\mem\colon\SF\SA\Rightarrow\idfunc{\CC'}\), 
then 
one has the \emph{triangle identities} (or \emph{counit-unit equations}) 
\begin{equation}
\label{eqn:triangle-identities}
\tensor[]{\mem}{_{\SF X}}(\SF\tensor[]{\tsadi}{_{X}})
    = \iden{\SF X}
\hspace{1cm}
\text{and}
\hspace{1cm}
(\SA\tensor[]{\mem}{_{Y}})\tensor[]{\tsadi}{_{\SA Y}}
    = \iden{\SA Y}
\end{equation}
for all \(X\inn\CC\) and \(Y\inn\CC'\); 
see e.g.\ 
\cite[Thm.~IV.1.1(ii)]{MacLane-categories-for-the-working-mathematician}. 
In terms of vertical and horizontal compositions of (\(n\)-exangulated) natural transformations, 
the equations in \eqref{eqn:triangle-identities} 
give 
\begin{equation}
\label{eqn:triangle-identities-as-whiskerings}
(\mem \circ_{h} \iden{(\SF,\Gamma)}) 
    \circ_{v}
(\iden{(\SF,\Gamma)} \circ_{h} \tsadi)
     = \iden{(\SF,\Gamma)}
\hspace{0.5cm}
\text{and}
\hspace{0.5cm}
(\iden{(\SA,\Xi)} \circ_{h} \mem) 
    \circ_{v}
(\tsadi \circ_{h} \iden{(\SA,\Xi)})
     = \iden{(\SA,\Xi)}.
\end{equation} 
If 
\((\SF,\SA)\) is an adjoint equivalence, 
we also have 
\(
\tensor*[]{\mem}{_{\SF X}^{-1}} = \SF\tensor[]{\tsadi}{_{X}}
\) 
in \(\CC'\) 
and 
\(
\tensor*[]{\tsadi}{_{\SA Y}^{-1}} = \SA\tensor[]{\mem}{_{Y}}
\)
in \(\CC\).

We use the following lemma to characterise \(n\)-exangulated equivalences. Notice the similarity between the equations in the statement below and the triangle identities above.

\begin{lem}
\label{lemma:formulas-for-adjunctions-with-n-exangulated-counits-and-units}
Let \(((\SF,\Gamma),(\SA,\NT))\) be an \(n\)-exangulated adjoint pair with unit \(\tsadi\) and counit \(\mem\). 
Then for all \(A,C\inn\CC\) and \(B,D\inn\CC'\), we have:
\begin{enumerate}[label=\textup{(\roman*)}]
\item\label{item:formula-for-adjunction-E-prime}
    \((\tensor[]{\mem}{_{\SF A}}\tensor[]{)}{_{\BE'}}
    (\tensor[]{\SF\tsadi}{_{C}}\tensor[]{)}{^{\BE'}}
    (\tensor[]{\Gamma}{_{\SA\times\SA}}\NT\tensor[]{)}{_{(\SF C,\SF A)}}
        = \iden{\BE'(\SF C,\SF A)}\); and
\item\label{item:formula-for-adjunction-E}
    \((\tensor[]{\SA\mem}{_{B}}\tensor[]{)}{_{\BE}}
    (\tensor[]{\tsadi}{_{\SA D}}\tensor[]{)}{^{\BE}}
    (\tensor[]{\NT}{_{\SF\times\SF}}\Gamma\tensor[]{)}{_{(\SA D,\SA B)}}
        = \iden{\BE(\SA D,\SA B)}\).
\end{enumerate}
\end{lem}

\begin{proof}
We just show \ref{item:formula-for-adjunction-E-prime}, 
as the proof of \ref{item:formula-for-adjunction-E}
is similar. 
Let \(\delta' \inn\BE'(\SF C, \SF A)\) be arbitrary. 
Since \(\mem\colon(\SF\SA,\tensor[]{\Gamma}{_{\SA\times\SA}}\NT)\Rightarrow(\idfunc{\CC'}, \iden{\BE'})\) is an \(n\)-exangulated natural transformation, we get 
\[
(\tensor[]{\SF\tsadi}{_{C}}\tensor[]{)}{^{\BE'}}(\tensor[]{\mem}{_{\SF A}}\tensor[]{)}{_{\BE'}}(\tensor[]{\Gamma}{_{\SA\times\SA}}\NT\tensor[]{)}{_{(\SF C,\SF A)}}(\delta')
        = (\tensor[]{\SF\tsadi}{_{C}}\tensor[]{)}{^{\BE'}}
        (\tensor[]{\mem}{_{\SF C}}\tensor[]{)}{^{\BE'}}(\delta')
        = \delta', 
\] 
where the first equality follows from \eqref{eqn:n-exangulated-natural-transformation-property} and the second from \eqref{eqn:triangle-identities}.
\end{proof}
\cref{prop:characterisation-of-n-exangulated-equivalence} below is a higher analogue of \cite[Prop.~2.13]{NakaokaOgawaSakai-localization-of-extriangulated-categories}, giving a characterisation of when an \(n\)-exangulated functor is an \(n\)-exangulated equivalence. 
One direction in the proof is provided in \cite{NakaokaOgawaSakai-localization-of-extriangulated-categories} for the extriangulated case and easily translates to the \(n\)-exangulated setting, so we omit this here. 
We provide a proof for the other implication. 
The following statement has also appeared independently; see \cite[Prop.\ 2.14]{HeHeZhou-localization-of-n-exangulated-categories}.

\begin{prop}
\label{prop:characterisation-of-n-exangulated-equivalence}
The pair \((\SF,\Gamma)\) is an \(n\)-exangulated equivalence if and only if \(\SF\) is an equivalence of categories and \(\Gamma\) is a natural isomorphism.
\end{prop}

\begin{proof}
\((\Rightarrow)\)\; 
Suppose that we are given an \(n\)-exangulated adjoint pair 
\(((\SF,\Gamma),(\SA,\NT))\) with corresponding \(n\)-exangulated natural isomorphisms
\(\tsadi\colon(\idfunc{\CC}, \iden{\BE}) \Rightarrow (\SA,\NT) \circ (\SF,\Gamma)\) 
and 
\(\mem\colon(\SF,\Gamma)\circ(\SA,\NT)\Rightarrow(\idfunc{\CC'}, \iden{\BE'})\) 
defined by the unit and counit, respectively.
It follows from classical theory that \(\SF\) is an equivalence. Thus, it remains to show that for all \(A,C\inn\CC\), the induced homomorphism
\(
\tensor[]{\Gamma}{_{(C,A)}}\colon \BE(C,A) \to \BE'(\SF C, \SF A)
\) is invertible. 
We claim that 
the composition
\[
(\tsadi_{A}^{-1}\tensor[]{)}{_{\BE}}
(\tensor[]{\tsadi}{_{C}}\tensor[]{)}{^{\BE}}
\tensor[]{\NT}{_{(\SF C,\SF A)}}\colon \BE'(\SF C,\SF A) \to \BE(C, A)
\] 
is a two-sided inverse of 
\(
\tensor[]{\Gamma}{_{(C,A)}}
\). 
First, for each \(\delta\inn\BE(C,A)\), notice that
\[
    (\tsadi_{A}^{-1}\tensor[]{)}{_{\BE}}
        (\tensor[]{\tsadi}{_{C}}\tensor[]{)}{^{\BE}}
        \tensor[]{\NT}{_{(\SF C,\SF A)}}\tensor[]{\Gamma}{_{(C,A)}}(\delta)
    = (\tsadi_{A}^{-1}\tensor[]{)}{_{\BE}}
        (\tensor[]{\tsadi}{_{A}}\tensor[]{)}{_{\BE}}
        \iden{\BE}(\delta)
    = \delta,
\]
where the first equality follows from 
\eqref{eqn:n-exangulated-natural-transformation-property} 
for \(\tsadi\). 
It remains to check that 
\(
(\tsadi_{A}^{-1}\tensor[]{)}{_{\BE}}
(\tensor[]{\tsadi}{_{C}}\tensor[]{)}{^{\BE}}
\tensor[]{\NT}{_{(C,A)}}
\)
is a right inverse of 
\(
\tensor[]{\Gamma}{_{(C,A)}}
\). 
For 
\(\delta'\inn\BE'(\SF C,\SF A)\), we have that 
\begin{align*}
    \tensor[]{\Gamma}{_{(C,A)}} 
        (\tsadi_{A}^{-1}\tensor[]{)}{_{\BE}}
        (\tensor[]{\tsadi}{_{C}}\tensor[]{)}{^{\BE}}
        \tensor[]{\NT}{_{(\SF C,\SF A)}}(\delta') 
    &= (\SF\tsadi_{A}^{-1}\tensor[]{)}{_{\BE'}}
    (\SF\tensor[]{\tsadi}{_{C}}\tensor[]{)}{^{\BE'}}
    \tensor[]{\Gamma}{_{(\SA \SF C, \SA \SF A)}}\tensor[]{\NT}{_{(\SF C,\SF A)}}(\delta') \\
    &= (\tensor[]{\mem}{_{\SF A}}\tensor[]{)}{_{\BE'}}
    (\SF\tensor[]{\tsadi}{_{C}}\tensor[]{)}{^{\BE'}}
    \tensor[]{\Gamma}{_{(\SA \SF C, \SA \SF A)}}\tensor[]{\NT}{_{(\SF C,\SF A)}}(\delta'),
\end{align*}
where the first equality is by the naturality of \(\Gamma\) and the second follows from \eqref{eqn:triangle-identities}. 
This last term 
is equal to \(\delta'\) by \cref{lemma:formulas-for-adjunctions-with-n-exangulated-counits-and-units}\ref{item:formula-for-adjunction-E-prime}, as required.
\end{proof}

%%%%%%%%%%%%%%%%%%%%%%%%%%%%%%%%%%%%%%%%%%%%%%%%%%%
%%%%%%%%%%%%%%%%%%%%%%%%%%%%%%%%%%%%%%%%%%%%%%%%%%%

\subsection{A \texorpdfstring{\(2\)}{2}-categorical viewpoint}
\label{subsec:2-categorical-viewpoint}

We start this subsection by using what we have shown so far to prove that \(\exang{n}\) is a \(2\)-category. More generally, we establish a \(\Hom\)-category for each pair of \(n\)-exangulated categories, as explained in the proposition below.

\begin{prop}
\label{prop:hom-category-of-n-exangulated-categories}
For each pair \((\CC,\BE,\fs)\) and \((\CC',\BE',\fs')\) of \(n\)-exangulated categories, there is a category 
\(\CN\deff\Exang{n}((\CC,\BE,\fs),(\CC',\BE',\fs'))\) 
whose objects are \(n\)-exangulated functors \((\CC,\BE,\fs)\to(\CC',\BE',\fs')\) and whose morphisms are \(n\)-exangulated natural transformations.
\end{prop}

\begin{proof}
Define composition of morphisms in \(\CN\) to be vertical composition \(\circ_{v}\) of natural transformations, which is well-defined by \cref{prop:compositions-are-n-exangulated}\ref{item:vertical-composition-is-n-exangulated}. 
\cref{prop:identity-is-n-exangulated-associativity-unitality-and-middle-four-exchange}\ref{item:identity-n-exangulated-transformation-is-n-exangulated} and \ref{item:assoc-unital-n-exangulated-nat-trans} imply that \(\CN\) is a category.
\end{proof}

\begin{rem}
\label{rem:not-2-categories}
Note that in this article, and in particular in \cref{prop:hom-category-of-n-exangulated-categories}, 
the term `category' does not require \emph{smallness}, 
i.e.\ the collections of objects and morphisms associated to a category are not assumed to form sets. 
In spite of this, we have so far usually
considered additive categories. 
These are necessarily \emph{locally small}, which means that the collection of morphisms between any two objects is a set, 
since a group is a set. 

Due to a set-theoretic issue, care must be taken when referring to the established notion of a \emph{\(2\)-category}. In order to ensure that the collection of \(2\)-cells between a pair of fixed \(1\)-cells forms a set instead of a proper class, it is common from a \(2\)-categorical viewpoint to only consider small categories; see e.g.\ \cite[p.~43]{MacLane-categories-for-the-working-mathematician}. 
Consequently, we do not refer to the categories \(\Exang{n}\) and \(\Exactcat\) as \(2\)-categories, since collections of \(2\)-cells between pairs of \(1\)-cells need not form sets.
Yet, it is still natural to adopt a \(2\)-categorical perspective on these categories.
Nevertheless, for accuracy and in the interest of not abusing established terminology, we avoid the terms `\(2\)-category' and `\(2\)-functor' in \cref{prop:hom-category-of-n-exangulated-categories} and \cref{thm:2-functor}.
\end{rem}
So far, we have used the notation \(\Exang{n}\) (resp.\ \(\Exactcat\)) for the \(1\)-category of all \(n\)-exangulated (resp.\ exact) categories. 
Similarly, we have used \(\exang{n}\) and \(\exactcat\) to denote the 
subcategories obtained by restricting to small categories. 
In order to formally place \(\exang{n}\) and \(\exactcat\) in a \(2\)-categorical framework, we make our terminology more precise by indicating below the \(0\)-cells, \(1\)-cells and \(2\)-cells of 
these structures.

\begin{notn}
\label{notn:i-cells}
We write \(\Exang{n}\) and \(\Exactcat\) for the collections of \(0\)-cells, \(1\)-cells and \(2\)-cells described in the table below.
\bigskip
\begin{center}
\begin{tabular}{|c|c|c|}\hline

            & \(\Exang{n}\)                 &   \(\Exactcat\)     \\\hline
\(0\)-cells & \(n\)-exangulated categories  &   exact categories\\\hline
\(1\)-cells & \(n\)-exangulated functors    & exact functors\\\hline
\(2\)-cells & \(n\)-exangulated natural transformations     &  natural transformations\\\hline
\end{tabular}
\bigskip
\end{center}
The \(i\)-cells above also induce \(i\)-cells for \(\exang{n}\) and  \(\exactcat\), where the only difference is that for \(0\)-cells we restrict to small categories.
\end{notn}

It is well-known that the composition of two exact functors is an exact functor, and hence \(\exactcat\) is a \(2\)-category; 
see e.g.\ \cite[1.4(a)]{Lack-a-2-categories-companion}. 
From the theory developed in Subsection~\ref{subsec:composing-n-exangulated-natural-transformations} and in this subsection so far, 
one can deduce that \(\Exang{n}\) has the characteristics of a \(2\)-category. 
When restricting to small categories, the set-theoretic issue described in \cref{rem:not-2-categories} is avoided and we obtain the following immediate corollary.

\begin{cor}
\label{cor:2-cats}
The category \(\exang{n}\) is a \(2\)-category.
\end{cor}

The final aim of this section is to provide a \(2\)-categorical understanding of 
how \(n\)-exan\-gu\-lat\-ed concepts relate to notions on the level of 
associated categories of extensions, bringing toge\-ther
our work in  Sections~\ref{section:3} and \ref{section:4}. 
We do this by constructing a functor \mbox{\(\updave \colon \Exang{n} \to \Exactcat\)} 
that respects the \(2\)-categorical properties of the categories involved. 
In particular, 
this induces a \emph{\(2\)-functor} 
\(\updave \colon \exang{n} \to \exactcat\) 
in the sense of \cite[p.~278]{MacLane-categories-for-the-working-mathematician}.

In order to have a \(2\)-functor, one must give an assignment of \(i\)-cells in the domain \mbox{$2$-category} to \(i\)-cells in the codomain $2$-category for \(i \inn \{0,1,2\}\), 
satisfying certain compatibility conditions. Based on the theory developed in \cref{section:3}, we can define the functor \(\updave\) on \(0\)-cells by sending an \(n\)-exangulated category  
to its category of extensions, 
and on \(1\)-cells by sending an \mbox{\(n\)-exangulated} functor \((\SF,\Gamma)\) to \(\tensor[]{\SE}{_{(\SF,\Gamma)}}\) as described in \cref{thm:characterisation-of-n-exangulated-functors}. It remains to consider how \(\updave\) acts on \(2\)-cells, 
that is, on \(n\)-exangulated natural transformations. The next lemma constitutes a first step in this direction. 
We use the Hebrew letter \(\aleph\) (aleph).

\begin{lem}
\label{lem:natural-transformations-between-extension-functors}
Any natural transformation 
\(
\aleph
        \colon 
    \tensor[]{\SE}{_{(\SF,\Gamma)}}
        \Rightarrow
    \tensor[]{\SE}{_{(\SG,\Lambda)}}
\) 
gives rise to natural transformations \(\tensor*[]{\aleph}{^{\ell}}\colon \SF\Rightarrow\SG\) and \(\tensor*[]{\aleph}{^{r}}\colon \SF\Rightarrow\SG\) 
such that 
\(
\tensor[]{\aleph}{_{\delta}}=(\tensor*[]{\aleph}{^{\ell}_{A}},\tensor*[]{\aleph}{^{r}_{C}})
\) 
for each  \(\delta\inn\BE(C,A)\). 
\end{lem}

\begin{proof}
Since 
\(
\aleph\colon\tensor[]{\SE}{_{(\SF,\Gamma)}}\Rightarrow\tensor[]{\SE}{_{(\SG,\Lambda)}}
\) is a natural transformation, 
for each \(\delta\inn\BE(C,A)\), 
there is a morphism 
\(
\tensor[]{\aleph}{_{\delta}}\colon\tensor[]{\Gamma}{_{(C,A)}}(\delta)\to \tensor[]{\Lambda}{_{(C,A)}}(\delta)
\)
in \(\BE'\dExt{\CC'}\). 
This implies that there are morphisms 
\(\tensor[]{\ell}{_{\delta}}\colon \SF A\to \SG A\) 
and 
\(\tensor[]{r}{_{\delta}}\colon \SF C\to \SG C\) 
in \(\CC'\) such that  
\(
\tensor[]{\aleph}{_{\delta}}=(\tensor[]{\ell}{_{\delta}},\tensor[]{r}{_{\delta}})
\). 
We claim that 
\(\tensor[]{\ell}{_{\delta}}\) depends only on the object \(A\). 
To see this, 
recall that for each \(Z\inn\CC\) we write \(\tensor[_{Z}]{0}{_{0}}\) for the trivial element of the abelian group \(\BE(0,Z)\). 
Consider the morphism 
\(
(\iden{A},0)\colon \tensor[_{A}]{0}{_{0}} \to \delta
\)
for a fixed extension \(\delta\inn\BE(C,A)\). 
By \cref{rem:equality-of-morphisms-of-extensions}, 
we have the equalities
\begin{align*}
((\SG \iden{A})\tensor[]{\ell\hspace{-4pt}}{_{\tensor[_{A}]{0}{_{0}}}},0)
    &= \tensor[]{\SE}{_{(\SG,\Lambda)}}(\iden{A},0)\circ\tensor*[]{\aleph\hspace{-4pt}}{_{\tensor[_{A}]{0}{_{0}}}} &&\text{as \(\tensor[]{\SE}{_{(\SG,\Lambda)}}\) respects morphisms over \(\SG\)}\\
    &= \tensor*[]{\aleph}{_{\delta}}\circ\tensor[]{\SE}{_{(\SF,\Gamma)}}(\iden{A},0) &&\text{as \(\aleph\) is natural}\\
    &= (\tensor[]{\ell}{_{\delta}}(\SF \iden{A}),0) &&\text{as \(\tensor[]{\SE}{_{(\SF,\Gamma)}}\) respects morphisms over \(\SF\),}
\end{align*}
as morphisms \(\tensor[]{\Gamma}{_{(0,A)}}(\tensor[_{A}]{0}{_{0}})\to \tensor[]{\Lambda}{_{(C,A)}}(\delta)\) in \(\BE'\dExt{\CC'}\). 
Consequently, the morphism \(
\tensor[]{\ell}{_{\delta}} 
    = \tensor[]{\ell\hspace{-4pt}}{_{\tensor[_{A}]{0}{_{0}}}}
\) 
depends only on \(A\), 
and we write 
\(
\tensor*[]{\aleph}{^{\ell}_{A}}
    \deff 
    \tensor[]{\ell}{_{\delta}}
    \colon \SF A\to \SG A
\). 
Similarly, the morphism \(\tensor[]{r}{_{\delta}}\) depends only on \(C\), and we write 
\(
\tensor*[]{\aleph}{^{r}_{C}} \coloneqq \tensor[]{r}{_{\delta}} \colon \SF C\to \SG C
\). 

It remains to show that 
\(
\tensor[]{\aleph}{^{\ell}}
    \deff \tensor[]{\set{\tensor*[]{\aleph}{^{\ell}_{A}}}}{_{A\inn\CC}}
\) 
and 
\(
\tensor[]{\aleph}{^{r}}
    \deff \tensor[]{\set{\tensor*[]{\aleph}{^{r}_{C}}}}{_{C\inn\CC}}
\) 
define natural transformations \(\SF\Rightarrow\SG\). 
Fix a morphism \(x\colon X\to Y\) in \(\CC\). 
The pair
\((x,0)\colon\tensor[_{X}]{0}{_{0}}\to\tensor[_{Y}]{0}{_{0}}\) is a morphism in \(\BE\dExt\CC\), 
so 
\(
\tensor[]{\aleph\hspace{-4pt}}{_{\tensor[_{Y}]{0}{_{0}}}}\tensor[]{\SE}{_{(\SF,\Gamma)}}(x,0)
    =\tensor[]{\SE}{_{(\SG,\Lambda)}}(x,0)\tensor[]{\aleph\hspace{-4pt}}{_{\tensor[_{X}]{0}{_{0}}}}
\)
as \(\aleph\) is natural. 
This yields \(\tensor*[]{\aleph}{^{\ell}_{Y}} \SF x = (\SG x) \tensor*[]{\aleph}{^{\ell}_{X}},\) 
so \(\tensor*[]{\aleph}{^{\ell}}\) is a natural transformation. 
The proof that \(\tensor*[]{\aleph}{^{r}}\) is natural is similar.
\end{proof}
Note that it does not necessarily follow that 
\(\tensor*[]{\aleph}{^{\ell}}\)
and 
\(\tensor*[]{\aleph}{^{r}}\) coincide in the lemma above, as  demonstrated by the following example.

\begin{example}
\label{example:unbalanced-natural-transformation}
Let \(\CC\) be a non-zero additive category. Equip \(\CC\) with its split \(n\)-exangulated 
structure \((\CC,\BE,\fs)\), which is induced by the split \(n\)-exact structure as explained in \cref{example:split-n-exangulated-structure}.
Consider the identity 
\(n\)-exangulated functor 
\(
(\iden{\CC},\iden{\BE}) 
    \colon (\CC,\BE,\fs) \to (\CC,\BE,\fs)
\). 
By \cref{thm:characterisation-of-n-exangulated-functors} and \cref{lem:identity-composition-n-exangulated-functors-and-extension-functors-respect-compositions}(iii), 
we obtain the additive functor 
\(
\tensor[]{\SE}{_{(\iden{\CC},\idfunc{\BE})}}
    =\iden{\BE\dExt{\CC}}
\). 
For 
\(\delta\inn\BE(C,A)\), define 
\(
\tensor[]{\aleph}{_{\delta}}
    \colon \delta = \tensor[]{\SE}{_{(\iden{\CC},\idfunc{\BE})}}(\delta)
        \to \tensor[]{\SE}{_{(\iden{\CC},\idfunc{\BE})}}(\delta) = \delta
\) 
by 
\(\tensor[]{\aleph}{_{\delta}} = (\iden{A},0)\). 
Note that \(\tensor[]{\aleph}{_{\delta}}\) is a morphism \mbox{\(\delta\to \delta\)} in \(\BE\dExt{\CC}\) since \(\BE(C,A)\) is trivial. 
It is straightforward to check that \(\aleph = \{\tensor[]{\aleph}{_{\delta}}\tensor[]{\}}{_{\delta\inn\BE\dExt{\CC}}}\) 
defines a natural transformation 
\(\tensor[]{\SE}{_{(\iden{\CC},\idfunc{\BE})}}\Rightarrow\tensor[]{\SE}{_{(\iden{\CC},\idfunc{\BE})}}\). 
Since \(
\tensor*[]{\aleph}{_{A}^{\ell}} 
    = \iden{A}
    \neq 0
    = \tensor*[]{\aleph}{_{A}^{r}}
\) for any non-zero \(A\inn\CC\), we have that \(\aleph\) is not \emph{balanced} in the sense of \cref{def:balanced-natural-transformation} below. 
\end{example}

\begin{defn}
\label{def:balanced-natural-transformation}
Let \(\aleph\colon \tensor[]{\SE}{_{(\SF,\Gamma)}}\Rightarrow\tensor[]{\SE}{_{(\SG,\Lambda)}}\) be a natural transformation.
In the notation of \cref{lem:natural-transformations-between-extension-functors}, 
we call \(\aleph\) \emph{balanced} provided that \(\tensor*[]{\aleph}{^{\ell}}=\tensor*[]{\aleph}{^{r}}\). 
\end{defn}

We can now prove \cref{thmx:nat-tran-of-n-exan-functors-induces-nat-trans-of-functors-on-Ext-categories} from \cref{sec:introduction}.

\begin{thm}
\label{thm:characterisation-of-n-exangulated-natural-transformations}
There is a one-to-one correspondence
\[
\begin{adjustbox}{scale=1,center}
$\displaystyle
\begin{aligned}[t]
    \Set{
        \begin{array}{c}
            n\textit{-exangulated natural transformations}\\
            \beth\colon (\SF,\Gamma)\Longrightarrow(\SG,\Lambda)
        \end{array}
    }
& \longleftrightarrow 
    \Set{
        \begin{array}{c}
            \textit{balanced natural transformations}\\
            \aleph\colon \tensor[]{\SE}{_{(\SF,\Gamma)}}\Longrightarrow\tensor[]{\SE}{_{(\SG,\Lambda)}}
        \end{array}
    }\\[5pt]
\beth 
\hspace{6pt}  
    &
    \xmapsto{\hspace{15.5pt}}\hspace{6pt}
    \lan \beth \ran 
\\
\tensor*[]{\aleph}{^{\ell}}=\tensor*[]{\aleph}{^{r}} 
\hspace{6pt}
    &
    \xmapsfrom{\hspace{15.5pt}}\hspace{6pt}
    \aleph,
\end{aligned}$
\end{adjustbox}
\]
where 
\(
\lan\beth\tensor[]{\ran}{_{\delta}}=(\tensor[]{\beth}{_{A}}, \tensor[]{\beth}{_{C}})
\)
for all 
\(
A,C\inn\CC 
\)
and each 
\(
\delta\inn\BE(C,A)
\).
\end{thm}

\begin{proof}
Suppose first that \(\beth\colon (\SF,\Gamma)\Rightarrow(\SG,\Lambda)\) is an \(n\)-exangulated natural transformation, and consider the collection 
 \(
\lan \beth \ran 
    = \tensor[]{\set{\lan\beth\tensor[]{\ran}{_{\delta}}}}{_{\delta\inn\BE\dExt{\CC}}}
\)
where \(\lan\beth\tensor[]{\ran}{_{\delta}}\) is as
defined in the statement of the theorem. 
Since \(\beth\) is \(n\)-exangulated, each pair 
\(
\lan\beth\tensor[]{\ran}{_{\delta}}
    = (\tensor[]{\beth}{_{A}},\tensor[]{\beth}{_{C}})
\) 
is a morphism from 
\(
\tensor[]{\SE}{_{(\SF,\Gamma)}}(\delta) 
    = \tensor[]{\Gamma}{_{(C,A)}}(\delta)
\)
to 
\(
\tensor[]{\Lambda}{_{(C,A)}}(\delta) 
    = \tensor[]{\SE}{_{(\SG,\Lambda)}}(\delta)
\) 
in \(\BE'\dExt{\CC'}\).
Let \((a,c)\colon \delta\to \rho\) be a morphism in \(\BE\dExt\CC\) with 
\(\delta\inn\BE(C,A)\) and \(\rho\inn\BE(D,B)\).
We must show that the square
\[
\begin{tikzcd}
\tensor[]{\SE}{_{(\SF,\Gamma)}}(\delta)
    \arrow{r}{\lan\beth\tensor[]{\ran}{_{\delta}}}
    \arrow{d}[swap]{\tensor[]{\SE}{_{(\SF,\Gamma)}}(a,c)}
& \tensor[]{\SE}{_{(\SG,\Lambda)}}(\delta)
    \arrow{d}{\tensor[]{\SE}{_{(\SG,\Lambda)}}(a,c)}\\
\tensor[]{\SE}{_{(\SF,\Gamma)}}(\rho) \arrow{r}{\lan\beth\tensor[]{\ran}{_{\rho}}}&  \tensor[]{\SE}{_{(\SG,\Lambda)}}(\rho)
\end{tikzcd}
\]
commutes in \(\BE'\dExt{\CC'}\). 
By \cref{rem:equality-of-morphisms-of-extensions}, it is enough to observe that 
\begin{align*}
\tensor[]{\SE}{_{(\SG,\Lambda)}}(a,c)\lan\beth\tensor[]{\ran}{_{\delta}}
    &= (\SG a,\SG c)(\tensor[]{\beth}{_{A}},\tensor[]{\beth}{_{C}}) &&\text{as \(\tensor[]{\SE}{_{(\SG,\Lambda)}}\) respects morphisms over \(\SG\)}\\
    &= (\tensor[]{\beth}{_{B}}\SF a,\tensor[]{\beth}{_{D}}\SF c) &&\text{since \(\beth\) is natural}\\
    &= \lan\beth\tensor[]{\ran}{_{\rho}}\tensor[]{\SE}{_{(\SF,\Gamma)}}(a,c)&&\text{as \(\tensor[]{\SE}{_{(\SF,\Gamma)}}\) respects morphisms over \(\SF\)}.
\end{align*}
This means that 
\(
\lan\beth\ran\colon\tensor[]{\SE}{_{(\SF,\Gamma)}}\Rightarrow\tensor[]{\SE}{_{(\SG,\Lambda)}}
\) 
is natural, 
and it 
is balanced by construction.

On the other hand, suppose \(\aleph\colon\tensor[]{\SE}{_{(\SF,\Gamma)}}\Rightarrow\tensor[]{\SE}{_{(\SG,\Lambda)}}\) is balanced. 
Write 
\(\tensor*[]{\aleph}{^{\ell}}=\tensor*[]{\aleph}{^{r}}\colon\SF\Rightarrow\SG\) for the natural transformation from \cref{lem:natural-transformations-between-extension-functors} satisfying  \(\tensor[]{\aleph}{_{\delta}}=(\tensor*[]{\aleph}{^{\ell}_{A}},\tensor*[]{\aleph}{^{\ell}_{C}})\) for each \(\delta\inn\BE(C,A)\). 
The pair 
\(
\tensor[]{\aleph}{_{\delta}} 
    = (\tensor*[]{\aleph}{^{\ell}_{A}},\tensor*[]{\aleph}{^{\ell}_{C}})
\) 
is a morphism from 
\(
\tensor[]{\Gamma}{_{(C,A)}}(\delta)
    = \tensor[]{\SE}{_{(\SF,\Gamma)}}(\delta) 
\)
to 
\(
    \tensor[]{\SE}{_{(\SG,\Lambda)}}(\delta)
        = \tensor[]{\Lambda}{_{(C,A)}}(\delta)
\), 
so 
\(
\tensor*[]{\aleph}{^{\ell}}
\) 
is an \(n\)-exangulated natural transformation. 

We can thus conclude that the two assignments \(\beth \mapsto \lan \beth \ran \) and \(\aleph \mapsto \tensor*[]{\aleph}{^{\ell}}=\tensor*[]{\aleph}{^{r}} \) from the statement of the theorem are well-defined. It is straightforward to check that they are mutually inverse, and hence define a one-to-one correspondence.
\end{proof}

Just as \cref{thm:characterisation-of-n-exangulated-functors} allowed us to define the functor \(\updave \colon \Exang{n} \to \Exactcat\) on \(1\)-cells, the characterisation in 
\cref{thm:characterisation-of-n-exangulated-natural-transformations} enables us to define \(\updave\) on \(2\)-cells by sending an \(n\)-exangulated natural transformation \(\beth\) to the balanced natural transformation \(\lan \beth \ran\). 
We can hence complete the definition of \(\updave\).
For \(i\inn\{0,1,2\}\), denote by \(\Exang{n}_{i}\) and \(\Exactcat_{i}\) the collection of \(i\)-cells of \(\Exang{n}\) and \(\Exactcat\), respectively.

\begin{defn}
\label{def:updave}
Let \(\updave = (\updave_{0},\updave_{1},\updave_{2}) \colon \Exang{n} \to \Exactcat\) be defined by the assignments
\(\updave_{i}\colon \Exang{n}_{i} \to \Exactcat_{i}\), 
where:
\begin{align*}
\updave_{0}  (\CC,\BE,\fs)  
    &\deff (\BE\dExt\CC,\CX_{\BE}),\\
\updave_{1}  (\SF,\Gamma)  
    &\deff\tensor[]{\SE}{_{(\SF,\Gamma)}},\\
\updave_{2} ( \beth ) 
    &\deff\lan\beth\ran.
\end{align*}
\end{defn}

\begin{rem}\label{rem:notation-remark-for-2-functor-theorem}
We discuss \cref{def:updave} with a view towards explaining \cref{thm:2-functor}.
\begin{enumerate}[label=\textup{(\roman*)}]
    
    \item\label{item:rem-0updave-0-and-1-give-functor-on-underyling-1-categories} The assignments \(\updave_{i}\) are well-defined: by \cref{prop:ECC-is-a-category}, the assignment \(\updave_{0}\) takes an object of \(\Exang{n}\) to an object of \(\Exactcat\);  
    by \cref{thm:characterisation-of-n-exangulated-functors} and \cref{prop:F-additive-iff-E-additive-iff-E-exact}, 
\(\updave_{1}\) 
associates an exact functor 
\(\updave_{1}  (\SF,\Gamma) = \tensor[]{\SE}{_{(\SF,\Gamma)}}\) 
from \(\updave_{0}(\CC,\BE,\fs)\) to \(\updave_{0}(\CC',\BE',\fs')\) 
to each \(n\)-exangulated functor 
\((\SF,\Gamma)\colon (\CC,\BE,\fs)\to (\CC',\BE',\fs')\); 
and \(\updave_{2}\) takes an \(n\)-exangulated natural transformation 
\(\beth\colon (\SF,\Gamma)\Rightarrow(\SG,\Lambda)\) 
to a natural transformation 
\(
\updave_{2}(\beth) 
    = \lan\beth\ran 
    \colon \updave_{1}(\SF,\Gamma) \to \updave_{1}(\SG,\Lambda)
\) 
by \cref{thm:characterisation-of-n-exangulated-natural-transformations}. 

    \item\label{item:rem-functor-on-Hom-categories} In \cref{thm:2-functor} below, given \(n\)-exangulated categories \((\CC,\BE,\fs)\) and \((\CC',\BE',\fs')\), 
we denote by 
\(\CA
    \deff 
    \Exactcat( 
        (\BE\dExt{\CC},\CX_{\BE}),
        (\BE'\dExt{\CC'},\CX_{\BE'})
    )
\)
the category 
whose objects are exact functors from \((\BE\dExt{\CC},\CX_{\BE})\)
to 
\((\BE'\dExt{\CC'},\CX_{\BE'})\)
and 
whose morphisms are natural transformations. 
Composition of morphisms in \(\CA\) is vertical composition of natural transformations. 
It follows from 
\cite[pp.~40, 43]{MacLane-categories-for-the-working-mathematician} 
that \(\CA\) is a category. 
\item If one ignores the set-theoretic issue from \cref{rem:not-2-categories}, then \cref{thm:2-functor} below can be interpreted as showing that
the triplet \(\updave\) satisfies the properties of a \(2\)-functor \(\Exang{n}\to\Exactcat\); see \cite[Prop.\ 4.1.8]{JohnsonYau-2-dimensional-categories}.
\end{enumerate}
\end{rem}

\begin{thm}
\label{thm:2-functor}
The following statements hold for the assignments
\(\updave_{0}\), \(\updave_{1}\) and \(\updave_{2}\).
\begin{enumerate}[label=\textup{(\roman*)}] 
    
    \item\label{item:updave-0-and-1-give-functor-on-underyling-1-categories}
    The pair  
    \((\updave_{0}, \updave_{1})\) 
    defines a functor 
    \(\Exang{n}\to\Exactcat\).
    
    \item\label{item:functor-on-Hom-categories} 
    Given \(n\)-exangulated categories \((\CC,\BE,\fs)\) and \((\CC',\BE',\fs')\), 
    the pair 
    \((\updave_{1}, \updave_{2})\) 
    defines a functor \( \CN \to \CA\) in the notation of \textup{\cref{prop:hom-category-of-n-exangulated-categories}} and \textup{\cref{rem:notation-remark-for-2-functor-theorem}\ref{item:rem-functor-on-Hom-categories}}.
    
    \item\label{item:updave2-preserves-horizontal-composition} The assignment \(\updave_{2}\) preserves horizontal composition.
\end{enumerate}
\end{thm}
\begin{proof}
\ref{item:updave-0-and-1-give-functor-on-underyling-1-categories}\; 
This part follows from \cref{rem:notation-remark-for-2-functor-theorem}\ref{item:rem-0updave-0-and-1-give-functor-on-underyling-1-categories} and \cref{lem:identity-composition-n-exangulated-functors-and-extension-functors-respect-compositions}\ref{item:induced-scriptE-functors-respect-identity-and-composition}.

\ref{item:functor-on-Hom-categories}\; 
The second statement holds by 
\cref{rem:notation-remark-for-2-functor-theorem}\ref{item:rem-0updave-0-and-1-give-functor-on-underyling-1-categories}
combined with 
a straightforward check to see that \(\updave_{2}\) preserves identity morphisms (i.e.\ identity \(n\)-exangulated natural transformations) and (vertical) composition.

\ref{item:updave2-preserves-horizontal-composition}\;
We must show that the natural transformations 
\(
\updave_{2}(\daleth\circ_{h}\beth)
    \colon 
        \tensor[]{\SE}{_{(\SL,\Phi)\circ(\SF,\Gamma)}}
        \Rightarrow 
        \tensor[]{\SE}{_{(\SM,\Psi)\circ(\SG,\Lambda)}}
\) 
and 
\(
(\updave_{2}\daleth)\circ_{h}(\updave_{2}\beth)
    \colon 
        \tensor[]{\SE}{_{(\SL,\Phi)}}\circ\tensor[]{\SE}{_{(\SF,\Gamma)}}
        \Rightarrow 
        \tensor[]{\SE}{_{(\SM,\Psi)}}\circ\tensor[]{\SE}{_{(\SG,\Lambda)}}
\)
are equal. 
Notice first that their domains (resp.\ codomains) agree by \cref{lem:identity-composition-n-exangulated-functors-and-extension-functors-respect-compositions}\ref{item:induced-scriptE-functors-respect-identity-and-composition}. 
Consequently, for \(\delta\inn\BE(C,A)\), the morphisms 
\(
(\updave_{2}(\daleth\circ_{h}\beth)\tensor[]{)}{_{\delta}}
\) 
and 
\(
((\updave_{2}\daleth)\circ_{h}(\updave_{2}\beth)\tensor[]{)}{_{\delta}}
\) 
have the same domain  
(resp.\ codomain). 
Thus, by \cref{rem:equality-of-morphisms-of-extensions}, it is enough to note that
\phantom{\qedhere}
\begin{align*}
(\updave_{2}(\daleth\circ_{h}\beth)\tensor[]{)}{_{\delta}}
    &= \lan\daleth\circ_{h}\beth\tensor[]{\ran}{_{\delta}}
        &&\text{by the definition of \(\updave_{2}\)}
        \\
    &= ((\daleth\circ_{h}\beth\tensor[]{)}{_{A}},(\daleth\circ_{h}\beth\tensor[]{)}{_{C}})
        &&\text{see \cref{thm:characterisation-of-n-exangulated-natural-transformations}}
        \\
    &= (\tensor[]{\daleth}{_{\SG A}}(\SL \tensor[]{\beth}{_{A}}),\tensor[]{\daleth}{_{\SG C}}(\SL \tensor[]{\beth}{_{C}}))
        &&\text{see \Cref{notn:vertical-composition-and-horizontal-composition-of-2-cells}}
        \\
    &= (\tensor[]{\daleth}{_{\SG A}},\tensor[]{\daleth}{_{\SG C}})(\SL \tensor[]{\beth}{_{A}},\SL \tensor[]{\beth}{_{C}})
        &&\text{by the definition of composition in \(\BE''\dExt{\CC''}\)}
        \\
    &= \lan\daleth\tensor[]{\ran}{_{\tensor[]{\SE}{_{(\SG,\Lambda)}}(\delta)}} \tensor[]{\SE}{_{(\SL,\Phi)}}(\lan\beth\tensor[]{\ran}{_{\delta}})
        &&\text{as \(\tensor[]{\SE}{_{(\SL,\Phi)}}\) respects morphisms over \(\SL\)}
        \\
    &= (\lan\daleth\ran \circ_{h}\lan\beth\ran\tensor[]{)}{_{\delta}}
        &&\text{using the definition of \(\circ_{h}\)}
        \\
    &= ((\updave_{2}\daleth)\circ_{h}(\updave_{2}\beth)\tensor[]{)}{_{\delta}}
        &&\text{by the definition of \(\updave_{2}\).}\tag*{$\qed$}
\end{align*}
\end{proof}

As \cref{thm:2-functor} establishes that the functor \(\updave\colon \Exang{n} \to \Exactcat\) behaves just like a \(2\)-functor, it enjoys similar properties. 
For example, it is known that \(2\)-functors preserve adjunctions; see e.g.\ \cite[Prop.~6.1.7]{JohnsonYau-2-dimensional-categories}. 
Thus, we deduce the following result, which is readily shown by applying \(\updave_{2}\) to the triangle identities \eqref{eqn:triangle-identities-as-whiskerings}, 
and using that \(\updave_{2}\) preserves vertical and horizontal composition and identities of \(2\)-cells.

\begin{cor}
\label{cor:n-exangulated-adjoint-pair-gives-adjoint-pair-of-functors-on-Ext-categories}
If \(((\SF,\Gamma),(\SA,\NT))\) is an \(n\)-exangulated adjoint pair, then \((\tensor[]{\SE}{_{(\SF,\Gamma)}},\tensor[]{\SE}{_{(\SA,\NT)}})\) is an adjoint pair of exact functors.
\end{cor}

\cref{cor:finalcorollary} below yields \cref{thmx:finalcorollary} from \cref{sec:introduction}.
It follows immediately from \cref{thm:2-functor} that the restriction of \(\updave = ( \updave_{0}, \updave_{1}, \updave_{2} )\) to \(\exang{n}\) is a \(2\)-functor. 
One checks that \(\updave\) is faithful on \(1\)-cells by using that any morphism \(x\colon X\to Y\) in \(\CC\) induces a morphism \((x,0)\colon \tensor[_{X}]{0}{_{0}}\to \tensor[_{Y}]{0}{_{0}}\) of extensions. 
For an example showing that \(\updave\) is not full on \(1\)-cells, see \cref{example:not-full} below. 
The \(2\)-functor is faithful but not full on \(2\)-cells by \cref{thm:characterisation-of-n-exangulated-natural-transformations} and 
\cref{example:unbalanced-natural-transformation}, respectively.

\begin{example}
\label{example:not-full}
Let \(\CC\) be a non-zero additive category. 
Equip \(\CC\) with the split \(n\)-exangulated structure \((\CC,\BE,\fs)\), which is induced by the split \(n\)-exact structure as explained in \cref{example:split-n-exangulated-structure}. 
Hence, the category \(\BE\dExt{\CC}\) consists of objects of the form \(\tensor[_{A}]{0}{_{C}}\), one for each pair \(A,C\inn\CC\). 
Any conflation in \((\BE\dExt{\CC},\CX_{\BE})\) is isomorphic to one of the form 
\(
\begin{tikzcd}[column sep=0.5cm]
\tensor[_{A}]{0}{_{C}} \arrow{r} & \tensor[_{A\oplus A'}]{0}{_{C\oplus C'}} \arrow{r} & \tensor[_{A'}]{0}{_{C'}}
\end{tikzcd}
\) 
given by 
\((\tensor[]{\iota}{_{A}},\tensor[]{\iota}{_{C}})\) and \((\tensor[]{\pi}{_{A'}},\tensor[]{\pi}{_{C'}})\); see the discussion around \eqref{eqn:conflation-in-exact-structure-canonical-split}.

Define 
\(\SE \colon \BE\dExt{\CC}\to \BE\dExt{\CC}\)
by \(\tensor[_{A}]{0}{_{C}} \mapsto \tensor[_{C}]{0}{_{A}}\) and \((a,c) \mapsto (c,a)\). 
It is straightforward to check that \(\SE\) is a 
functor by noting that any pair of morphisms \(a\colon A\to B\) and \(c\colon C\to D\) in \(\CC\) defines a morphism \((a,c)\colon \tensor[_{A}]{0}{_{C}}\to \tensor[_{B}]{0}{_{D}} \) in  \(\BE\dExt{\CC}\). Furthermore,
a conflation of the form described above is sent under \(\SE\) to the sequence 
\(
\begin{tikzcd}[column sep=0.5cm]
\tensor[_{C}]{0}{_{A}} \arrow{r} & \tensor[_{C\oplus C'}]{0}{_{A\oplus A'}} \arrow{r} & \tensor[_{C'}]{0}{_{A'}} 
\end{tikzcd}
\) given by the morphisms \((\tensor[]{\iota}{_{C}},\tensor[]{\iota}{_{A}})\) and \((\tensor[]{\pi}{_{C'}},\tensor[]{\pi}{_{A'}})\). 
This is an element of \(\CX_{\BE}\), so  \(\SE\) is exact.

We claim that \(\SE \neq \tensor[]{\SE}{_{(\SF,\Gamma)}}\) for every \(n\)-exangulated functor \((\SF,\Gamma)\colon(\CC,\BE,\fs)\to(\CC,\BE,\fs)\). 
To see this, assume there exists some additive functor \(\SF \colon \CC \to \CC\) such that \(\SE\) respects morphisms over \(\SF\).  
Then it must be the case that \(\SF X = X\) for all \(X \inn \CC\), as 
\[
(\iden{X},\iden{X}) 
    = \SE(\iden{X},\iden{X}) 
    =
    (\SF\iden{X}, \SF\iden{X}) 
    = (\iden{\SF X},\iden{\SF X}). 
\]
Now 
consider
\(A,C\inn\CC\) with \(A\niso C\).
The identity
\( (\iden{A},\iden{C}) \colon \tensor[_{A}]{0}{_{C}} \to \tensor[_{A}]{0}{_{C}}\) 
is equal to
\(
    (\iden{\SF A},\iden{\SF C}) 
    = (\SF\iden{A}, \SF\iden{C}) 
    = \SE (\iden{A},\iden{C}) 
    = (\iden{C},\iden{A}). 
\)
This implies that \(\iden{A} = \iden{C}\), which is a contradiction.
\end{example}

\begin{cor}
\label{cor:finalcorollary}
Restriction of the assignments \(\updave_{i}\) from \textup{\cref{def:updave}} yields a \(2\)-functor 
\(
\updave = ( \updave_{0}, \updave_{1}, \updave_{2} )
    \colon \exang{n}\to\exactcat
\). 
This \(2\)-functor is faithful on \(1\)-cells and \(2\)-cells, but is full on neither \(1\)-cells nor \(2\)-cells.
\end{cor}

%%%%%%%%%%%%%%%%%%%%%%%%%%%%%%%%%%%%%%%%%%%%%%%%%%%%%%%%%%
%%%%%%%%%%%%%%%%%%%%%%%%%%%%%%%%%%%%%%%%%%%%%%%%%%%%%%%%%%

\section{Examples of \texorpdfstring{\(n\)}{n}-exangulated categories, functors and natural transformations}
\label{sec:examples-of-n-exangulated-categories-and-functors}

Let \(n\geq 1\) be an integer. 
We begin this section by recalling some known  classes of \(n\)-exangulated categories 
arising from extriangulated, \((n+2)\)-angulated and \(n\)-exact settings; see Examples~\ref{example:extriangulated-is-1-exangulated}, \ref{example:n+2-angulated-category-is-n-exangulated} and \ref{example:n-exact-category-is-n-exangulated}, respectively. 
In each of these cases, we discuss what it means for functors and natural transformations to respect the \(n\)-exangulated structure.
We then show that any additive category admits a ``smallest'' \(n\)-exangulated structure in \cref{example:split-n-exangulated-structure}.

We move on to considering \(n\)-exangulated functors for which the type of structure of the domain category \emph{differs} from that of the codomain category, providing examples of \(n\)-exangulated functors which are neither \((n+2)\)-angulated nor \(n\)-exact in general.  
Our characterisation in \cref{thm:characterisation-of-n-exangulated-functors} is applied to establish many of these examples.
In Examples~\ref{example:delta-functors} and 
\ref{example:n-exact-inside-nplus2-angulated}, 
we study structure-preserving functors from 
\(n\)-exact to \((n+2)\)-angulated categories, 
before the canonical functor from a Frobenius 
\(n\)-exangulated category to its \((n+2)\)-angulated stable category is shown to be \(n\)-exangulated in \cref{example:frobenius-n-exact-to-n-angulated-quotient}. 
Additionally, we demonstrate how the relative theory of \(n\)-exangulated categories 
can be used to equip a triangulated category with its pure-exact extriangulated structure; 
see \cref{example:extriangulated-functor-purity-CGTcase,example:rest-yoneda-extrian-func}. 
In doing so, we show that the restricted Yoneda embedding gives an example of 
an extriangulated functor which is neither exact nor triangulated in general.

Our first three examples 
each 
consists of two parts. 
In 
part (i) we focus on \(n\)-exan\-gu\-lat\-ed functors. In part (ii) we discuss what it means for natural transformations to be \(n\)-exan\-gu\-lat\-ed.

\begin{example}
\label{example:extriangulated-is-1-exangulated}
\emph{Extriangulated} categories were introduced by Nakaoka--Palu in \cite{NakaokaPalu-extriangulated-categories-hovey-twin-cotorsion-pairs-and-model-structures} as a simultaneous generalisation of triangulated and exact categories. 
Examples, which are neither triangulated nor exact in general, include extension-closed subcategories \cite[Rem.~2.18]{NakaokaPalu-extriangulated-categories-hovey-twin-cotorsion-pairs-and-model-structures} and certain ideal quotients  \cite[Prop.~3.30]{NakaokaPalu-extriangulated-categories-hovey-twin-cotorsion-pairs-and-model-structures} of triangulated categories. 
A category is extriangulated if and only if it is \(1\)-exangulated \cite[Prop.~4.3]{HerschendLiuNakaoka-n-exangulated-categories-I-definitions-and-fundamental-properties}. 
Thus, we obtain a plethora of categories with interesting \(n\)-exangulated structures for \(n=1\). 
\begin{enumerate}[label=\textup{(\roman*)}]
\item  
A \(1\)-exangulated functor is also called \emph{extriangulated} \cite[Def.~2.32]{Bennett-TennenhausShah-transport-of-structure-in-higher-homological-algebra}.
In \cref{example:delta-functors,example:rest-yoneda-extrian-func}, 
we exhibit extriangulated functors from work of Keller \cite{Keller-derived-categories-and-universal-problems} and Krause \cite{Krause-Smashing-subcategories}, respectively.

\item Following (i), by an \emph{extriangulated natural transformation} we refer to the case \(n=1\) in \cref{def:n-exangulated-natural-transformation}.
\emph{Morphisms of extriangulated functors} were introduced by Nakaoka--Ogawa--Sakai \cite[Def.~2.11(3)]{NakaokaOgawaSakai-localization-of-extriangulated-categories}. 
The equation \cite[(2.2)]{NakaokaOgawaSakai-localization-of-extriangulated-categories} defining such morphisms is precisely \eqref{def:n-exangulated-natural-transformation} in the case \(n=1\), and so this notion 
coincides with that of an extriangulated natural transformation. 
In \cref{example:delta-functors} we exhibit extriangulated natural transformations which arise in 
\cite{Keller-derived-categories-and-universal-problems}.
\end{enumerate}
\end{example}

In Examples \ref{example:n+2-angulated-category-is-n-exangulated} and \ref{example:n-exact-category-is-n-exangulated}, we use the following notation.

\begin{notn}
\label{notn:first_examples_notation}
Suppose that \((\CC,\BE,\fs)\) and \((\CC',\BE',\fs')\) are \(n\)-exangulated categories, and let \(\beth\colon\SF\Rightarrow\SG\) be a natural transformation of additive functors \(\CC\to\CC'\). 
For \(
\delta\inn\BE(C,A)
\) 
with \(\fs(\delta)= [X^{\combul}]\), 
we note that setting
\(\tensor*[]{\beth}{_{X}^{i}} 
    \deff \tensor*[]{\beth}{_{X^{i}}}
\) for each \(i\) defines a morphism of complexes 
\(
\tensor*[]{\beth}{_{X}^{\combul}} \colon \SF_{\com}X^{\combul}\to \SG_{\com}X^{\combul}
\) by the naturality of \(\beth\). 
In particular, we have
\(\tensor*[]{\beth}{_{X}^{0}} 
    = \tensor*[]{\beth}{_{A}}
\)
and 
\(\tensor*[]{\beth}{_{X}^{n+1}} 
    = \tensor*[]{\beth}{_{C}}
\). 
\end{notn}

\begin{example}
\label{example:n+2-angulated-category-is-n-exangulated}

An \emph{\((n+2)\)-angulated} category (see \cite[Def.~2.1]{GeissKellerOppermann-n-angulated-categories}) is the higher homological analogue of a triangulated category. 
Any \((n+2)\)-angulated category \((\CC,\tensor[]{\sus}{_{n}},\pent)\) 
has the structure of an \(n\)-exangulated category \((\CC,\BE_{\pent},\fs_{\pent})\); see \cite[Sec.\ 4.2]{HerschendLiuNakaoka-n-exangulated-categories-I-definitions-and-fundamental-properties}. In this case, we have 
\(
\BE_{\pent}(C,A)\deff \CC(C,\tensor[]{\sus}{_{n}} A)
\) 
and \(\fs_{\pent}(\delta)\deff [X^{\combul}]\) whenever \(\delta\inn\BE_{\pent}(X^{n+1},X^{0})= \CC(X^{n+1},\tensor[]{\sus}{_{n}} X^{0})\) completes to a distinguished \((n+2)\)-\emph{angle}
\begin{equation}
\label{eqn:nplus2-angle-in-nplus2-angulated-example}
	\begin{tikzcd}[column sep=0.6cm]
	X^{0} \arrow{r}{\tensor*[]{d}{_{X}^{0}}}& X^{1} \arrow{r}& \cdots \arrow{r}& X^{n}\arrow{r}{\tensor*[]{d}{_{X}^{n}}}& X^{n+1}\arrow{r}{\delta}& \tensor[]{\sus}{_{n}} X^{0}.
	\end{tikzcd}
\end{equation}

\begin{enumerate}[label=\textup{(\roman*)}]
    \item For \((n+2)\)-angulated categories \((\CC,\tensor[]{\sus}{_{n}},\pent)\) and \((\CC',\tensor*[]{\sus}{^{\prime}_{n}},\pent')\), the notion of an \((n+2)\)-angulated functor \(\SF\colon \CC\to\CC'\) was introduced by Bergh--Thaule 
\cite[Sec.\ 4]{BerghThaule-the-grothendieck-group-of-an-n-angulated-category}. 
The functor \(\SF\) is \emph{\((n+2)\)-angulated} if it is additive and comes
equipped with a natural isomorphism \(\Theta\colon \SF \tensor[]{\sus}{_{n}}
\Rightarrow \tensor*[]{\sus}{^{\prime}_{n}}\SF\) such that for any distinguished \((n+2)\)-angle of the form \eqref{eqn:nplus2-angle-in-nplus2-angulated-example} one obtains a distinguished \((n+2)\)-angle 
\begin{center}
    \(
	\begin{tikzcd}[column sep=1cm]
	\SF X^{0} \arrow{r}{\SF\tensor*[]{d}{_{X}^{0}}}& \SF X^{1} \arrow{r}& \cdots \arrow{r}& \SF X^{n}\arrow{r}{\SF\tensor*[]{d}{_{X}^{n}}}& \SF X^{n+1}\arrow{rr}{\Theta_{X_{0}}\circ\SF\delta}&& \tensor*[]{\sus}{^{\prime}_{n}}\SF X^{0}.
	\end{tikzcd}
	\)
\end{center}

For \(\SF\) and \(\Theta\) as above, setting \(\tensor[]{\Gamma}{_{(X^{n+1},X^0)}}(\delta) = \Theta_{X_{0}}\circ\SF\delta\) defines a natural transformation \(\Gamma\colon \CC(-,\tensor[]{\sus}{_{n}}-)\Rightarrow \CC'(\SF-,\tensor*[]{\sus}{^{\prime}_{n}}\SF-)\). 
In \cite[Thm.~2.33]{Bennett-TennenhausShah-transport-of-structure-in-higher-homological-algebra}, it is shown that 
the existence of 
an \((n+2)\)-angulated functor \(\SF\) with natural isomorphism \(\Theta\) is equivalent to the existence of an \(n\)-exangulated functor \((\SF,\Gamma)\) from
\((\CC,\BE_{\pent},\fs_{\pent})\) to \((\CC',\BE_{\pent'},\fs_{\pent'})\).
\item 
As noted in \cite[Rem.~2.12]{NakaokaOgawaSakai-localization-of-extriangulated-categories}, the notion of an extriangulated natural transformation is equivalent to the definition of a \emph{morphism of triangulated functors} in the sense of Kashiwara--Schapira \cite[Def.~10.1.9(ii)]{Kashiwara-Schapira-Categories-and-sheaves}, whenever the extriangulated categories involved correspond to triangulated categories.

Keeping the notation from above, suppose that \((\SF,\Gamma)\) and \((\SG,\Lambda)\) are \(n\)-exangulated functors \((\CC,\BE_{\pent},\fs_{\pent}) \to (\CC',\BE_{\pent'},\fs_{\pent'})\) corresponding to \((n+2)\)-angulated functors. Note that a natural transformation \(\beth\colon\SF \Rightarrow\SG\) satisfies    
\eqref{eqn:n-exangulated-natural-transformation-property} 
if and only if 
each 
square 
\begin{center}
\(
\begin{tikzcd}[column sep=1.3cm, row sep=0.7cm]
\SF X^{n+1} \arrow{d}[swap]{\beth_{X^{n+1}}} \arrow{rr}{\tensor[]{\Gamma}{_{(X^{n+1},X^0)}}(\delta)} & & \sus'_{n}\SF X^0 \arrow{d}{\sus'_{n}\beth_{X^0}}\\
\SG X^{n+1} \arrow{rr}[swap]{\tensor[]{\Lambda}{_{(X^{n+1},X^0)}}(\delta)} & & \sus'_{n}\SG X^0
\end{tikzcd}
\)
\end{center}
commutes, in which case the sequence 
\(
(\tensor*[]{\beth}{_{X}^{1}},\dots,\tensor*[]{\beth}{_{X}^{n}})
\) 
defines a \emph{morphism of} \((n+2)\)-\(\sus'_{n}\)-\emph{sequences} in \(\CC'\) in the sense of 
\cite[Def.~2.1]{GeissKellerOppermann-n-angulated-categories}. 
\end{enumerate}
\end{example}

Parallel to the \((n+2)\)-angulated story is the \(n\)-exact one.

\begin{example}
\label{example:n-exact-category-is-n-exangulated}
Higher versions of abelian and exact categories 
were introduced in \cite{Jasso-n-abelian-and-n-exact-categories}. 
Analogously to the classical theory, every \emph{\(n\)-abelian} category carries 
an \emph{\(n\)-exact} structure 
\cite[Thm.~4.4]{Jasso-n-abelian-and-n-exact-categories}. 
Any skeletally small \(n\)-exact category \((\CC,\CX)\), where \(\CX\) is the collection of \emph{admissible} \(n\)-exact sequences, 
gives rise to an \(n\)-exangulated category \((\CC,\BE_{\CX},\fs_{\CX})\); see \cite[Sec.\ 4.3]{HerschendLiuNakaoka-n-exangulated-categories-I-definitions-and-fundamental-properties}. 
In this case, we have
\(
\BE_{\CX}(X^{n+1},X^{0}) \deff 
\Set{ [X^{\combul}] 
| X^{\combul}\inn\CX }
\), 
where \([X^{\combul}]\) is the equivalence class with respect to the homotopy relation \(\sim\) described in \cref{sec:n-exangulated-categories}. 

\begin{enumerate}[label=\textup{(\roman*)}]

    \item For \(n\)-exact categories \((\CC,\CX)\) and \((\CC',\CX')\), an additive functor \(\SF\colon \CC\to\CC'\) is called \mbox{\emph{\(n\)-exact}} provided \(\SF_{\com} X^{\combul}\inn\CX'\) whenever \(X^{\combul}\inn\CX\)  \cite[Def.~2.18]{Bennett-TennenhausShah-transport-of-structure-in-higher-homological-algebra}.  
    Assuming that the categories involved are skeletally small and considering the associated \(n\)-exangulated structures, the existence of an \(n\)-exact functor \(\SF\) is equivalent to the existence of an \(n\)-exangulated functor \((\SF,\Gamma)\) from \((\CC,\BE_{\CX},\fs_{\CX})\) to \((\CC',\BE_{\CX'},\fs_{\CX'})\); see \cite[Thm.~2.34]{Bennett-TennenhausShah-transport-of-structure-in-higher-homological-algebra}. 
    In this case, the natural transformation \(\Gamma\) is uniquely defined by 
    \(\tensor[]{\Gamma}{_{(X^{n+1},X^{0})}}([X^{\combul}]) = [\SF_{\com} X^{\combul}]\).

    \item Keeping the notation from above, suppose that \(\SF\) and \(\SG\) are \(n\)-exact functors \(\CC\to\CC'\) and consider a natural transformation \(\beth\colon\SF\Rightarrow\SG\).
    As noted in 
    \cite[Rem.~2.12]{NakaokaOgawaSakai-localization-of-extriangulated-categories},
    equation \eqref{eqn:n-exangulated-natural-transformation-property} is automatically satisfied for the corresponding \(n\)-exangulated functors in the case \(n=1\). We now show that this holds for any integer \(n\geq 1\). 

Consider the equivalence class \(\delta = [X^{\combul}]\inn\BE_{\CX}(C,A)\) of an admissible \(n\)-exact sequence \(X^{\combul} \inn \CX\). As \(\SF\) and \( \SG\) are \(n\)-exact, we have \(\SF_{\com}X^{\combul}, \SG_{\com}X^{\combul}\inn\CX'\), so \(
\tensor*[]{\beth}{_{X}^{\combul}}
        \colon 
    \SF_{\com}X^{\combul}
        \to
    \SG_{\com}X^{\combul}
\) 
is a morphism of admissible \(n\)-exact sequences. By the existence of \(n\)-pushout diagrams in the \(n\)-exact category \((\CC',\CX')\), we obtain a morphism 
\(
p^{\combul}
    \colon \SF_{\com} X^{\combul} \to (\tensor[]{\beth}{_{A}}\tensor[]{)}{_{\BE_{\CX'}}} \SF_{\com} X^{\combul}
\) with 
\(p^{0} = \tensor[]{\beth}{_{A}}\) 
and 
\(p^{n+1} = \iden{\SF C}\); see  \cite[Def.~4.2(E2) and Prop.~4.8]{Jasso-n-abelian-and-n-exact-categories}. Dually, there is also a morphism 
\(
q^{\combul}
    \colon (\tensor[]{\beth}{_{C}}\tensor[]{)}{^{\BE_{\CX'}}}\SG_{\com} X^{\combul}
    \to \SG_{\com} X^{\combul}
\) with 
\(q^{0} = \iden{\SG A}\) 
and 
\(q^{n+1} = \tensor[]{\beth}{_{C}}\). By \cite[Prop.~4.9]{Jasso-n-abelian-and-n-exact-categories}, there exists a morphism
\(
l^{\combul}
    \colon (\tensor[]{\beth}{_{A}}\tensor[]{)}{_{\BE_{\CX'}}}\SF_{\com} X^{\combul} \to \SG_{\com} X^{\combul}
\) satisfying 
\(
l^{\combul}p^{\combul} \sim \tensor*[]{\beth}{_{X}^{\combul}}
\), 
\(l^{0} = \iden{\SG A}\) 
and 
\(l^{n+1} = \tensor[]{\beth}{_{C}}\). 
On the other hand, the dual of \cite[Prop.~4.9]{Jasso-n-abelian-and-n-exact-categories} yields a morphism 
\(
m^{\combul}
    \colon (\tensor[]{\beth}{_{A}}\tensor[]{)}{_{\BE_{\CX'}}}\SF_{\com} X^{\combul} \to (\tensor[]{\beth}{_{C}}\tensor[]{)}{^{\BE_{\CX'}}}\SG_{\com} X^{\combul}
\) with
\(
q^{\combul}m^{\combul} \sim l^{\combul}
\), 
\(m^{0} = l^{0} = \iden{\SG A}\) 
and 
\(m^{n+1} = \iden{\SF C}\).
Note that \(m^{\combul}\) is an \emph{equivalence} of \(n\)-exact sequences in the sense of \cite[Def.~2.9]{Jasso-n-abelian-and-n-exact-categories}. This implies that
\(
(\tensor[]{\beth}{_{A}}\tensor[]{)}{_{\BE_{\CX'}}} [ \SF_{\com} X^{\combul} ]
    = [(\tensor[]{\beth}{_{A}}\tensor[]{)}{_{\BE_{\CX'}}}\SF_{\com} X^{\combul}] 
    = [(\tensor[]{\beth}{_{C}}\tensor[]{)}{^{\BE_{\CX'}}}\SG_{\com} X^{\combul}]
    = (\tensor[]{\beth}{_{C}}\tensor[]{)}{^{\BE_{\CX'}}} [ \SG_{\com} X^{\combul} ]
\)
by \cite[Prop.~4.10]{Jasso-n-abelian-and-n-exact-categories}, which verifies \eqref{eqn:n-exangulated-natural-transformation-property}.
    
\end{enumerate}
\end{example}

Our next example shows that any additive category \(\CC\) carries a smallest \(n\)-exangulated structure \((\CC,\BE_{n\text{-split}},\fs_{n\text{-split}})\) arising from an \(n\)-exact structure for each integer \(n \geq 1\). 
In particular, we have that  \((\CC,\BE_{n\text{-split}},\fs_{n\text{-split}})\) is an \(n\)-exangulated subcategory, in the sense of \cite[Def.~3.7]{Haugland-the-grothendieck-group-of-an-n-exangulated-category}, of any \(n\)-exangulated structure \((\CC,\BE,\fs)\) we impose on \(\CC\). 
Recall that an \emph{\(n\)-exangulated subcategory} of \((\CC,\BE,\fs)\) is an isomorphism-closed subcategory \(\CA\) of \(\CC\) with an \(n\)-exangulated structure \((\CA,\BE',\fs')\) for which the inclusion \(\SF\) of \(\CA\) in \(\CC\) gives an \(n\)-exangulated functor 
\((\SF,\Gamma)\colon(\CA,\BE',\fs')\to(\CC,\BE,\fs)\) where each  \(\tensor[]{\Gamma}{_{(C,A)}}\) is an inclusion of abelian groups.

\begin{example}
\label{example:split-n-exangulated-structure}
By \cite[Rem.~4.7]{Jasso-n-abelian-and-n-exact-categories}, any additive category \(\CC\) admits a smallest \(n\)-exact structure \((\CC,\CX_{n\text{-split}})\), where
\(\CX_{n\text{-split}}\) denotes the class of all
contractible 
\(n\)-exact sequences. 
Recall that a sequence is \emph{contractible} if it is homotopy equivalent to the zero complex. We call \(\CX_{n\text{-split}}\) the \emph{split \(n\)-exact structure} of \(\CC\).  
Notice that for all \(A,C\inn\CC\), any contractible \(n\)-exact sequence starting in \(A\) and ending in \(C\) is homotopic to
\begin{equation}\label{eqn:split-n-exact-seq}
\begin{tikzcd}[column sep=0.9cm]
    A\arrow{r}{\iden{A}}
        &A\arrow{r}
        &0\arrow{r}
        &\cdots \arrow{r}
        &0 \arrow{r}
        & C \arrow{r}{\iden{C}}
        & C.
\end{tikzcd}
\end{equation}
Note that in the case $n=1$, the sequence \eqref{eqn:split-n-exact-seq} is of the form 
$
\begin{tikzcd}[column sep=1.2cm, ampersand replacement=\&]
    A\arrow{r}{\begin{psmallmatrix}\iden{A} \\ 0\end{psmallmatrix}}
		\& A\oplus C \arrow{r}{\begin{psmallmatrix}0 & \iden{C}\end{psmallmatrix}}
		\& C.
\end{tikzcd}
$
As seen in \cref{example:n-exact-category-is-n-exangulated}, the \(n\)-exact category \((\CC,\CX_{n\text{-split}})\) yields an \(n\)-exangulated structure \((\CC,\BE_{n\text{-split}},\fs_{n\text{-split}})\), which we call the \emph{split \(n\)-exangulated structure} of \(\CC\). 
Note that no set-theoretic issues arise as 
\(\BE_{n\text{-split}}(C,A) 
    = \{ \tensor[_{A}]{0}{_{C}} \}
\) 
is the trivial abelian group for all \(A,C\inn\CC\).

We claim that this is the smallest \(n\)-exangulated structure on \(\CC\). To see this, suppose that \((\CC,\BE,\fs)\) is an \(n\)-exangulated category. Consider the identity functor \(\idfunc{\CC}\colon\CC\to\CC\) and the natural transformation 
\(
\Gamma 
    \colon 
\BE_{n\text{-split}} 
    \Rightarrow 
    \BE\) 
given by 
\(\tensor[]{\Gamma}{_{(C,A)}} (\tensor[_{A}]{0}{_{C}})
    = \tensor[_{A}]{0}{_{C}}\). 
Note that, by \ref{R2} and 
\cite[Prop.~3.3]{HerschendLiuNakaoka-n-exangulated-categories-I-definitions-and-fundamental-properties}, 
we have 
\[
\fs_{n\text{-split}}(\tensor[_{A}]{0}{_{C}}) 
    =\fs(\tensor[_{A}]{0}{_{C}})
    =[\begin{tikzcd}[column sep=0.9cm]
    A\arrow{r}{\iden{A}}
        &A\arrow{r}
        &0\arrow{r}
        &\cdots \arrow{r}
        &0 \arrow{r}
        & C \arrow{r}{\iden{C}}
        & C
    \end{tikzcd}].
\]
Consequently, the pair \((\idfunc{\CC},\Gamma) \colon (\CC,\BE_{n\text{-split}},\fs_{n\text{-split}}) \to (\CC,\BE,\fs)\) is an \(n\)-exangulated functor. As \(\tensor[]{\Gamma}{_{(C,A)}}\) is an inclusion for all \(A,C \inn \CC\), this implies that \((\CC,\BE_{n\text{-split}},\fs_{n\text{-split}})\) is an \(n\)-exangulated subcategory of \((\CC,\BE,\fs)\). 
In particular, the split \(n\)-exangulated structure is the smallest \mbox{\(n\)-exangulated} structure we can impose on \(\CC\).
\end{example}

In \cref{example:delta-functors} we describe extriangulated functors from an exact or abelian category to a triangulated category. Some of these examples are classical and others have been of very recent interest. 
The authors would like to thank Peter J{\o{}}rgensen for pointing out Linckelmann \cite[Rem.~6.8]{Linckelmann-on-abelian-subcategories-of-triangulated-categories}, which is used in (ii) below.

\begin{example}
\label{example:delta-functors}
Structure-preserving functors from exact to triangulated categories, or more generally from exact to \emph{suspended} categories in the sense of 
Keller--Vossieck \cite{KellerVossieck-sous-les-categories-derivees}, have been considered in the literature previously. 
Such functors are called \emph{\(\delta\)-functors} 
in 
\cite[pp.\ 701--702]{Keller-derived-categories-and-universal-problems}. 
Let \(\CA\) be an exact category for which \(\Ext^{1}_{\CA}(C,A)\) is a set for all \(A,C\inn\CA\). 
Consider a triangulated category \(\CC\) with suspension functor \(\sus\).
Suppose that \(\SF\colon \CA\to\CC\) is an additive functor and that \(\Gamma\colon \Ext^{1}_{\CA}(-,-) \Rightarrow \CC(\SF-,\sus\SF-)\) is a natural transformation. 
The pair \((\SF,\Gamma)\) is called a \emph{\(\delta\)-functor} if 
each conflation 
\begin{equation}
\label{eqn:delta-functor-exact-sequence}
\begin{tikzcd}
0 \arrow{r}
    & A \arrow{r}{f}
    & B \arrow{r}{g}
    & C \arrow{r}
    &0
\end{tikzcd}
\end{equation}
in \(\CA\) is sent to a distinguished triangle 
in \(\CC\) 
of the form
\begin{equation}
\label{eqn:delta-functor-distinguished-triangle}
\begin{tikzcd}[column sep=1.6cm]
\SF A\arrow{r}{\SF f}
&\SF B\arrow{r}{\SF g}
&\SF C\arrow{r}{\tensor[]{\Gamma}{_{(C,A)}}(\rho)}
&\sus\SF A,
\end{tikzcd}
\end{equation}
where \(\rho\) denotes the equivalence class of the conflation (\ref{eqn:delta-functor-exact-sequence}).
Before translating the language of \(\delta\)-functors to that of extriangulated functors, we recall some examples. 
\begin{enumerate}[label=\textup{(\roman*)}]
    
    \item \label{item:canonical-functor-is-delta} 
    Let \(\CA\) 
    be a skeletally small exact category and  \(\CC\) 
    its \emph{derived} category 
    with suspension functor \(\sus\); 
    see Keller \cite[p.\ 692]{Keller-derived-categories-and-their-uses}. 
    If \(\SF \colon \CA \to \CC\) denotes the canonical inclusion, then there is a natural transformation 
    \(\Gamma\colon \Ext^{1}_{\CA}(-,-) \Rightarrow\CC(\SF-,\sus\SF-)\) such that the equivalence class of a conflation 
    \eqref{eqn:delta-functor-exact-sequence} is sent to a distinguished triangle \eqref{eqn:delta-functor-distinguished-triangle}; see \cite[pp.\ 701--702]{Keller-derived-categories-and-universal-problems}. 
The pair \((\SF,\Gamma)\) is a \(\delta\)-functor \(\CA\to\CC\). 
    
    \item\label{item:distinguished-abelian-is-delta} 
    Suppose \(\CC\) 
    is a triangulated category with suspension functor \(\sus\). 
    Let \(\SF\colon\CA\to\CC\) denote the 
    inclusion 
    of an 
    isomorphism-closed additive subcategory \(\CA\sse\CC\). 
    Assume that \(\CA\) is 
    a \emph{distinguished abelian subcategory}
    in the sense of  \cite[Def.~1.1]{Linckelmann-on-abelian-subcategories-of-triangulated-categories}. 
    This means that \(\CA\) is abelian and, given a short exact sequence \eqref{eqn:delta-functor-exact-sequence} in \(\CA\), there exists a morphism \(\tensor[]{\Gamma}{_{(C,A)}}(\rho) \colon \SF C \to \sus\SF A\) in \(\CC\) such that
\eqref{eqn:delta-functor-distinguished-triangle} 
is a distinguished triangle in \(\CC\). 
A morphism of short exact sequences in the abelian category \(\CA\) uniquely determines a morphism of distinguished triangles in the triangulated category \(\CC\); see \cite[Rem.~6.8]{Linckelmann-on-abelian-subcategories-of-triangulated-categories}. 
This implies that 
\mbox{\(
\Gamma=\{\tensor[]{\Gamma}{_{(C,A)}}\tensor[]{\}}{_{(C,A)\inn\CA^{\scalebox{0.7}{\op}}\times\CA}}
    \colon \Ext^{1}_{\CA}(-,-) \Rightarrow \CC(\SF-,\sus\SF-)
\)}
is a natural transformation. 
The pair \((\SF,\Gamma)\) is hence a \(\delta\)-functor \(\CA\to\CC\). 
Similarly, the inclusions of \emph{proper abelian subcategories} (see J{\o{}}rgensen \cite[Def.~1.2]{Jorgensen-abelian-subcategories-of-triangulated-categories-induced-by-simple-minded-systems}) and \emph{admissible abelian subcategories} 
(see Be{\u{\i}}linson--Bernstein--Deligne \cite[Def.~1.2.5]{BeilinsonBernsteinDeligne-perverse-sheaves}) into their ambient triangulated categories give rise to \(\delta\)-functors; see \cite[Rem.~1.3]{Jorgensen-abelian-subcategories-of-triangulated-categories-induced-by-simple-minded-systems}.
\end{enumerate}
Consider the extriangulated categories that arise from \(\CA\) and \(\CC\) being exact and triangulated, respectively. 
Following \cref{example:n+2-angulated-category-is-n-exangulated} and \cref{example:n-exact-category-is-n-exangulated} with \(n=1\), 
we obtain extriangulated categories \((\CA,\BE,\fs)\) and \((\CC,\BE', \fs')\) with \( \BE\deff\Ext^{1}_{\CA}(-,-)\) and 
\(\BE'\deff \CC(-,\sus-)\). 
The pair \((\SF,\Gamma)\) is a \(\delta\)-functor  \(\CA\to\CC\) if and only if it is an extriangulated functor \((\CA,\BE,\fs)\to(\CC,\BE', \fs')\).

Furthermore, \emph{morphisms of \(\delta\)-functors} 
are also defined and studied in \cite[p.~702]{Keller-derived-categories-and-universal-problems}. 
If \((\SF,\Gamma)\) and \((\SG,\Lambda)\) are \(\delta\)-functors \(\CA\to\CC\), 
then a morphism 
\((\SF,\Gamma) \Rightarrow (\SG,\Lambda)\) 
is given by a natural transformation 
\(\beth\colon \SF \Rightarrow \SG\) 
satisfying 
\(
(\sus\tensor[]{\beth}{_{A}})\circ \tensor[]{\Gamma}{_{(C,A)}}(\rho)
    = \tensor*[]{\Lambda}{_{(C,A)}}(\rho)\circ \tensor[]{\beth}{_{C}}
\) 
for each extension 
\(\rho\inn \BE(C,A)\).
This defining condition is precisely \eqref{eqn:n-exangulated-natural-transformation-property} in the context of this example. Thus, the notion of a morphism of \(\delta\)-functors coincides with that of 
an extriangulated natural transformation. 
\end{example}

\begin{rem}
The prototypical example of a \(\delta\)-functor is the canonical embedding of an exact category into its derived category; see Example \ref{example:delta-functors}\ref{item:canonical-functor-is-delta}.
Morphisms of \(\delta\)-functors were studied in \cite{Keller-derived-categories-and-universal-problems} in the search for some universal property of the derived category. 
Thus, interesting examples of \(n\)-exangulated natural transformations with \(n>1\) may arise from a generalisation of this
to a higher-dimensional setting. 
This would require the construction of a derived category of an \(n\)-exact category---a problem for which there seems so far to be no obvious solution; see
Jasso--K{\"{u}}lshammer \cite{JassoKulshammer-the-naive-approach-for-constructing-the-derived-category-of-a-d-abelian-category-fails}.
\end{rem}

Although \(\delta\)-functors provide formal language 
to express what it means for functors to send conflations to distinguished triangles in a functorial way, 
it has two apparent limitations. 
First, in its current form this notion cannot be used 
in higher homological algebra. Second, a \(\delta\)-functor must go from an exact category to a triangulated (or suspended) category, but not vice versa. 
The language of \(n\)-exangulated functors addresses both these limitations, as we will see across Examples~\ref{example:n-exact-inside-nplus2-angulated}, 
\ref{example:frobenius-n-exact-to-n-angulated-quotient} 
and 
\ref{example:rest-yoneda-extrian-func}.

In the next example, we show how recent work of Klapproth \cite{Klapproth-n-exact-categories-arising-from-nplus2-angulated-categories} (see also \cite{HeZhou-n-exact-categories-arising-from-n-exangulated-categories, Zhou-n-extension-closed-subcategories-of-nplus2-angulated-categories}) produces examples of \(n\)-exangulated functors from an \(n\)-exact category into an ambient \((n+2)\)-angulated category. 
This also gives more examples of \(n\)-exangulated subcategories.

\begin{example}
\label{example:n-exact-inside-nplus2-angulated}
Let \((\CC,\tensor[]{\sus}{_{n}},\pent)\) be a Krull--Schmidt \((n+2)\)-angulated category 
and consider the \(n\)-exangulated structure  
\((\CC,\BE_{\pent},\fs_{\pent})\) described in 
\cref{example:n+2-angulated-category-is-n-exangulated}. 
Let \(\CA\) be a 
subcategory of \(\CC\) which is closed under direct sums and summands, and \emph{closed under \(n\)-extensions}, meaning that 
for all \(A,C\inn\CA\) and each \(\delta\inn\CA(C,\tensor[]{\sus}{_{n}}A)\) there is a distinguished \((n+2)\)-angle
\begin{equation}
\label{eqn:Klapproth-nplus2-angle}
	\begin{tikzcd}%[column sep=0.6cm]
	A \arrow{r}{\tensor*[]{d}{_{X}^{0}}}& X^{1} \arrow{r}& \cdots \arrow{r}& X^{n}\arrow{r}{\tensor*[]{d}{_{X}^{n}}}& C\arrow{r}{\delta}& \tensor[]{\sus}{_{n}} A
	\end{tikzcd}
\end{equation}
in \(\CC\) 
with \(X^{i}\inn\CA\) for \(i=1,\ldots,n\); see \cite[Def.~1.1]{Klapproth-n-exact-categories-arising-from-nplus2-angulated-categories}. 
This also 
implies \(\CA\) is an \emph{extension closed} subcategory of 
\((\CC,\BE_{\pent},\fs_{\pent})\)
in the sense of \cite[Def.~4.1]{HerschendLiuNakaoka-n-exangulated-categories-II}.

Suppose, moreover, that \(\CA(\tensor[]{\sus}{_{n}}C,A)=0\) for any \(A,C\inn\CA\). 
Following \cite[Sec.\ 3]{Klapproth-n-exact-categories-arising-from-nplus2-angulated-categories}, an \(\CA\)-\emph{conflation} is a complex 
\(
X^{\combul}
    \colon\;
    \begin{tikzcd}[column sep=0.6cm]
	A \arrow{r}
	& X^{1} \arrow{r}
	& \cdots \arrow{r}
	& X^{n}
	\arrow{r}
	& C
	\end{tikzcd}
\)
that forms part of a distinguished \((n+2)\)-angle \eqref{eqn:Klapproth-nplus2-angle}. 
Let \(\BE_{\CA}\deff\CA(-,\tensor[]{\sus}{_{n}}-)\) be the restriction of \(\BE_{\pent}\) to \(\CA^{\op}\times \CA\). 
Restricting 
\(\fs_{\pent}\) 
defines an exact realisation \(\fs_{\CA}\) of \(\BE_{\CA}\) by \cite[Prop.~4.2(1)]{HerschendLiuNakaoka-n-exangulated-categories-II}. 
In an \(\CA\)-conflation \(X^{\combul}\) as above, the morphism \(\tensor*[]{d}{_{X}^{0}}\) is called an \(\CA\)-\emph{inflation} and \(\tensor*[]{d}{_{X}^{n+1}}\) is called an \(\CA\)-\emph{deflation}. 
By \cite[Lem.~3.8]{Klapproth-n-exact-categories-arising-from-nplus2-angulated-categories} and its dual, we have that both \(\CA\)-inflations and \(\CA\)-deflations are closed under composition.  
This implies that 
\((\CA,\BE_{\CA},\fs_{\CA})\) is an \(n\)-exangulated category by \cite[Prop.~4.2(2)]{HerschendLiuNakaoka-n-exangulated-categories-II}. 
Furthermore, it is an \(n\)-exangulated subcategory of \((\CC,\BE_{\pent},\fs_{\pent})\); see \cite[Exam.~3.8(1)]{Haugland-the-grothendieck-group-of-an-n-exangulated-category}.

We denote the collection of all 
\(\CA\)-conflations by \(\CX_{\CA}\), each member of which is an \(n\)-exact sequence by \cite[Lem.~3.3]{Klapproth-n-exact-categories-arising-from-nplus2-angulated-categories}. 
It follows from \cite[Thm.~I(1)]{Klapproth-n-exact-categories-arising-from-nplus2-angulated-categories} that the pair \((\CA,\CX_{\CA})\) is an \(n\)-exact category. 
For \(A,C\inn\CA\), there is an abelian group \(\yExtn{(\CA,\CX_{\CA})}{n}(C,A)\) of equivalence classes \([X^{\combul}]\) of \(\CA\)-conflations \(X^{\combul}\inn\CX_{\CA}\) with \(A=X^{0}\) and \(C=X^{n+1}\). Note that 
no set-theoretic problems arise here, since the size of \(\yExtn{(\CA,\CX_{\CA})}{n}(C,A)\) is bounded by the size of the set \(\CA(C,\tensor[]{\sus}{_{n}}A)\). 
Hence, this gives 
a biadditive functor \(\BE\deff\yExtn{(\CA,\CX_{\CA})}{n}\colon \CA^{\op}\times\CA \to \Ab\) and an \(n\)-exangulated category 
\((\CA,\BE,\fs)\) as in \cref{example:n-exact-category-is-n-exangulated}. 
By \cite[Thm.~I(2)]{Klapproth-n-exact-categories-arising-from-nplus2-angulated-categories}, there is a natural isomorphism 
\(
\Gamma
    \colon
    \BE
        \Rightarrow 
    \BE_{\CA}
\) 
given by 
\(
\Gamma([X^{\combul}]) 
    = \delta
\), 
where \(X^{\combul}\) is part of \eqref{eqn:Klapproth-nplus2-angle}. 
Thus, the pair 
\(
(\iden{\CA},\Gamma)
    \colon (\CA,\BE,\fs)
        \to 
    (\CA,\BE_{\CA},\fs_{\CA})
\) 
is an \(n\)-exangulated equivalence by \cref{prop:characterisation-of-n-exangulated-equivalence}. 
\end{example}

\emph{Frobenius exact categories} are studied in Happel \cite[Sec.\ I.2]{Happel-triangulated-cats-in-rep-theory}, and their higher analogues were introduced in \cite[Sec.\ 5]{Jasso-n-abelian-and-n-exact-categories}. 
In these setups, the quotient functor from a Frobenius exact (resp.\ \(n\)-exact) category to its stable category sends admissible exact (resp.\ \(n\)-exact) sequences to distinguished triangles (resp.\ \((n+2)\)-angles). 
In \cref{example:frobenius-n-exact-to-n-angulated-quotient} we follow the terminology introduced by Liu--Zhou \cite[Def.\ 3.2]{LiuZhou-frobenius-n-exangulated-categories}, and we show that these aforementioned quotient functors are instances of extriangulated (resp.\ \(n\)-exan\-gu\-lat\-ed) functors. 

\begin{example}
\label{example:frobenius-n-exact-to-n-angulated-quotient}
Let \((\CC,\BE,\fs)\) be an \(n\)-exangulated category. An object \(I\inn\CC\) is called \emph{\(\BE\)-injective} if for each distinguished \(n\)-exangle 
\(\lan X^{\combul},\delta\ran\), which is depicted as
\begin{equation}
\label{eqn:frobenius-delta-exangle}
\begin{tikzcd}
X^{0} \arrow{r}{\tensor*[]{d}{_{X}^{0}}}& X^{1} \arrow{r}{\tensor*[]{d}{_{X}^{1}}}& \cdots \arrow{r}{\tensor*[]{d}{_{X}^{n-1}}}&  X^{n} \arrow{r}{\tensor*[]{d}{_{X}^{n}}}& X^{n+1} \arrow[dashed]{r}{\delta}&{},
\end{tikzcd}
\end{equation}
and for each \(x\inn\CC(X^{0},I)\), there exists \(y\inn\CC(X^{1},I)\) such that \(yd_{X}^{0}=x\). 
The category \(\CC\) is said to have \emph{enough \(\BE\)-injectives} if for any \(X^{0}\inn\CC\), there is a distinguished \(n\)-exangle 
\begin{equation}
\label{eqn:frobenius-injectives-exangle}
\begin{tikzcd}
X^{0} \arrow{r}{}& I^{1} \arrow{r}{}& \cdots \arrow{r}{}&  I^{n} \arrow{r}{}& Z \arrow[dashed]{r}{\tensor[]{\delta}{_{X^{0}}}}&{},
\end{tikzcd}
\end{equation}
where \(I^{i}\) is \(\BE\)-injective 
for \(1\leq i\leq n\). 
Dually, one defines what it means for an object of \(\CC\)  to be \emph{\(\BE\)-projective} and for the category \(\CC\) to \emph{have enough \(\BE\)-projectives}. 
If \(\CC\) has enough \(\BE\)-projectives and enough \(\BE\)-injectives, and if an object in \(\CC\) is \(\BE\)-projective if and only if it is \(\BE\)-injective, then we say that \((\CC,\BE,\fs)\) is a \emph{Frobenius} \(n\)-exangulated category.

Let \((\CC,\BE,\fs)\) be Frobenius \(n\)-exangulated. 
Denote by \(\ol{\CC}\) the stable category \(\CC / \CI\) in the sense of \cite[p.\ 169]{LiuZhou-frobenius-n-exangulated-categories}, where \(\CI\) is the 
subcategory of \(\BE\)-projective-injectives. 
Consider the canonical quotient functor \(\SQ\colon \CC \to \ol{\CC}\).
Note that \(\SQ(X) = X\) in \(\ol{\CC}\) for each \(X\inn\CC\). 
Since \((\CC,\BE,\fs)\) is Frobenius, setting 
\(SX^0 \deff Z\) 
in \eqref{eqn:frobenius-injectives-exangle} yields a well-defined autoequivalence of \(\ol{\CC}\); see \cite[Prop.~3.7]{LiuZhou-frobenius-n-exangulated-categories}.
Given any distinguished \(n\)-exangle 
\eqref{eqn:frobenius-delta-exangle}, 
there is a morphism
\begin{equation}
\label{eqn:frobenius-morphism}
\begin{tikzcd}
X^{0} \arrow{r}{\tensor*[]{d}{_{X}^{0}}}\arrow[equals]{d}& X^{1} \arrow{r}{\tensor*[]{d}{_{X}^{1}}}\arrow{d}& \cdots \arrow{r}{\tensor*[]{d}{_{X}^{n-1}}}&  X^{n} \arrow{r}{\tensor*[]{d}{_{X}^{n}}}\arrow{d}& X^{n+1} \arrow{d}{\tensor*[]{d}{_{X}^{n+1}}} \arrow[dashed]{r}{\delta}&{}\\
X^{0} \arrow{r}{}& I^{1} \arrow{r}{}& \cdots \arrow{r}{}&  I^{n} \arrow{r}{}& S X^{0} \arrow[dashed]{r}{\tensor[]{\delta}{_{X^{0}}}}&{}
\end{tikzcd}
\end{equation}
of distinguished \(n\)-exangles in \(\CC\), 
using that \(I^{1}\) is \(\BE\)-injective 
and 
\cite[Prop.~3.6(1)]{HerschendLiuNakaoka-n-exangulated-categories-I-definitions-and-fundamental-properties}. 
Furthermore, there is a natural isomorphism 
\(\BE(-,-) \iso \ol{\CC}(\SQ-,S\SQ-)\) given by 
\(
\delta \mapsto \SQ(\tensor*[]{d}{_{X}^{n+1}})
\); see \cite[Lem.~3.11]{LiuZhou-frobenius-n-exangulated-categories}. It is shown in \cite[Thm.~3.13]{LiuZhou-frobenius-n-exangulated-categories} 
(see also Zheng--Wei \cite[Prop.~3.17]{ZhengWei-nplus2-angulated-quotient-categories}) 
that there is an \((n+2)\)-angulation of \((\ol{\CC}, S)\) consisting of \((n+2)\)-angles of the form 
\begin{equation}
\label{eqn:frobenius-nplus2-angle}
\begin{tikzcd}[column sep=1.3cm]
X^{0} \arrow{r}{\SQ(\tensor*[]{d}{_{X}^{0}})}& X^{1} \arrow{r}{\SQ(\tensor*[]{d}{_{X}^{1}})}& \cdots \arrow{r}{\SQ(\tensor*[]{d}{_{X}^{n-1}})}&  X^{n} \arrow{r}{\SQ(\tensor*[]{d}{_{X}^{n}})}& X^{n+1} \arrow{r}{\SQ(\tensor*[]{d}{_{X}^{n+1}})}& S X^{0}.
\end{tikzcd}
\end{equation}
This gives an \(n\)-exangulated category 
\((\ol{\CC}, \ol{\BE}, \ol{\fs})\), 
where 
\(\ol{\BE}(X^{n+1},X^{0}) 
    = \ol{\CC}(X^{n+1}, S X^{0})
\) 
and
\[
\ol{\fs}(\SQ(\tensor*[]{d}{_{X}^{n+1}}))
    = [\begin{tikzcd}[column sep=1.3cm]
        X^{0} \arrow{r}{\SQ(\tensor*[]{d}{_{X}^{0}})}& X^{1} \arrow{r}{\SQ(\tensor*[]{d}{_{X}^{1}})}& \cdots \arrow{r}{\SQ(\tensor*[]{d}{_{X}^{n-1}})}&  X^{n} \arrow{r}{\SQ(\tensor*[]{d}{_{X}^{n}})}& X^{n+1}
        \end{tikzcd}]
\]
whenever the morphism \(\SQ(\tensor*[]{d}{_{X}^{n+1}}) \inn\ol{\CC}(X^{n+1}, S X^{0})\) fits into an \((n+2)\)-angle \eqref{eqn:frobenius-nplus2-angle}.

Using \cref{thm:characterisation-of-n-exangulated-functors}, 
it is 
straightforward 
to check that there exists a natural transformation 
\(\Gamma \colon \BE(-,-) \Longrightarrow \ol{\BE}(\SQ-,\SQ-)\), 
such that 
\((\SQ,\Gamma) \colon (\CC,\BE,\fs) \to (\ol{\CC},\ol{\BE},\ol{\fs}) \) 
is an \(n\)-exangulated functor. 
Indeed, in the notation above, 
consider the functor \(\SE\colon \BE\dExt{\CC} \to \ol{\BE}\dExt{\ol{\CC}}\) given by 
\(\SE(\delta) = \SQ(\tensor*[]{d}{_{X}^{n+1}})\) on objects and 
by
\(
\SE(a,c) 
    = (\SQ a, \SQ c)
\) 
on morphisms. 
Note that \((\SQ(f^{0}), \SQ(f^{n}))\) is a morphism in \(\ol{\BE}\dExt{\ol{\CC}}\) by \cite[Lem.~3.11]{LiuZhou-frobenius-n-exangulated-categories}. 
It is clear that \(\SE\) respects morphisms and distinguished \(n\)-exangles over \(\SQ\).
By \cref{prop:F-additive-iff-E-additive-iff-E-exact}, we see that \(\SE\) is additive, and thus \cref{thm:characterisation-of-n-exangulated-functors} applies. 
Without the theorem, 
proving the naturality of \(\Gamma\)
requires a non-trivial amount of extra work.
\end{example}

In \cref{example:extriangulated-functor-purity-CGTcase,example:rest-yoneda-extrian-func},
we use \emph{relative} \(n\)-exangulated structures. 
We provide some key definitions here, 
but 
refer the reader to \cite[Sec.\ 3.2]{HerschendLiuNakaoka-n-exangulated-categories-I-definitions-and-fundamental-properties} 
for details. 
Let \((\CC, \BE,\fs)\) be an \(n\)-exangulated category and \(\CI\) be a 
subcategory of \(\CC\). 
Then the assignment \(\tensor[]{\BE}{_{\CI}}\colon \CC^{\op}\times\CC \to \Ab\) given by 
\(
\tensor[]{\BE}{_{\CI}}(C,A) \deff 
\Set{ \delta\inn\BE(C,A) | \tensor[^{\BE}]{\delta}{_{X}} = 0 \text{ for all } X\inn\CI}
\)
defines a subfunctor of \(\BE\). 
The restriction \(\tensor[]{\fs}{_{\CI}}\) of \(\fs\) to the extensions \(\delta\inn\tensor[]{\BE}{_{\CI}}(C,A)\) for \(A,C\inn\CC\) is an exact realisation of \(\tensor[]{\BE}{_{\CI}}\). 
Moreover, it follows from 
\cite[Props.~3.16, 3.19]{HerschendLiuNakaoka-n-exangulated-categories-I-definitions-and-fundamental-properties} that 
\((\CC,\tensor[]{\BE}{_{\CI}},\tensor[]{\fs}{_{\CI}})\) is an \(n\)-exangulated category, 
and 
by \cite[Thm.~2.12]{JorgensenShah-grothendieck-groups-of-d-exangulated-categories-and-a-modified-CC-map} that 
it is an \(n\)-exan\-gu\-lat\-ed subcategory of 
\((\CC,\BE,\fs)\).

The next two examples concern compactly generated triangulated categories. We recall relevant definitions in \cref{example:extriangulated-functor-purity-CGTcase}; for more details, see work of the first author \cite{Bennett-Tennenhaus-some-characterisations}, 
Garkusha--Prest \cite{Garkusha-Prest-Ziegler}, 
Krause \cite{Krause-Smashing-subcategories}, 
Neeman \cite{Neeman-the-grothendieck-duality-theorem-via-bousfields-techniques-and-brown-representability}, 
and 
Prest \cite{Prest-purity-spectra-localisation}. 
In this example we recall how equipping a compactly generated triangulated category with its class of pure-exact triangles results in an extriangulated substructure of the triangulated structure. 
This was first noted in 
Hu--Zhang--Zhou \cite[Rem.~3.3]{HuZhangZhou-proper-classes-and-gorensteinness-in-extriangulated-categories} from a different perspective. 
In \cref{example:rest-yoneda-extrian-func} we show that a certain restricted Yoneda functor is extriangulated, and preserves and reflects injective 
objects.

\begin{example}
\label{example:extriangulated-functor-purity-CGTcase}
Let \(\CC\) be a triangulated category with suspension functor \(\sus\). Then \(\CC\) has the structure of an extriangulated category, which we denote by \((\CC,\BE,\fs)\); see \cref{example:n+2-angulated-category-is-n-exangulated}.  
Suppose \(\CC\) has all set-indexed coproducts. 
Following Neeman \cite[pp.\ 210--211]{Neeman-the-grothendieck-duality-theorem-via-bousfields-techniques-and-brown-representability}, an object \(X\inn\CC\) is called \emph{compact} if the functor \(\CC(X,-)\) commutes with all coproducts, and we denote by \(\CC^{\raisebox{0.5pt}{\scalebox{0.6}{\textup{c}}}}\) the 
subcategory of \(\CC\) consisting of 
compact objects. 
The category \(\CC\) is \emph{compactly generated} provided there is a set \(\CS\) of objects in  \(\CC^{\raisebox{0.5pt}{\scalebox{0.6}{\textup{c}}}}\) such that, for each \(A\inn\CC\), if \(\CC(X,A)=0\) for all \(X\inn\CS\), then \(A\) must 
be the zero object. 
Suppose that \(\CC\) is compactly generated.

As defined 
in 
\cite[Def.~1.1(3)]{Krause-Smashing-subcategories}, a distinguished triangle 
\begin{equation}
\label{eqn:delta-triangle-pure-example}
\begin{tikzcd}%[column sep=0.5cm]
A\arrow{r}& B\arrow{r}& C\arrow{r}{\delta}& \sus A
\end{tikzcd}
\end{equation}
in \(\CC\) is called \emph{pure-exact} if, 
for any object \(X\inn\CC^{\raisebox{0.5pt}{\scalebox{0.6}{\textup{c}}}}\), 
there is an induced short exact sequence 
\(
\begin{tikzcd}[column sep=0.5cm]
0 \arrow{r} & \CC(X,A)\arrow{r} & \CC(X,B)\arrow{r} & \CC(X,C) \arrow{r} & 0
\end{tikzcd}
\)
of abelian groups. 
A morphism \(\delta\colon C\to \sus A\) in \(\CC\) is \emph{phantom} if 
\(\CC(X,\delta)\) 
is the zero morphism \(\CC(X,C) \to \CC(X,\sus A)\) for all \(X\inn\CC^{\raisebox{0.5pt}{\scalebox{0.6}{\textup{c}}}}\) 
(see \cite[p.~104]{Krause-Smashing-subcategories}). 
For a morphism 
\(\delta\colon C \to \sus A\) that fits into a distinguished triangle \eqref{eqn:delta-triangle-pure-example}, 
there is a natural transformation
\(
\tensor[^{\BE}]{\delta}{_{-}} = \CC(-,\delta)
\colon \CC(-,C)\Rightarrow  \CC(-,\sus A)
\)
(see \cref{sec:n-exangulated-categories}).
Consequently, the morphism \(\delta\) is phantom if and only if 
\(\tensor[^{\BE}]{\delta}{_{X}} = 0\) 
for all \(X\inn\CC^{\raisebox{0.5pt}{\scalebox{0.6}{\textup{c}}}}\).
By Krause 
\cite[Lem.~1.3]{Krause-localization-theory-for-triangulated-categories}, one has that the distinguished triangle \eqref{eqn:delta-triangle-pure-example}
is pure-exact 
if and only if \(\delta\colon C\to \sus A\) is phantom.
This implies that 
\((\CC,\tensor[]{\BE}{_{\CC^{\textup{c}}}}, \tensor[]{\fs}{_{\CC^{\textup{c}}}})\) is an extriangulated category in which 
\(
\lan \begin{tikzcd}[column sep=0.5cm]
A\arrow{r}{f}& B\arrow{r}{g}& C
\end{tikzcd}, \delta\ran
\) 
is a distinguished extriangle if and only if 
\(\begin{tikzcd}[column sep=0.5cm]
A\arrow{r}{f}& B\arrow{r}{g}& C\arrow{r}{\delta}& \sus A
\end{tikzcd}\) 
is a pure-exact triangle in \(\CC\). 
By our discussion above this example, we have that \((\CC,\tensor[]{\BE}{_{\CC^{\textup{c}}}}, \tensor[]{\fs}{_{\CC^{\textup{c}}}})\) is an \emph{extriangulated} (i.e.\ 1-exangulated) subcategory of the (ex)triangulated category 
\((\CC,\BE,\fs)\).
Furthermore, if \(\CC\) contains a non-zero object, then \((\CC,\tensor[]{\BE}{_{\CC^{\textup{c}}}}, \tensor[]{\fs}{_{\CC^{\textup{c}}}})\) is not triangulated;
and if there is a non-split pure-exact distinguished triangle, then it is not exact (cf.\ \cite[Rem.~3.3]{HuZhangZhou-proper-classes-and-gorensteinness-in-extriangulated-categories}). 
\end{example}

Building on \cref{example:extriangulated-functor-purity-CGTcase}, the next example shows that a certain restricted Yoneda functor 
\(\SY \colon \CC \to \rMod{\CC^{\textup{c}}}\) 
preserves extriangles. 
Like in \cref{example:frobenius-n-exact-to-n-angulated-quotient}, the characterisation of \(n\)-exangulated functors 
of 
in \cref{thm:characterisation-of-n-exangulated-functors} makes light work of this. 
We also show that \(\SY\) both preserves and reflects injective 
objects. 
Relevant definitions are provided as needed.

\begin{example}
\label{example:rest-yoneda-extrian-func}
Suppose that \(\CC\) is a triangulated category with suspension functor \(\sus\). Assume that 
\(\CC\) has all set-indexed coproducts and  
is compactly generated. 
Denote by \((\CC,\BE,\fs)\) the extriangulated category arising from the triangulated structure on \(\CC\). 
In \cref{example:extriangulated-functor-purity-CGTcase} we showed that the relative structure induced by the 
subcategory \(\CC^{\raisebox{0.5pt}{\scalebox{0.6}{\textup{c}}}}\) of compact objects in \(\CC\) 
gives an extriangulated subcategory 
\((\CC,\tensor[]{\BE}{_{\CC^{\textup{c}}}},\tensor[]{\fs}{_{\CC^{\textup{c}}}})\).
Recall that the distinguished extriangles of \((\CC,\tensor[]{\BE}{_{\CC^{\textup{c}}}},\tensor[]{\fs}{_{\CC^{\textup{c}}}})\) 
correspond to the pure-exact triangles in \(\CC\).

Assume that \(\CC^{\raisebox{0.5pt}{\scalebox{0.6}{\textup{c}}}}\) is skeletally small. 
Let \(\rMod\CC^{\raisebox{0.5pt}{\scalebox{0.6}{\textup{c}}}}\) be the category of additive functors \((\CC^{\raisebox{0.5pt}{\scalebox{0.6}{\textup{c}}}})^{\op}\to \Ab\), which 
is a Grothendieck abelian category; 
see \cite[Thm.~10.1.3]{Prest-purity-spectra-localisation}. 
In particular, it has enough injectives (see \cite[Thm.~E.1.8]{Prest-purity-spectra-localisation}), so 
for all \(L,N\inn\rMod{\CC^{\textup{c}}}\), the collection of equivalence classes of short exact sequences of the form 
\(
\begin{tikzcd}[column sep=0.5cm]
0 \arrow{r}
    & L \arrow{r}
    & - \arrow{r}
    & N \arrow{r}
&0
\end{tikzcd}
\)
is a set. 
Thus, we have that the \(\Ext\)-bifunctor 
\(
\BE' 
    \deff 
    \Ext_{\rMod\CC^{\textup{c}}}^{1}        \colon
    (\rMod\CC^{\raisebox{0.5pt}{\scalebox{0.6}{\textup{c}}}})^{\op}\times \rMod\CC^{\raisebox{0.5pt}{\scalebox{0.6}{\textup{c}}}}\to \Ab
\)
is well-defined. 
Consequently, 
equipping \(\rMod\CC^{\raisebox{0.5pt}{\scalebox{0.6}{\textup{c}}}}\) with \(\BE'\) and the canonical realisation \(\fs'\) 
yields an extriangulated category 
\((\rMod\CC^{\raisebox{0.5pt}{\scalebox{0.6}{\textup{c}}}}, \BE', \fs')\).

Write \(\SY\colon\CC\to  \rMod\CC^{\raisebox{0.5pt}{\scalebox{0.6}{\textup{c}}}}\) for the \emph{restricted Yoneda functor}, which is defined on objects by 
\(\SY(Z) = \restr{\CC(-,Z)}{\CC^{\textup{c}}}\) 
(see \cite[p.\ 105]{Krause-Smashing-subcategories}). 
In the rest of this example, we show the following two statements. 
\begin{enumerate}[label=\textup{(\roman*)}]
    \item There is an extriangulated functor 
        \(
            (\SY,\Gamma)\colon 
            (\CC,\tensor[]{\BE}{_{\CC^{\textup{c}}}},\tensor[]{\fs}{_{\CC^{\textup{c}}}}) 
                \to  
            (\rMod\CC^{\raisebox{0.5pt}{\scalebox{0.6}{\textup{c}}}}, \BE', \fs')
        \).
    
    \item  The functor \(\SY\) both preserves and reflects injective objects.
\end{enumerate}
\medskip

Let us first prove (i). 
By \cref{thm:characterisation-of-n-exangulated-functors}, it is 
sufficient to define an additive functor 
\begin{center}
\(
\SE\colon\tensor[]{\BE}{_{\CC^{\textup{c}}}}\dExt{\CC}\to\BE'\dExt{\rMod\CC^{\raisebox{0.5pt}{\scalebox{0.6}{\textup{c}}}}}
\)
\end{center} 
which respects both morphisms and distinguished extriangles over \(\SY\). 
Given 
\(
\delta\inn\tensor[]{\BE}{_{\CC^{\textup{c}}}}(C,A)
\), there is a pure-exact distinguished triangle 
\(
\begin{tikzcd}[column sep=0.5cm]
A\arrow{r}& B\arrow{r}& C\arrow{r}{\delta}& \sus A
\end{tikzcd}
\)
that is unique up to isomorphism. 
Thus, we define \(\SE\) on objects by setting 
\[
\SE(\delta) 
    \deff [
    \begin{tikzcd}%[column sep=0.5cm]
    0 \arrow{r}& \SY(A)\arrow{r}& \SY(B)\arrow{r}& \SY(C)\arrow{r}& 0
    \end{tikzcd}
    ].
\]
A morphism \((a,c)\colon \delta \to \delta'\) in \(\tensor[]{\BE}{_{\CC^{\textup{c}}}}\dExt{\CC}\) extends to a morphism of triangles between the pure-exact triangles associated to \(\delta\) and \(\delta'\). 
Hence, the pair 
\((\SY a, \SY c)\colon\SE(\delta)\to\SE(\delta')\) is a morphism 
in \(\BE'\dExt{\rMod\CC^{\raisebox{0.5pt}{\scalebox{0.6}{\textup{c}}}}}\), 
and we can define \(\SE(a,c)\deff(\SY a, \SY c)\). 
It is straightforward to check that \(\SE\) is a functor, and it respects morphisms and distinguished extriangles over \(\SY\) by construction. 
Since \(\SY\) is additive, so is \(\SE\) 
by \cref{prop:F-additive-iff-E-additive-iff-E-exact}.  
This verifies (i). 
Note that if \(\CC\) contains a non-zero object and at least one non-split 
pure exact triangle, 
then \((\SY,\tensor[]{\Gamma}{_{(\SY,\SE)}})\) is neither an exact functor nor a triangulated functor.

Now we show (ii).
A morphism \(f\colon A\to B\) in the triangulated category \(\CC\) is a \emph{pure monomorphism} if the morphism 
\(\SY(f\tensor[]{)}{_{X}} = \CC(X,f)\colon \CC(X,A)\to\CC(X,B)\) 
is injective for each compact object \(X\inn\CC^{\raisebox{0.5pt}{\scalebox{0.6}{\textup{c}}}}\) 
(see \cite[Def.~1.1(1)]{Krause-Smashing-subcategories}). 
An object \(C\inn\CC\) is \emph{pure-injective} provided any pure monomorphism 
with domain \(C\) splits (see \cite[Def.~1.1(2)]{Krause-Smashing-subcategories}). 
Following the dual of \cite[Def.~3.23, Prop.~3.24]{NakaokaPalu-extriangulated-categories-hovey-twin-cotorsion-pairs-and-model-structures}, 
an object \(I\) in an extriangulated category 
\((\CD,\BF,\ft)\) 
is \emph{\(\BF\)-injective} 
if and only if \(\BF(D,I)=0\) for all \(D\inn\CD\). 
Thus, it follows from \cite[Lem.~1.4]{Krause-Smashing-subcategories} that 
an object in the extriangulated category  
\((\CC,\tensor[]{\BE}{_{\CC^{\textup{c}}}}, \tensor[]{\fs}{_{\CC^{\textup{c}}}})\) 
is \(\tensor[]{\BE}{_{\CC^{\textup{c}}}}\)-injective if and only if 
it is pure-injective in the triangulated category \((\CC,\BE,\fs)\). 
Furthermore, 
by \cite[Thm.~1.8]{Krause-Smashing-subcategories}, 
we have that \(A\inn\CC\) is pure-injective in \((\CC,\BE,\fs)\) 
if and only if 
\(\SY A\) is \(\BE'\)-injective in \((\rMod\CC^{\raisebox{0.5pt}{\scalebox{0.6}{\textup{c}}}}, \BE', \fs')\). 
That is, the extriangulated functor 
\(
(\SY,\Gamma)
\) 
from 
\(
(\CC,\tensor[]{\BE}{_{\CC^{\textup{c}}}},\tensor[]{\fs}{_{\CC^{\textup{c}}}}) 
\)  
to 
\(
    (\rMod\CC^{\raisebox{0.5pt}{\scalebox{0.6}{\textup{c}}}}, \BE', \fs')
\)
preserves and reflects injective objects. 
\end{example}

\begin{rem}
Suppose that in \cref{example:extriangulated-functor-purity-CGTcase,example:rest-yoneda-extrian-func} we replace the compactly generated triangulated category with a \emph{finitely accessible} 
category, 
and 
also that we 
swap compact objects with \emph{finitely presented} objects. 
With this exchange, 
one can make analogous observations about the restricted Yoneda functor to those made in \cref{example:rest-yoneda-extrian-func} using results of Crawley-Boevey \cite{Crawley-Boevey-locally-finitely-presented-additive-categories}.  
We omit this example, however, 
since 
the restricted Yoneda functor in this case is in fact an exact functor.
\end{rem}

%%%%%%%%%%%%%%%%%%%%%%%%%%%%%%%%%%%%%%%%%%%%%%%%
%%%%%%%%%%%%%%%%%%%%%%%%%%%%%%%%%%%%%%%%%%%%%%%%

{\setstretch{1}\begin{acknowledgements}
The authors are very grateful to, and thank: 
Thomas Br{\"{u}}stle for their helpful comments that led to 
\cref{prop:ECC-is-a-category} in its present form;
Hiroyuki Nakaoka for their time and a discussion that led to the development of \cref{section:4}; and 
Peter J{\o{}}rgensen 
for a useful discussion relating to 
\cref{example:delta-functors}.

The first author is grateful to have been supported during part of this work by the Alexander von Humboldt Foundation in the framework of an Alexander von Humboldt Professorship endowed by the German Federal Ministry of Education.
Parts of this work were carried out while the second author participated in the Junior Trimester Program ``New Trends in Representation Theory'' at the Hausdorff Research Institute for Mathematics in Bonn. She would like to thank the Institute for excellent working conditions.
The third author is grateful to have been supported by Norwegian Research Council project 301375, ``Applications of reduction techniques and computations in representation theory''.
The fourth author gratefully acknowledges support from: the Danish National Research Foundation (grant DNRF156); the Independent Research Fund Denmark (grant 1026-00050B); the Aarhus University Research Foundation (grant AUFF-F-2020-7-16); the Engineering and Physical Sciences Research Council (grant EP/P016014/1); and the London Mathematical Society with support from Heilbronn Institute for Mathematical Research (grant ECF-1920-57). 
\end{acknowledgements}}
%%%%%%%%%%%%%%%%%%%%%%%%%%%%%%%%%%%%%%%%%%%%%%%%
%%%%%%%%%%%%%%%%%%%%%%%%%%%%%%%%%%%%%%%%%%%%%%%%
{\setstretch{1}

}

\begin{thebibliography}{10}

\bibitem{AmiotIyamaReiten-stable-categories-of-Cohen-Macaulay-modules-and-cluster-categories}
C.~Amiot, O.~Iyama, and I.~Reiten.
\newblock {\em Stable categories of {C}ohen-{M}acaulay modules and cluster
  categories}.
\newblock Amer. J. Math., {\bf 137}(3):813--857, 2015.

\bibitem{BeilinsonBernsteinDeligne-perverse-sheaves}
A.~A. Be{\u{\i}}linson, J.~Bernstein, and P.~Deligne.
\newblock {\em Faisceaux pervers}.
\newblock In {\em Analysis and topology on singular spaces, {I} ({L}uminy,
  1981)}, volume 100 of {\em Ast{\'{e}}risque}, pages 5--171. Soc. Math.
  France, Paris, 1982.

\bibitem{Bennett-Tennenhaus-some-characterisations}
R.~Bennett-Tennenhaus.
\newblock {\em Characterisations of {$\Sigma$}-pure-injectivity in triangulated
  categories and applications to endocoperfect objects}.
\newblock Fund. Math., {\bf 261}:133--155, 2023.

\bibitem{Bennett-TennenhausHauglandSandoyShah-structure-preserving-functors-between-higher-exangulated-categories}
R.~Bennett-Tennenhaus, J.~Haugland, M.~H. Sand{\o}y, and A.~Shah.
\newblock {\em Structure-preserving functors in higher homological algebra}.
\newblock In preparation.

\bibitem{Bennett-TennenhausHauglandSandoyShah-the-category-of-extensions-and-the-Krull-Remak-Schmidt-property}
R.~Bennett-Tennenhaus, J.~Haugland, M.~H. Sand{\o}y, and A.~Shah.
\newblock {\em The category of extensions and idempotent completion}.
\newblock Preprint, 2023.
\newblock \url{https://arxiv.org/abs/2303.07306}.

\bibitem{Bennett-TennenhausShah-transport-of-structure-in-higher-homological-algebra}
R.~Bennett-Tennenhaus and A.~Shah.
\newblock {\em Transport of structure in higher homological algebra}.
\newblock J. Algebra, {\bf 574}:514--549, 2021.

\bibitem{BerghThaule-the-grothendieck-group-of-an-n-angulated-category}
P.~A. Bergh and M.~Thaule.
\newblock {\em The {G}rothendieck group of an {\(n\)}-angulated category}.
\newblock J. Pure Appl. Algebra, {\bf 218}(2):354--366, 2014.

\bibitem{Borve-Trygsland-factorization-extensions}
E.~D. {B{\o{}}rve} and P.~Trygsland.
\newblock {\em A theorem of {R}etakh for exact $\infty$-categories and higher
  extension functors}.
\newblock Preprint, 2021.
\newblock \url{https://arxiv.org/abs/2110.05138}.

\bibitem{Brustle-Hille-Matrices-over-upper-triangular-bimodules}
T.~Br{\"u}stle and L.~Hille.
\newblock {\em Matrices over upper triangular bimodules and $\delta$-filtered
  modules over quasi-hereditary algebras}.
\newblock Colloq. Math., {\bf 83}:295--303, 2000.

\bibitem{Buhler-exact-categories}
T.~B{\"{u}}hler.
\newblock {\em Exact categories}.
\newblock Expo. Math., {\bf 28}(1):1--69, 2010.

\bibitem{Crawley-Boevey-Matrix-problems-and-Drozds-theorem}
W.~Crawley-Boevey.
\newblock {\em Matrix problems and {D}rozd's theorem}.
\newblock Banach Center Publ., {\bf 26}:199--222, 1990.

\bibitem{Crawley-Boevey-locally-finitely-presented-additive-categories}
W.~Crawley-Boevey.
\newblock {\em Locally finitely presented additive categories}.
\newblock Comm. Algebra, {\bf 22}(5):1641--1674, 1994.

\bibitem{DraxlerReitenSmaloSolberg-exact-categories-and-vector-space-categories}
P.~Dr{\"{a}}xler, I.~Reiten, S.~O. Smal{\o}, and {\O}.~Solberg.
\newblock {\em Exact categories and vector space categories}.
\newblock Trans. Amer. Math. Soc., {\bf 351}(2):647--682, 1999.
\newblock With an appendix by B. Keller.

\bibitem{Dyckerhoff-Jasso-Lekili-the-symplectic-geometry-of-higher-auslander-algebras-symmetric-products-of-disks}
T.~Dyckerhoff, G.~Jasso, and Y.~Lekili.
\newblock {\em The symplectic geometry of higher {A}uslander algebras:
  symmetric products of disks}.
\newblock Forum Math. Sigma, {\bf 9}:e10, 49, 2021.

\bibitem{enomoto2022grothendieck}
H.~Enomoto and S.~Saito.
\newblock {\em Grothendieck monoids of extriangulated categories}.
\newblock Preprint, 2022.
\newblock \url{https://arxiv.org/abs/2208.02928}.

\bibitem{EvansPugh-the-Nakayama-automorphism-of-the-almost-Calabi-Yau-algebras-associated-to-SU3-modular-invariants}
D.~E. Evans and M.~Pugh.
\newblock {\em The {N}akayama automorphism of the almost {C}alabi-{Y}au
  algebras associated to {\(su(3)\)} modular invariants}.
\newblock Comm. Math. Phys., {\bf 312}(1):179--222, 2012.

\bibitem{GabrielNazarovaRoiterSergeichukVossieck-tame-and-wild-subspace-problems}
P.~Gabriel, L.~A. Nazarova, A.~V. Roiter, V.~V. Sergeichuk, and D.~Vossieck.
\newblock {\em Tame and wild subspace problems}.
\newblock Ukrainian Mathematical Journal, {\bf 45}:335--372, 1993.

\bibitem{GabrielRoiter-reps-of-finite-dimensional-algebras}
P.~Gabriel and A.~V. Roiter.
\newblock {\em Representations of finite-dimensional algebras}.
\newblock Springer-Verlag, Berlin, 1997.
\newblock Translated from the Russian, With a chapter by B. Keller, Reprint of
  the 1992 English translation.

\bibitem{Garkusha-Prest-Ziegler}
G.~Garkusha and M.~Prest.
\newblock {\em Triangulated categories and the {Z}iegler spectrum}.
\newblock Algebr. Represent. Theory, {\bf 8}(4):499--523, 2005.

\bibitem{geiss-Deformations-of-bimodule-problems}
C.~Gei{\ss}.
\newblock {\em Deformations of bimodule problems}.
\newblock Fund. Math., {\bf 150}(3):255--264, 1996.

\bibitem{GeissKellerOppermann-n-angulated-categories}
C.~Geiss, B.~Keller, and S.~Oppermann.
\newblock {\em {\(n\)}-angulated categories}.
\newblock J. Reine Angew. Math., {\bf 675}:101--120, 2013.

\bibitem{Happel-triangulated-cats-in-rep-theory}
D.~Happel.
\newblock {\em Triangulated categories in the representation theory of
  finite-dimensional algebras}, volume 119 of {\em London Math. Soc. Lecture
  Note Ser.}
\newblock Cambridge Univ. Press, Cambridge, 1988.

\bibitem{HassounShah-integral-and-quasi-abelian-hearts-of-twin-cotorsion-pairs-on-extriangulated-categories}
S.~Hassoun and A.~Shah.
\newblock {\em Integral and quasi-abelian hearts of twin cotorsion pairs on
  extriangulated categories}.
\newblock Comm. Algebra, {\bf 48}(12):5142--5162, 2020.

\bibitem{Haugland-the-grothendieck-group-of-an-n-exangulated-category}
J.~Haugland.
\newblock {\em The {G}rothendieck {G}roup of an {\(n\)}-exangulated
  {C}ategory}.
\newblock Appl. Categ. Structures, {\bf 29}(3):431--446, 2021.

\bibitem{HauglandSandoy-higher-Koszul-duality-and-connections-with-n-hereditary-algebras}
J.~Haugland and M.~H. Sand{\o}y.
\newblock {\em Higher {K}oszul duality and connections with \(n\)-hereditary
  algebras}.
\newblock Preprint, 2021.
\newblock \url{https://arxiv.org/abs/2101.12743}.

\bibitem{HeHeZhou-localization-of-n-exangulated-categories}
J.~He, J.~He, and P.~Zhou.
\newblock {\em Localization of $n$-exangulated categories}.
\newblock Preprint, 2022.
\newblock \url{https://arxiv.org/abs/2205.07644}.

\bibitem{HeZhou-n-exact-categories-arising-from-n-exangulated-categories}
J.~He and P.~Zhou.
\newblock {\em \(n\)-exact categories arising from \(n\)-exangulated
  categories}.
\newblock Preprint, 2021.
\newblock \url{https://arxiv.org/abs/2109.12954}.

\bibitem{HerschendIyamaMinamotoOppermann-representation-theory-ofGeigle-Lenzing-complete-intersections}
M.~Herschend, O.~Iyama, H.~Minamoto, and S.~Oppermann.
\newblock {\em Representation theory of {G}eigle-{L}enzing complete
  intersections}.
\newblock Mem. Amer. Math. Soc., {\bf 285}(1412):vii+141, 2023.

\bibitem{HerschendLiuNakaoka-n-exangulated-categories-I-definitions-and-fundamental-properties}
M.~Herschend, Y.~Liu, and H.~Nakaoka.
\newblock {\em n-exangulated categories ({I}): Definitions and fundamental
  properties}.
\newblock J. Algebra, {\bf 570}:531--586, 2021.

\bibitem{HerschendLiuNakaoka-n-exangulated-categories-II}
M.~Herschend, Y.~Liu, and H.~Nakaoka.
\newblock {\em {\(n\)}-exangulated categories ({II}): {C}onstructions from
  {\(n\)}-cluster tilting subcategories}.
\newblock J. Algebra, {\bf 594}:636--684, 2022.

\bibitem{HuZhangZhou-proper-classes-and-gorensteinness-in-extriangulated-categories}
J.~Hu, D.~Zhang, and P.~Zhou.
\newblock {\em Proper classes and {G}orensteinness in extriangulated
  categories}.
\newblock J. Algebra, {\bf 551}:23--60, 2020.

\bibitem{Iyama-higher-dimensional-auslander-reiten-theory-on-maximal-orthogonal-subcategories}
O.~Iyama.
\newblock {\em Higher-dimensional {A}uslander-{R}eiten theory on maximal
  orthogonal subcategories}.
\newblock Adv. Math., {\bf 210}(1):22--50, 2007.

\bibitem{Iyama-cluster-tilting-for-higher-auslander-algebras}
O.~Iyama.
\newblock {\em Cluster tilting for higher {A}uslander algebras}.
\newblock Adv. Math., {\bf 226}(1):1--61, 2011.

\bibitem{IyamaWemyss-maximal-modifications-and-Auslander-Reiten-duality-for-non-isolated-singularities}
O.~Iyama and M.~Wemyss.
\newblock {\em Maximal modifications and {A}uslander-{R}eiten duality for
  non-isolated singularities}.
\newblock Invent. Math., {\bf 197}(3):521--586, 2014.

\bibitem{IyamaYoshino-mutation-in-tri-cats-rigid-CM-mods}
O.~Iyama and Y.~Yoshino.
\newblock {\em Mutation in triangulated categories and rigid {C}ohen-{M}acaulay
  modules}.
\newblock Invent. Math., {\bf 172}(1):117--168, 2008.

\bibitem{Jasso-n-abelian-and-n-exact-categories}
G.~Jasso.
\newblock {\em {\(n\)}-abelian and {\(n\)}-exact categories}.
\newblock Math. Z., {\bf 283}(3-4):703--759, 2016.

\bibitem{JassoKulshammer-the-naive-approach-for-constructing-the-derived-category-of-a-d-abelian-category-fails}
G.~Jasso and J.~K{\"{u}}lshammer.
\newblock {\em The naive approach for constructing the derived category of a
  \(d\)-abelian category fails}.
\newblock Note, 2016.
\newblock \url{https://arxiv.org/abs/1604.03473}.

\bibitem{JohnsonYau-2-dimensional-categories}
N.~Johnson and D.~Yau.
\newblock {\em 2-dimensional categories}.
\newblock Oxford Univ. Press, Oxford, 2021.

\bibitem{Jorgensen-tropical-friezes-and-the-index-in-higher-homological-algebra}
P.~J{\o}rgensen.
\newblock {\em Tropical friezes and the index in higher homological algebra}.
\newblock Math. Proc. Cambridge Philos. Soc., {\bf 171}(1):23--49, 2021.

\bibitem{Jorgensen-abelian-subcategories-of-triangulated-categories-induced-by-simple-minded-systems}
P.~J{\o}rgensen.
\newblock {\em Abelian subcategories of triangulated categories induced by
  simple minded systems}.
\newblock Math. Z., {\bf 301}(1):565--592, 2022.

\bibitem{JorgensenShah-grothendieck-groups-of-d-exangulated-categories-and-a-modified-CC-map}
P.~J{\o}rgensen and A.~Shah.
\newblock {\em Grothendieck groups of \(d\)-exangulated categories and a
  modified {C}aldero-{C}hapoton map}.
\newblock Preprint, 2021.
\newblock \url{https://arxiv.org/abs/2106.02142}.

\bibitem{JorgensenShah-the-index-with-respect-to-a-rigid-subcategory}
P.~J{\o}rgensen and A.~Shah.
\newblock {\em The index with respect to a rigid subcategory of a triangulated
  category}.
\newblock Int. Math. Res. Not. IMRN, 2022.
\newblock To appear.
\newblock \url{https://arxiv.org/abs/2201.00740}.

\bibitem{Kashiwara-Schapira-Categories-and-sheaves}
M.~Kashiwara and P.~Schapira.
\newblock {\em Categories and sheaves}, volume 332 of {\em Grundlehren Math.
  Wiss.}
\newblock Springer-Verlag, Berlin, 2006.

\bibitem{Keller-derived-categories-and-universal-problems}
B.~Keller.
\newblock {\em Derived categories and universal problems}.
\newblock Comm. Algebra, {\bf 19}(3):699--747, 1991.

\bibitem{Keller-derived-categories-and-their-uses}
B.~Keller.
\newblock {\em Derived categories and their uses}.
\newblock In {\em Handbook of algebra, {V}ol. 1}, volume~1 of {\em Handb.
  Algebr.}, pages 671--701. Elsevier/North-Holland, Amsterdam, 1996.

\bibitem{KellerVossieck-sous-les-categories-derivees}
B.~Keller and D.~Vossieck.
\newblock {\em Sous les cat{\'{e}}gories d{\'{e}}riv{\'{e}}es}.
\newblock C. R. Acad. Sci. Paris S{\'{e}}r. I Math., {\bf 305}(6):225--228,
  1987.

\bibitem{Klapproth-n-exact-categories-arising-from-nplus2-angulated-categories}
C.~Klapproth.
\newblock {\em $n$-exact categories arising from $(n+2)$-angulated categories}.
\newblock Preprint, 2021.
\newblock \url{https://arxiv.org/abs/2108.04596}.

\bibitem{Krause-Smashing-subcategories}
H.~Krause.
\newblock {\em Smashing subcategories and the telescope conjecture---an
  algebraic approach}.
\newblock Invent. Math., {\bf 139}(1):99--133, 2000.

\bibitem{Krause-localization-theory-for-triangulated-categories}
H.~Krause.
\newblock {\em Localization theory for triangulated categories}.
\newblock In {\em Triangulated categories}, volume 375 of {\em London Math.
  Soc. Lecture Note Ser.}, pages 161--235. Cambridge Univ. Press, Cambridge,
  2010.

\bibitem{Lack-a-2-categories-companion}
S.~Lack.
\newblock {\em A 2-categories companion}.
\newblock In {\em Towards higher categories}, volume 152 of {\em IMA Vol. Math.
  Appl.}, pages 105--191. Springer, New York, 2010.

\bibitem{Linckelmann-on-abelian-subcategories-of-triangulated-categories}
M.~Linckelmann.
\newblock {\em On abelian subcategories of triangulated categories}.
\newblock Preprint, 2020.
\newblock \url{https://arxiv.org/abs/2008.03199}.

\bibitem{LiuYNakaoka-hearts-of-twin-cotorsion-pairs-on-extriangulated-categories}
Y.~Liu and H.~Nakaoka.
\newblock {\em Hearts of twin cotorsion pairs on extriangulated categories}.
\newblock J. Algebra, {\bf 528}:96--149, 2019.

\bibitem{LiuZhou-frobenius-n-exangulated-categories}
Y.~Liu and P.~Zhou.
\newblock {\em Frobenius {\(n\)}-exangulated categories}.
\newblock J. Algebra, {\bf 559}:161--183, 2020.

\bibitem{MacLane-categories-for-the-working-mathematician}
S.~{Mac Lane}.
\newblock {\em Categories for the working mathematician}, volume~5 of {\em
  Grad. Texts in Math.}
\newblock Springer-Verlag, New York, second edition, 1998.

\bibitem{Msapato-the-karoubi-envelope-and-weak-idempotent-completion-of-an-extriangulated-category}
D.~Msapato.
\newblock {\em The {K}aroubi envelope and weak idempotent completion of an
  extriangulated category}.
\newblock Appl. Categ. Structures, {\bf 30}(3):499--535, 2022.

\bibitem{NakaokaOgawaSakai-localization-of-extriangulated-categories}
H.~Nakaoka, Y.~Ogawa, and A.~Sakai.
\newblock {\em Localization of extriangulated categories}.
\newblock J. Algebra, {\bf 611}:341--398, 2022.

\bibitem{NakaokaPalu-extriangulated-categories-hovey-twin-cotorsion-pairs-and-model-structures}
H.~Nakaoka and Y.~Palu.
\newblock {\em Extriangulated categories, {H}ovey twin cotorsion pairs and
  model structures}.
\newblock Cah. Topol. G{\'{e}}om. Diff{\'{e}}r. Cat{\'{e}}g., {\bf
  60}(2):117--193, 2019.

\bibitem{Neeman-the-grothendieck-duality-theorem-via-bousfields-techniques-and-brown-representability}
A.~Neeman.
\newblock {\em The {G}rothendieck duality theorem via {B}ousfield's techniques
  and {B}rown representability}.
\newblock J. Amer. Math. Soc., {\bf 9}(1):205--236, 1996.

\bibitem{OppermannThomas-higher-dimensional-cluster-combinatorics-and-representation-theory}
S.~Oppermann and H.~Thomas.
\newblock {\em Higher-dimensional cluster combinatorics and representation
  theory}.
\newblock J. Eur. Math. Soc. (JEMS), {\bf 14}(6):1679--1737, 2012.

\bibitem{PadrolPaluPilaudPlamondon-associahedra-for-finite-type-cluster-algebras-and-minimal-relations-between-g-vectors}
A.~Padrol, Y.~Palu, V.~Pilaud, and P.-G. Plamondon.
\newblock {\em Associahedra for finite type cluster algebras and minimal
  relations between \(\mathbf{g}\)-vectors}.
\newblock Preprint, 2019.
\newblock \url{https://arxiv.org/abs/1906.06861}.

\bibitem{Prest-purity-spectra-localisation}
M.~Prest.
\newblock {\em Purity, spectra and localisation}, volume 121 of {\em
  Encyclopedia of Mathematics and its Applications}.
\newblock Cambridge Univ. Press, Cambridge, 2009.

\bibitem{Prest-Rajani-Structure-sheaves}
M.~Prest and R.~Rajani.
\newblock {\em Structure sheaves of definable additive categories}.
\newblock J. Pure Appl. Algebra, {\bf 214}(8):1370--1383, 2010.

\bibitem{Shah-AR-theory-quasi-abelian-cats-KS-cats}
A.~Shah.
\newblock {\em Auslander-{R}eiten theory in quasi-abelian and {K}rull-{S}chmidt
  categories}.
\newblock J. Pure Appl. Algebra, {\bf 224}(1):98--124, 2020.

\bibitem{Liu-Tan-Resolution-Dimension-Relative-to-Resolving-Subcategories-in-Extriangulated-Categories}
L.~Tan and L.~Liu.
\newblock {\em Resolution dimension relative to resolving subcategories in
  extriangulated categories}.
\newblock Mathematics, {\bf 9}(9):980, 129, 2021.

\bibitem{Tiefenbrunner-phd-thesis}
R.~Tiefenbrunner.
\newblock {\em Darstellungen von Bimodulproblemen}.
\newblock PhD thesis, FU Berlin, 1995.

\bibitem{Williams-new-interpretations-of-the-higher-stasheff-tamari-orders}
N.~J. Williams.
\newblock {\em New interpretations of the higher {S}tasheff-{T}amari orders}.
\newblock Adv. Math., {\bf 407}:Paper No. 108552, 49, 2022.

\bibitem{ZhengWei-nplus2-angulated-quotient-categories}
Q.~Zheng and J.~Wei.
\newblock {\em {\((n+2)\)}-angulated quotient categories}.
\newblock Algebra Colloq., {\bf 26}(4):689--720, 2019.

\bibitem{Zhou-n-extension-closed-subcategories-of-nplus2-angulated-categories}
P.~Zhou.
\newblock {\em {$n$}-extension closed subcategories of {$(n+2)$}-angulated
  categories}.
\newblock Arch. Math. (Basel), {\bf 118}(4):375--382, 2022.

\end{thebibliography}
\end{document}